\documentclass[12pt,oneside]{book}
\usepackage{amsmath}
\usepackage{amsfonts}
\usepackage{amssymb}
\usepackage{amsthm}
\usepackage{mathrsfs}
\usepackage{enumerate}
\usepackage[all]{xy}
\usepackage{pstricks}
\usepackage{mathtools}
\usepackage{tikz}
\usetikzlibrary{arrows}
\usetikzlibrary{calc}
\usepackage{arydshln,leftidx}
\usepackage[makeroom]{cancel}
\usepackage{arcs}
\usepackage{etex}
\usepackage{comment}

\usepackage{color} 
\usepackage{hyperref}

\newcommand{\CC}{\mathbb{C}} 
\newcommand{\FF}{\mathbb{F}} 
\newcommand{\NN}{\mathbb{N}} 
\newcommand{\QQ}{\mathbb{Q}} 
\newcommand{\RR}{\mathbb{R}} 
\newcommand{\ZZ}{\mathbb{Z}} 
\newcommand{\TT}{\mathbb{T}} 

\newcommand{\fm}{\mathfrak{m}} 

\newcommand{\Hom}{\textrm{Hom}} 
\newcommand{\End}{\textrm{End}} 
\newcommand{\Aut}{\textrm{Aut}} %
\newcommand{\im}{\textrm{Im}} 

\newcommand{\Auts}{\boldsymbol{Aut}} 
\newcommand{\Ends}{\boldsymbol{End}} 

\newcommand{\id}{\textrm{id}} 
\newcommand{\lie}{\textrm{Lie}\,} 

\newcommand{\ul}{\underline} 

\newcommand{\spec}{\textrm{Spec}\,} 

\newcommand{\ilim}{\mathop{\varinjlim}\limits} 
\newcommand{\plim}{\mathop{\varprojlim}\limits} 

\theoremstyle{plain}
\newtheorem{thm}{Theorem}[chapter]
\newtheorem{prop}[thm]{Proposition}
\newtheorem{lemma}[thm]{Lemma}
\newtheorem{cor}[thm]{Corollary}

\theoremstyle{definition}
\newtheorem{defn}[thm]{Definition}

\newtheorem{rmk}[thm]{Remark}
\newtheorem{exam}[thm]{Example}

\theoremstyle{definition}
\newtheorem*{definition}{Definition}
\newtheorem*{exam*}{Example}
\newtheorem*{rmk*}{Remark}
\newtheorem*{fact*}{Fact}

\makeatletter
\renewcommand*\@makechapterhead[1]{%
  \vspace*{0.75in}%
  {\parindent \z@ \raggedright \normalfont
    \ifnum \c@secnumdepth >\m@ne
        \Large\bfseries \@chapapp\space \thechapter
        \par\nobreak
        \vskip 10\p@
    \fi
    \interlinepenalty\@M
    \Large \bfseries #1\par\nobreak
    \vskip 20\p@
  }}
\renewcommand*\@makeschapterhead[1]{%
  \vspace*{0.75in}%
  {\parindent \z@ \raggedright
    \normalfont
    \interlinepenalty\@M
    \Large \bfseries  #1\par\nobreak
    \vskip 20\p@
  }}
\makeatother

\begin{document}

\title{Canonical Barsotti--Tate Groups of Finite Level}

\author{Zeyu (Ding) Ding \\ SUNY Binghamton  \\ Binghamton, NY 13902 \\ ding@math.binghamton.edu}

\date{}

\maketitle

\chapter*{Abstract}

Let $k$ be an algebraically closed field of characteristic $p>0$. Let $c,d\in \NN$ be such that $h=c+d>0$. Let $H$ be a $p$-divisible group of codimension $c$ and dimension $d$ over $k$. For $m\in\NN^\ast$ let $H[p^m]=\ker([p^m]:H\rightarrow H)$. It is a finite commutative group scheme over $k$ of $p$ power order, called a Barsotti-Tate group of level $m$. 

We study a particular type of $p$-divisible groups $H_\pi$, where $\pi$ is a permutation on the set $\{1,2,\dots,h\}$. Let $(M,\varphi_\pi)$ be the Dieudonn\'e module of $H_\pi$. Each $H_\pi$ is uniquely determined by $H_\pi[p]$ and by the fact that there exists a maximal torus $T$ of $GL_M$ whose Lie algebra is normalized by $\varphi_\pi$ in a natural way. Moreover, if $H$ is a $p$-divisible group of codimension $c$ and dimension $d$ over $k$, then $H[p]\cong H_\pi[p]$ for some permutation $\pi$. We call these $H_\pi$ canonical lifts of Barsotti--Tate groups of level $1$. We obtain new formulas of combinatorial nature for the dimension of $\Auts(H_\pi[p^m])$ and for the number of connected components of $\Ends(H_\pi[p^m])$.


\tableofcontents


\chapter*{Introduction}
\addcontentsline{toc}{chapter}{Introduction}

Let $k$ be an algebraically closed field of characteristic $p>0$. Let $c,d\in \NN$ be such that $h:=c+d>0$. Recall that 

\begin{defn}
A {\it $p$-divisible group} over $k$ of {\it height} $h$ is an inductive system $H=(G_n,i_n)_{n\geq0}$, where $G_n$ is a finite group scheme over $k$ of order $p^{nh}$, such that the sequences $0 \rightarrow G_n \xrightarrow{i_n} G_{n+1} \xrightarrow{p^n} G_{n+1}$ are exact for all $n\geq0$. 

Each $G_m=H[p^m]:=\ker([p^m]:H\rightarrow H)$ is called a {\em Barsotti--Tate group of level $m$}. 

\end{defn}

An important tool for studying $p$-divisible groups is the Dieudonn\'{e} theory. Let $W(k)$ be the ring of Witt vectors over $k$. Let $\sigma$ be the Frobenius of $W(k)$. A {\it Dieudonn\'{e} module} over $k$ is a pair $(M,\varphi)$ where M is a free $W(k)$-module of finite rank and $\varphi$ is a $\sigma$-linear endomorphism such that $pM\subseteq \varphi(M)\subseteq M$. 
It is well known that the category of $p$-divisible groups over $k$ is anti-equivalent to the category of Dieudonn\'e modules over $k$. Let $(M,\varphi)$ be the Dieudonn\'e module of $H$. The  \emph{dimension} of $H$ is $\dim_kM/\varphi(M)$, the \emph{codimension} of $H$ is $\dim_k\varphi(M)/pM$.

Let $H$ be a $p$-divisible group over $k$ of codimension $c$ and dimension $d$. Its height is $h$. It is well-known that $H$ is determined by some finite truncation $H[p^m]$ of sufficiently large level $m$. This allows us to associate to $H$ two numerical invariants: the {\em isomorphism number} $n_H$ is the least level $m$ such that $H[p^m]$ determines $H$ up to isomorphism, and the {\em isogeny cutoff} $b_H$ is the least level $m$ such that $H[p^m]$ determines $H$ up to isogeny. More preciously, we have

\begin{defn}
	Let $H$ be a $p$-divisible group  over $k$ of dimension $d$ and codimension $c$. The \emph{isomorphism number $n_H$} (resp. the \emph{isogeny cutoff $b_H$}) is the smallest nonnegative integer $m$ with the following property: If $H'$ is a $p$-divisible group over $k$ of dimension $d$ and codimension $c$ such that $H'[p^m]$ is isomorphic to $H[p^m]$, then $H'$ is isomorphic (resp. isogenous) to $H$.
\end{defn}
 For the existence of $n_H$, we refer to \cite[Chapter III, Section 3]{Manin:formalgroups}, \cite[Theorem 3]{Traverso:pisa}, \cite[Theorem 1]{Traverso:pdivisible}, \cite[Corollary 1.7]{Oort:foliation} or \cite[Corollary 1.3]{CBP}. Traverso's truncation conjecture \cite{Traverso:specializations} predicts that $n_H \leq \textrm{min}\,\{c,d\}$. This prediction turned out to be true only for certain values of $c$ and $d$ and other special cases. Such cases were first proved in \cite[Corollary 3.2]{supersingular} and \cite[Theorem 1.5.2]{Vasiu:reconstr}. Recently, E. Lau, M. Nicole and A. Vasiu have found that $n_H \leq \lfloor \frac{2cd}{c+d} \rfloor$ and shown that this is optimal if the Dieudonn\'e module corresponding to $H$ is isoclinic (means that Newton slopes of the Dieudonn\'e module are all equal); see \cite[Theorem 1.4, Proposition 9.16]{traversosolved}. They even proved a stronger statement that $n_H \leq \lfloor 2\nu(c) \rfloor$ where $\nu$ is the Newton polygon of $H$. Their results disprove Traverso's conjecture as $\lfloor \frac{2cd}{c+d} \rfloor$ is in general greater than $\textrm{min}\,\{c,d\}$. 

Thus to classify $p$-divisible groups $H$ over $k$, one is led to consider four basic questions.

	(1) Determine the set $N_{c,d}$ of possible values of $n_H$.
		
	(2) For $1\leq m\leq N_{c,d}$ classify all Barsotti--Tate groups of level $m$ over $k$.
	
	(3) Find invariants that go up under specializations.
	
	(4) Find principles that govern group actions whose orbits parametrize the isomorphism classes of Barsotti--Tate groups of level $m$ over $k$.

Let $\Auts(H[p^m])$ be the automorphism group scheme of $H[p^m]$ over $k$ and let $\gamma_H(m):=\dim(\Auts(H[p^m]))$. The importance of the number $\gamma_H(m)$ stems out from the following three main facts (cf. \cite{GV1}):

	(i) They are codimension of the versal level $m$ strata.
	
	(ii) They can compute the isomorphism number $n_H$.
	
	(iii) They are a main source of invariants that go up under specializations.

We call $(\gamma_H(m))_{m\in\NN}$ the {\em centralizing sequence} of $H$ and we call $s_H:=\gamma_H(n_H)$ the {\em specializing heigh}t of $H$.

The classification of Barsotti--Tate groups of level $1$ over $k$ is well-known. In the unpublished manuscript \cite{Kraft} H. Kraft showed that fixing $h:=c+d$, there are only finitely many such group schemes up to isomorphism. This result was re-obtained independently by Oort. Together with Ekedahl he used it to define and study a stratification of the moduli space of principally polarized abelian varieties over $k$. Their results can be found in \cite{Oort:stramoduli}.

A. Vasiu provided another way to parameterize Barsotti--Tate groups of level $1$ over $k$ .  Let $c,d$ and $h$ be as before. Let $\pi$ be a permutation of $J=\{1,2,\hdots,h\}$. Let $H_\pi$  be the $p$-divisible group whose Dieudonn\'{e} module is $(M, \varphi_\pi)$, where
\[
\varphi_\pi(e_i) = \left\{ 
\begin{array}{l l}
pe_{\pi(i)} & \quad 1\leq i\leq d,\\
e_{\pi(i)} & \quad d<i\leq r.\\
\end{array} \right.
\]
Then we have the following:

\begin{thm}[Vasiu]
Let $H$ be a $p$-divisible group of codimension $c$ and dimension $d$ over $k$.
Then $H[p]\cong H_\pi[p]$ for some permutation $\pi$.
\end{thm}

The $H_\pi$'s are called the  {\em canonical lifts} of Barsotti--Tate groups of level $1$ over $k$. Each $H_\pi$ is uniquely determined by $H_\pi[p]$ and by the fact that there exists a maximal torus $T$ of $GL_M$ such that $\varphi_\pi(\lie(T))=\lie(T)$. Hence $\varphi_\pi$ acts on $\End(M)[\frac{1}{p}]$ and $\lie(T)\subseteq \End(M)\subseteq \End(M)[\frac{1}{p}]$; cf. \cite[Corollary 11.1(d)]{Vasiu:modpshimura}. It turns out that some of the invariants of $H_\pi[p^m]$ can be calculated using the permutation $\pi$. To describe this, we use the following notations. 

Let $J_+=\{(i,j)\in J^2| i\leq d < j\}$, $J_0=\{(i,j)\in J^2| i,j\leq d \textrm{ or } i,j>d\}$ and $J_-=\{(i,j)\in J^2| j\leq d < i\}$, as illustrated by the following diagram 

\[ \left( \begin{array}{c:c} J_0 &  J_+ \\ \hdashline
J_- & J_0 \end{array} \right). \] 

For each pair $(i,j)\in J^2$, we assign a number $\varepsilon_{i,j}\in \{-1,0,1\}$ according to the rule \[\varepsilon_{i,j}=\left\{
\begin{array}{ll}
1 &\textrm{if } (i,j)\in J_+,\\
0 &\textrm{if } (i,j)\in J_0,\\
-1 &\textrm{if } (i,j)\in J_-.
\end{array}
\right.\]

Let $\mathcal{O}=((i_1,j_1), \dots, (i_l,j_l))$ be an orbit of $(\pi,\pi)$ on $J^2$. For simplicity we write $\varepsilon_s=\varepsilon_{i_s,j_s}$. Let $\varepsilon_\mathcal{O}=(\varepsilon_1,\dots,\varepsilon_l)$.
We make the following definition.

\begin{defn}
	Let $\mathcal{O}=((i_1,j_1), (i_2,j_2), \dots, (i_l,j_l))$ be an orbit and $\varepsilon_{\mathcal{O}}=(\varepsilon_1,\varepsilon_2,\dots,\varepsilon_l)$ as before. Let $n\in \NN^\ast$. A segment $<\varepsilon_s,\varepsilon_{s+1},\dots,\varepsilon_{t}>$ of $\varepsilon_{\mathcal{O}}$ is called a \emph{free linear segment of level $n$} if it satisfies the following conditions.
	\begin{itemize}
		\item We have $\varepsilon_s=-1, \varepsilon_{t}=1$.
		\item We have $\sum_{i=s}^t \varepsilon_{i} = 0$.
		\item For all $s\leq j<t$ , we have $-n\leq \sum_{i=s}^j \varepsilon_{i} < 0$.
		\item There exists a $j_0$, $s\leq j_0<t$, such that $\sum_{i=s}^{j_0} \varepsilon_{i}=-n$.
	\end{itemize}
	We write $a_n(\mathcal{O})$ for the number of free linear segments of level $n$ in $\varepsilon_{\mathcal{O}}$.
\end{defn}

\begin{defn}
	Let $\mathcal{O}=((i_1,j_1), (i_2,j_2), \dots, (i_l,j_l))$ be an orbit and $\varepsilon_{\mathcal{O}}=(\varepsilon_1,\varepsilon_2,\dots,\varepsilon_l)$ as before. Let $n\geq 0$. Then $\mathcal{O}$ is called a \emph{circular orbit of level $n$} if $\varepsilon_\mathcal{O}$ satisfies the following conditions.
	\begin{itemize}
		\item $\sum_{i=1}^l \varepsilon_i = 0.$
		
		\item For all $u,v\in \{1,2,\dots,l\}$, $ |\sum_{i=u}^v \varepsilon_i| \leq n$.
		
		\item There exist $u,v \in \{1,2,\dots,l\}$ such that $ |\sum_{i=u}^v \varepsilon_i| = n$.
	\end{itemize}
	We write $\mathcal{C}_\pi(n)$ for the set of circular orbits of level $n$.
\end{defn}

Our main results are

\begin{thm}
	Let $\pi$, $H_\pi$ be as before, then
	\[\gamma_{H_\pi}(m)=\sum_{\mathcal{O}}\sum_{n=1}^{m}a_n(\mathcal{O}).\]
	\qed
\end{thm}

\begin{thm}
	Let $\pi$, $H_\pi$ be as before, then the number of connected components of $\Ends(H_\pi[p^m])$ is $p^{c_m}$ where $$c_m=\sum_{n=0}^{m-1}\sum_{\mathcal{O}\in\mathcal{C}_\pi(n)}(m-n)|\mathcal{O}|.$$
	\qed
\end{thm}

The case $m=1$ is a special case of \cite[1.2. Basic Theorem A]{Vasiu:modpshimura}. 

The first three chapters give a (mostly) self-contained introduction to finite flat group schemes over $k$, the ring of Witt vector $W(k)$ over $k$ and the Dieudonn\'e theory. Our new results are proved in the last chapter.

\chapter{Group Schemes}\label{gs}

This chapter has essentially the character of a survey of group schemes. Its aim is to review basic definitions and related concepts within the framework of algebraic geometry, as well as to give a brief account of the fundamental result of P. Gabriel on the structure of finite commutative group schemes over perfect fields. At the end of this chapter, we will give the motivation and definition of a $p$-divisible group. We begin by introducing the categorical language of describing group structures in various categories.


\section{Groups in categories}\label{Groups}

Let $\mathcal{C}$ be a category and $S$ an object of $\mathcal{C}$. Let $\mathcal{C}_S$  be the category of $S$-objects. Its objects are pairs $(X,f)$ where $X$ is an object of $\mathcal{C}$ and $f:X\rightarrow S$ is a morphism in $\mathcal{C}$.  A morphism from the pair $(X,f)$ to the pair $(Y,g)$ is a morphism $h:X\rightarrow Y$ in $\mathcal{C}$ such that the diagram
\begin{displaymath}
\xymatrix
{
X \ar[rr]^{h}\ar[rd]_f & & Y \ar[ld]^g \\
& S  &
}
\end{displaymath}
commutes.

Sometimes it is useful to consider the following more general situation. Let $\Gamma\subset \text{Hom} (S,S)$ be a monoid with unit. That is to say, $\Gamma$ is a subset of the set of endomorphisms of $S$ such that: (1) it contains the identity morphism $1_S: S\rightarrow S$; (2) the composition of two endomorphisms in $\Gamma$ is again in $\Gamma$. Let $\sigma\in \Gamma$. A $\sigma$-morphism from  $(X,f)$ to $(Y,g)$ is a morphism $h:X\rightarrow Y$ in $\mathcal{C}$ such that the diagram
\begin{displaymath}
\xymatrix
{
X \ar[r]^{h}\ar[d]_f  & Y \ar[d]^g \\
S \ar[r]_\sigma &S
}
\end{displaymath}
commutes. The composition of a $\sigma$-morphism by a $\tau$-morphism is clearly a $\tau\circ\sigma$-morphism. Therefore, we can modify $\mathcal{C}_S$ by defining $\text{Hom}((X,f),(Y,g))$ to be the set of all $\sigma$-morphisms  from  $(X,f)$ to $(Y,g)$  for some $\sigma\in\Gamma$. Note that if $\Gamma$ contains only the identity morphism, then this generalization reduces to the previous one.

Suppose that $\mathcal{C}_S$ is a category with products. Let $(Z,h)=(X,f)\times (Y,g)$ be a product of $(X,f)$ and $(Y,g)$ in $\mathcal{C}_S$. Then the object $Z$ is called a fibre product of $X$ and $Y$ in the original category $\mathcal{C}$ and is denoted by $X\times_SY$. If for each object $S$ in $\mathcal{C}$ there is a product defined on $\mathcal{C}_S$, we say $\mathcal{C}$ is a category with fibre products.

An object $e$ is called a final object in $\mathcal{C}$ if for any object $X$ in $\mathcal{C}$ there is a unique morphism from $X$ to $e$. Since all final objects are canonically isomorphic, we may therefore consider the final object to be unique. If $e$ is a final object, the category $\mathcal{C}_e$ is equivalent to $\mathcal{C}$. The category $\mathcal{C}_S$ has a final object: the pair $(S,1_S)$.

In what follows we focus on concrete categories $\mathcal{C}$ that satisfy the following two axioms: (1) $\mathcal{C}$ is a category with products; (2) $\mathcal{C}$ has a final object $e$.

Let $X$ be an object in $\mathcal{C}$. In the formulation of the next definition we denote by $d:X\rightarrow X\times X$ the diagonal morphism, by $\epsilon: X\rightarrow e$ the unique morphism from $X$ to $e$, and by $s: X\times X\rightarrow X\times X$ the morphism that interchanges the two factors.

A \emph{group law} on $X$ is a morphism $c: X\times X \rightarrow X$ that satisfies the following three axioms:
\begin{itemize}
\item Axiom of associativity. The diagram
\begin{displaymath}
\xymatrix
{
X\times X\times X \ar[r]^{c\times1}\ar[d]_{1\times c}  & X\times X \ar[d]^c \\
X\times X \ar[r]_{c} & X
}
\end{displaymath}
commutes.

\item Axiom of a left unit. There is a morphism $\eta: e\rightarrow X$ such that the diagram
\begin{displaymath}
\xymatrix
{
X\times  X\ar[rr]^{(\eta\circ\epsilon)\times 1_X}  & & X\times X \ar[ld]^c \\
& X\ar[lu]^d &
}
\end{displaymath}
commutes.

\item Axiom of a left inverse. There is a morphism $a:X\rightarrow X$ such that the diagram
\begin{displaymath}
\xymatrix
{
X\times X \ar[r]^{a\times1}  & X\times X \ar[d]^c \\
X \ar[r]_{\eta\circ\epsilon}\ar[u]^d & X
}
\end{displaymath}
commutes.
\end{itemize}

One can formulate the axiom of a right unit and of a right inverse in the obvious way. It is not hard to check that the `left unit' morphism is also a `right unit', and is uniquely determined. Similarly, the `left inverse' morphism is also a `right inverse' and is unique.

The group law $c$ is said to be commutative or abelian if moreover it satisfies the following axiom:

\begin{itemize}
\item Axiom of commutativity. The diagram
\begin{displaymath}
\xymatrix
{
X\times  X\ar[rr]^{s}\ar[rd]_c  & & X\times X \ar[ld]^c \\
& X &
}
\end{displaymath}
commutes.
\end{itemize}

The object $X$ in the category $\mathcal{C}$ together with a group law $c$ defined on it is called a $\mathcal{C}$-group. It is called a commutative $\mathcal{C}$-group if $c$ is commutative. We often denote the $\mathcal{C}$-group simply by $X$ if the group law $c$ is subunderstood.

A morphism of $\mathcal{C}$-groups from $(X,c_X)$ to $(Y, c_Y)$ is a morphism $f:X\rightarrow Y$ such that the diagram

\begin{displaymath}
\xymatrix
{
X\times  X\ar[r]^{f\times f}\ar[d]_{c_X}  & Y\times Y \ar[d]^{c_Y} \\
X\ar[r]_f & Y
}
\end{displaymath}
commutes.

Starting from the axioms, it is not hard to show that $\mathcal{C}$-group morphisms commute not only with group laws, but also with the unit and inverse morphisms.

The class of $\mathcal{C}$-groups with these morphisms form a category. The class of commutative $\mathcal{C}$-groups form a full subcategory. For example, if $\mathcal{C}$ is the category of sets, topological spaces or analytic manifolds, then the corresponding $\mathcal{C}$-groups are represented by abstract groups, topological groups or Lie groups respectively.

\section{Schemes}

There are essentially two ways to define a scheme in modern algebraic geometry: (1) the geometric method in which one defines schemes to be topological spaces with structure sheaves of rings that are locally affine, a definition that is given in most books, cf. \cite[Chapter 2]{Hartshorne}; (2) the functorial method in which one defines schemes in terms of certain functors, cf. \cite[Chapter 1]{DG}. Both definitions are important. In the theory of algebraic groups however, the functorial point of view is often more beneficial for different group theoretical constructions. We will briefly recall the definition of schemes as functors in this section.

The following are the notations we will use for some categories in this section:

\begin{itemize}
\item $\mathcal{S}$: the category of sets.
\item $\mathcal{R}$: the category of rings. By a ring we will always mean a commutative ring with unit.
\item $\mathcal{R}_k$: the category of $k$-algebras.  Here $k$ will be some base field. A $k$-algebra is a ring $R$ with a ring homomorphism $k\rightarrow R$.
\item $\mathcal{G}$: the category of geometric spaces. A geometric space is a pair $(X, \mathcal{O}_X)$ where $X$ is a topological space and $\mathcal{O}_X$ is a sheaf of local rings on $X$ (i.e., $\mathcal{O}_X$ is a ring sheaf such that for each $x\in X$, the stalk $\mathcal{O}_{X,x}$ is a local ring). We often write $X$ instead of $(X, \mathcal{O}_X)$ if the structural ring sheaf $\mathcal{O}_X$ is implied.
\end{itemize}

\subsection{Affine schemes}

A {\it $\ZZ$-functor} is a (covariant) functor from $\mathcal{R}$ to $\mathcal{S}$. As a trivial but important example, the forgetful functor $O$ that takes $R\in \mathcal{R}$ to the underlying set $O(R)=R$ , is a $\ZZ$-functor.

The $\ZZ$ here refers to the base ring $\ZZ$, as every ring is a $\ZZ$-algebra in a unique way. Later on we will extend base ring from $\ZZ$ to $k$, and define $k$-functors in the same way.

\begin{exam}
Let $A\in \mathcal{R}$. Let $\textrm{Sp\,}A$ be the functor that takes $R\in \mathcal{R}$ to $\textrm{Sp\,}A(R)=\textrm{Hom}_\mathcal{R}(A,R)$, the set of ring homomorphisms from $A$ to $R$. If $R\xrightarrow{f} S$ is a ring homomorphism, $\textrm{Sp\,}A(f)$ takes $\phi\in\textrm{Sp\,}A(R)$ to $f\circ\phi\in\textrm{Sp\,}A(S)$. We call $\textrm{Sp\,}A$ the $\ZZ$-functor represented by $A$. This functor is sometimes denoted by $\text{Hom}_\mathcal{R}(A,\ast)$.
\end{exam}

\begin{exam}
Let $X\in\mathcal{G}$ be a geometric space. Define $S(X)$ to be the $\ZZ$-functor that takes a ring $R$ to $\textrm{Hom}_\mathcal{G}(\textrm{Spec\,}R, X)$, where $\textrm{Spec\,}R=(\textrm{Spec\,}R, \mathcal{O}_{\textrm{Spec\,}R})$ is the prime spectrum of $R$. In particular, $S(\textrm{Spec\,}A)(R)=\textrm{Hom}_\mathcal{G}(\textrm{Spec\,}R, \textrm{Spec\,}A)=\textrm{Hom}_\mathcal{R}(A,R)=\textrm{Sp\,}A(R)$.
\end{exam}

\begin{defn}
An affine scheme is a $\ZZ$-functor that is isomorphic to the functor $\textrm{Sp\,}A$ for some $A\in\mathcal{R}$.
\end{defn}

For example, $O$ is an affine scheme because $O(R)=R\cong\textrm{Sp\,} \ZZ[x](R)$.

\subsection{Schemes}

Loosely speaking, a scheme is a $\ZZ$-functor that has an open covering by affine schemes. To define this, first we need to define what it means for a subfunctor to be open.

Let $A\in\mathcal{R}$. Recall that as a topological space, $\textrm{Spec\,}A$ has a basis of closed sets of the form $V(I)=\{P \textrm{ a prime ideal of } A|I\subseteq P\}$ for some ideal $I$ of $A$. Let $D(I)$ be the open set $\textrm{Spec\,}R$\textbackslash$V(I)$.

If $\phi: A\rightarrow R$ is a morphism in $\mathcal{R}$, we have a continuous map $\tilde{\phi}: \textrm{Spec\,}R\rightarrow \textrm{Spec\,}A$, and it is easy to check that $\tilde{\phi}^{-1}(D(I))=D(R\phi(I))$. So $\tilde{\phi}$ factors through $D(I)$ if and only if $R\phi(I)=R$. As a result, we have $S(D(I))(R)=\textrm{Hom}_\mathcal{G}(\textrm{Spec\,}R, D(I))=\{\phi\in\text{Hom}_\mathcal{R}(A,R) | R\phi(I)=R\}$. We denote this functor by $(\textrm{Sp\,}A)_I$ and call it the subfunctor defined by $I$.

\begin{defn}
Let $X$ be a  $\ZZ$-functor and $Y$ a subfunctor of $X$. We say $Y$ is an open subfunctor of $X$ if for each $A\in \mathcal{R}$ and every morphism $f: \textrm{Sp\,}A\rightarrow X$, the subfunctor $f^{-1}(Y)$ of $\textrm{Sp\,}A$ is defined by some ideal $I$ of $A$.
\end{defn}

In particular, take $X=\textrm{Sp\,}A$ and $f=\text{id}$, one easily gets that the open subfunctors of $\textrm{Sp\,}A$ are exactly those of the form $(\textrm{Sp\,}A)_I$ for some ideal $I$.

\begin{exam}
Let $X\in\mathcal{G}$, let $Y$ be an open subspace of $X$.  Then $S(Y)$ is an open subfunctor of $S(X)$ because if $\alpha: \textrm{Spec\,}A\rightarrow X$ is an element of $S(X)(A)=\text{Hom}_\mathcal{G}(\textrm{Spec\,}A,X)$, $\alpha^{-1}(Y)$ is an open subset of $\textrm{Spec\,}A$ and therefore it is of the form $D(I)$ for some ideal $I$.
\end{exam}

\begin{defn}
Let $X$ be a $\ZZ$-functor. A family of subfunctors $(Y_i)_{i\in\Lambda}$ is said to cover $X$ if for any field $k$, we have $X(k)=\underset{i\in\Lambda}{\cup}Y_i(k)$.
\end{defn}

\begin{exam}
Let $X\in\mathcal{G}$, let $(Y_i)_{i\in\Lambda}$ be an (open) covering of $X$, then $(S(Y_i))_{i\in\Lambda}$ is an (open) covering of $S(X)$.  In particular, let $A\in\mathcal{R}$ and let $(f_i,x_i)_{i\in\Lambda}$ be a partition of unity in $A$. That is, $\Lambda$ is finite and $f_i, x_i$ are elements of $A$  such that $1=\underset{i\in\Lambda}{\sum}f_i x_i$. Then $(S(D(Af_i)))_{i\in\Lambda}$ is an open covering of $S(\text{Spec}A)=\textrm{Sp\,}A$.
\end{exam}

Next we shall need the notion of a local functor.

Let $X$ be a $\ZZ$-functor, $R\in\textrm{Ob}(\mathcal{R})$ and $1=\underset{i\in\Lambda}{\sum}f_i x_i$ a partition of unity in $R$. We associate to $R$ the sequence of maps
\begin{equation}\label{local}
X(R)\xrightarrow{f}\prod_i X(R_{f_i})\underset{h}{\overset{g}{\rightrightarrows}} \prod_{i,j}X(R_{f_if_j})\tag{$\ast$}
\end{equation}
defined as follows: if $\alpha_i$ (resp. $\alpha_{ji}$) denotes the canonical homomorphisms from $R$ to $R_{f_i}$ (resp. from $R_{f_i}$ to $R_{f_if_j}$), we set
\begin{align*}
\textrm{Pr}_i\circ f &= X(\alpha_i),\\
\textrm{Pr}_{i,j}\circ g &= X(\alpha_{ji})\circ \textrm{Pr}_i, \\
\textrm{Pr}_{i,j}\circ h &= X(\alpha_{ij})\circ \textrm{Pr}_j.
\end{align*}
where $\textrm{Pr}_i$ (resp.  $\textrm{Pr}_{i,j}$ ) is the projection from the first (resp. second) product to the $i$ (resp. $i,j$) factor.

\begin{defn}
A $\ZZ$-functor $X$ is said to be local if for each $R\in\mathcal{R}$ and each partition of unity $1=\underset{i\in\Lambda}{\sum}f_i x_i$ in $R$, the sequence \eqref{local} is exact.
\end{defn}

\begin{defn}
A $\ZZ$-functor is called a scheme if it is local and has a covering by affine open subfunctors.
\end{defn}

\begin{rmk}
Some readers are probably more familiar with the following definition.

\begin{itemize}
\item An affine scheme is a geometric space $(X,\mathcal{O}_X)$ which is isomorphic to the prime spectrum of some ring $A$.

\item A scheme is a geometric space $(X,\mathcal{O}_X)$ in which every point has an open neighborhood $U$ such that the topological space $U$, together with the restricted sheaf $\mathcal{O}_X|_U$ is an affine scheme.
\end{itemize}

Our definition is equivalent to this one in the following sense.

We have seen in a previous example that there is a functor $S$ from the category of geometric spaces to the category of $\ZZ$-functors. This functor $S$ has a left adjoint functor, usually denoted by $|\ast|$, called the geometric realization functor. Under these two, the $\ZZ$-functors that satisfy our definition of schemes correspond exactly to the geometric spaces that satisfy the above definition in the remark. Also, $\textrm{Sp}\,A$ corresponds to $\spec A$, the prime spectrum of the ring $A$, and we will not distinguish between these two.

\end{rmk}

\begin{exam}
Let $X\in\mathcal{G}$, then $S(X)$ is local. We have:
\begin{align*}
S(X)(R) &= \mathcal{G}(\textrm{Spec\,}R,X),\\
S(X)(R_{f_i}) &= \mathcal{G}(\textrm{Spec\,}R_{f_i},X)=\mathcal{G}(D(f_i),X),\\
S(X)(R_{f_if_j}) &= \mathcal{G}(\textrm{Spec\,}R_{f_if_j},X)=\mathcal{G}(D(f_i)\cap D(f_j),X)=\mathcal{G}(D(f_if_j),X).
\end{align*}
The exactness of \eqref{local} means that a morphism $\alpha: \textrm{Spec\,}R\rightarrow X$ is determined by its restrictions to the open sets $D(f_i)$ and that these restrictions satisfy the usual matching conditions. It follows that $S(X)$ is a scheme if $X$ is a scheme in terms of the definition in the remark.
\end{exam}

Let $k$ be a field. By replacing the category of rings $\mathcal{R}$ with the category of $k$-algebras $\mathcal{R}_k$, we can easily generalize the definitions above to $k$-functors, affine schemes over $k$ and schemes over $k$.

\section{Affine group schemes}

From now on we shall fix a base field $k$.

\begin{defn}
	An \emph{affine (resp. affine commutative) group scheme} over $k$ (or over $\spec k$) is an affine scheme $G$ over $k$ with a group law (resp. commutative group law).
\end{defn}	

Unless otherwise stated, all group schemes in this paper are assumed to be affine and commutative.

\begin{defn}
	Let $G=\spec A$ be a group scheme over $k$. We say $G$ is {\it finite} if $A$ is a finite dimensional $k$-vector space . We say $G$ is {\it algebraic} if $A$ is a finitely generated $k$-algebra.
\end{defn}	 

\begin{defn}\label{schemedim}
	Let $G=\spec A$ be an algebraic group scheme over $k$. The \emph{dimension} of $G$, denoted by $\dim G$, is the Krull dimension of $A$.
\end{defn}

When $A$ is an integral domain, this is equal to the transcendence degree of the field of fractions of $A$ over $k$.	
	
\begin{defn}
	Let $G=\spec A$ be a finite group scheme over $k$. The \emph{order} (or \emph{rank}) of $G$, denoted by $|G|$, is the dimension of $A$ as a $k$-vector space.
\end{defn}	

The significance of the order is shown in the following 

\begin{thm}[Deligne]
	Let $G=\spec A$ be a finite group scheme over $k$. Let $m=|G|$. Then $[m]$ annihilates $G$, i.e., repeating the group law $m$ times gives a form vanishing identically on the scheme (the neutral element).
\end{thm}	
\begin{proof}
	\cite[p.~4]{TO}
\end{proof}

Group schemes, algebraic group schemes and finite group schemes over $k$ all form categories with fibre products and final object $\spec k$. 

Let $G=\spec A$ be an affine group scheme over $k$ (not necessarily commutative). According to Yoneda's lemma, to give a group law $\pi_G: G\times G\rightarrow G$ is the same as to give a $k$-algebra homomorphism $\Delta: A\rightarrow A\otimes_k A$ satisfying the following dual axioms:

\begin{itemize}

\item The diagram

\begin{displaymath}\tag{Coassociativity}
    \xymatrix{
A\otimes A \otimes A   & A\otimes A \ar[l]_-{\Delta \otimes 1_A}\\
              A\otimes A \ar[u]^{1_A \otimes \Delta} & A\ar[l]^{\Delta}\ar[u]_{\Delta}
}
\end{displaymath}
commutes.

\item There is a $k$-algebra homomorphism $\varepsilon: A\rightarrow k$ such the diagram

\begin{displaymath}\tag{Left Counit}
    \xymatrix{
k\otimes_k A   & A\otimes_k A \ar[l]_-{\varepsilon \otimes 1_A}\\
& A\ar[u]_{\Delta} \ar[ul]^\cong
}
\end{displaymath}
commutes.

\item There is a $k$-algebra homomorphism $s: A\rightarrow A$ such the diagram

\begin{displaymath}\tag{Left Coinverse}
    \xymatrix{
A  & A\otimes_k A \ar[l]_{s\otimes 1_A}\\
k \ar[u] & A\ar[l]^\varepsilon \ar[u]_{\Delta}
}
\end{displaymath}
commutes.
\end{itemize}

This simple observation leads to the following definition:

\begin{defn}
A {\it Hopf algebra} over $k$ (or a \emph{Hopf} $k$-algebra) is a $k$-algebra $A$ with $k$-algebra homomorphisms $\Delta: A\rightarrow A\otimes_kA$, $\varepsilon:  A \rightarrow k$ and $S: A\rightarrow A$ satisfying the above three axioms.
\end{defn}
Hopf algebras are also called {\it bialgebras} in some books.

From the above discussion, we have an anti-equivalence between the category of affine group schemes over $k$ and the category of Hopf $k$-algebras. Here are some first examples of affine group schemes over $k$.

\begin{exam}\label{aaa}

1. Consider the functor $\boldsymbol{SL_n}$ that takes a $ k$-algebra $R$ to $\boldsymbol{SL_n}(R)$, the set of $n$ by $n$ matrices with entries in $R$ and determinant $1$. Usual matrix multiplication makes $\boldsymbol{SL_n}$ into an affine group scheme over $k$. This group  is represented by 
$$k[x_{ij}|1\leq i,j\leq n]/(\det(x_{ij})_{1\leq i,j\leq n}-1)$$, with
\begin{align*}
\Delta(x_{ij}) &= \sum_{k=1}^n x_{ik}\otimes x_{kj},\\
\varepsilon(x_{ij}) &= \delta_{ij},\\
S(x_{ij}) &= (-1)^{i+j}\text{det}(x_{st})_{s\neq j, t\neq i}.
\end{align*}

2. Similarly, one can define the group of invertible matrices $\boldsymbol{GL_n}$, the group of upper triangular invertible matrices $\boldsymbol{T_n}$ and the group of upper triangular unipotent matrices $\boldsymbol{U_n}$. Their names are fairly self-explanatory. They are represented by
\begin{align*}
&k[x_{ij},\frac{1}{\det(x_{ij})_{1\leq i,j\leq n}}| 1\leq i,j\leq n], \\
&k[x_{ij}, \frac{1}{x_{11}\cdots x_{nn}}|1\leq i,j\leq n], \\
&k[x_{ij}| 1\leq j< i \leq n],
\end{align*}
respectively.  In particular, $\boldsymbol{GL_1}$ is an important group and is usually denoted by $\boldsymbol{G_m}$. It is represented by \[k[x,y]/(xy-1) = k[x, \frac{1}{x}]\] with
\begin{align*}
\Delta(x) &=x\otimes x,\\
\varepsilon (x) &= 1,\\
S(x) &=\frac{1}{x}.
\end{align*}

3. The functor $\boldsymbol{\mu_n}$ that takes $R$ to $\boldsymbol{\mu_n}(R)=\{x\in R | x^n =1 \}$, the group of $n$-th roots of unity in $R$, is represented by $k[x]/(x^n-1)$ with
\begin{align*}
\Delta(x)&=x\otimes x,\\
\varepsilon(x) &= 1,\\
S(x) &=\frac{1}{x}.
\end{align*}

4. If char($k$)=$p>0$, let $\boldsymbol{\alpha_p}$ be the functor that takes $R$ to $\boldsymbol{\alpha_p}(R)=\{x\in R | x^p = 0\}$. It is a group under addition. This functor is represented by $k[x]/(x^p)$ with
\begin{align*}
\Delta(x) &=x\otimes 1 + 1\otimes x,\\
\varepsilon(x) &= 0,\\
S(x) &=-x.
\end{align*}

\end{exam}

\begin{exam}[Diagonalizable Group schemes]

The anti-equivalence of categories discussed above allows us to define affine group schemes by constructing Hopf algebras. Let $M$ be an abelian group,  let $k[M]$ be the  group algebra ($k$-vector space with basis the elements of $M$, multiplication induced by that of $M$ on elements of the basis). For $m\in M$ let $\Delta(m)=m\otimes m, \varepsilon(m)=1, S(m)=m^{-1}$. This does give a Hopf algebra since the required identities are satisfied on the elements of the basis. The corresponding group schemes are called {\it diagonalizable group schemes}.

For example, $\boldsymbol{G_m}$ and $\boldsymbol{\mu_n}$ defined in ~\ref{aaa} are diagonalizable group schemes. It is not hard to see that $\boldsymbol{G_m}$ corresponds to the infinite cyclic group $\ZZ$, and $\boldsymbol{\mu_n}$ corresponds to the finite cyclic group $\ZZ/n\ZZ$. Actually, we have the following simple fact.

\begin{fact*}
Let $G=\spec A$ be an algebraic group scheme over $k$. If $G$ is diagonalizable and $k$ is algebraically closed, then $G$ is a finite product of copies of $\boldsymbol{G_m}$ and various $\boldsymbol{\mu_n}$.
\end{fact*}
\end{exam}

\begin{exam}[Constant group schemes]

Let $\Gamma$ be a finite group. In general, the functor that assigns $\Gamma$ to every $k$-algebra $R$ is not an affine group scheme. However, something very close to this functor can be made to be representable, hence the name of constant group schemes.

Let $A=k^\Gamma$ be the set of all functions from $\Gamma$ to $k$. The addition and multiplication of $k$ make $A$ into a ring in the obvious way. Furthermore, given $\sigma\in\Gamma$ let $e_\sigma$ be the function that is $1$ on $\sigma$ and $0$ elsewhere, then $\{e_\sigma\}$ is a basis of $A$ over $k$.
As a ring $k^\Gamma \cong k\times \cdots \times k$ and $\{e_\sigma\}$ is a set of orthogonal idempotents. Suppose $R$ is a $k$-algebra with no nontrivial idempotents, than a homomorphism $\varphi: A\rightarrow R$ must send one $e_\sigma$ to $1$ and all others to $0$, and thus can be identified naturally with $\sigma$. In other words, we have $\spec A(R)\cong \Gamma$ as sets.

Define $\Delta(e_\rho)=\sum_{\sigma \delta = \rho} (e_\sigma \otimes e_\delta)$, $S(e_\sigma)=e_{\sigma^{-1}}$ and $\varepsilon(e_\sigma)$ be $1$ if $\sigma$ is the identity and $0$ otherwise. This gives us a Hopf algebra structure on $A$ for which the induced group structure matches with the multiplication in $\Gamma$. The group scheme defined above is called constant group scheme associated to $\Gamma$ and is denoted by $\underline{\Gamma}$.

\end{exam}

Next, we make the construction of duals explicit for finite commutative groups schemes.

Let $G=\spec A$ be a finite commutative group scheme over $k$. Let $A^\vee=\textrm{Hom}_k(A,k)$ be the space of linear forms on A. From $A$ we have the following list of maps defining its Hopf $k$- algebra structure:

\begin{flalign*}
\qquad \Delta &: A\rightarrow A\otimes_k A, &\text{comultiplication}\\
\qquad \varepsilon &: A\rightarrow k, &\textrm{counit}\\
\qquad s &: A\rightarrow A, &\text{coinverse}\\
\qquad m &: A\otimes_k A \rightarrow A,&\text{multiplication in $A$}\\
\phi &: k\rightarrow A. &\text{$k$-algebra structure map}
\end{flalign*}

Passing to  the duals, we get a similar list of maps:

\begin{flalign*}
m^\ast &: A^\vee\rightarrow A^\vee\otimes_k A^\vee, &\\
\phi^\ast &: A^\vee\rightarrow k, &\\
s^\ast &: A^\vee \rightarrow A^\vee, &\\
\Delta^\ast &: A^\vee\otimes_k A^\vee \rightarrow A^\vee, &\\
\varepsilon^\ast &: k\rightarrow A^\vee. &
\end{flalign*}

Naturally one would ask the following question: is $A^\vee$ with the above list of maps also a Hopf $k$-algebra? For this we have the following:
\begin{thm}{ (Cartier Duality)}
Let $G=\spec A$ be a finite commutative group scheme over $k$. Then $A^\vee$ also represents a finite commutative group scheme over $k$ to be denoted by $G^\vee$.
\end{thm}

We call $A^\vee$ the (linear) dual of the Hopf $k$-algebra $A$, and $G^\vee=\spec A^\vee$ the (Cartier) dual of $G$.

\begin{proof}

We need to verify that $A^\vee$ with the above maps satisfies the axioms of a commutative Hopf $k$-algebra. We shall present here the verification of only one axiom to show that the commutativity assumption on $\spec A$ is truly needed.

We will check that $s^\ast$ is a $k$-algebra homomorphism. Thus we have to check that we have a commutative diagram (note that the multiplication in $A^\vee$ is given by $\Delta^\ast$)

\begin{displaymath}
    \xymatrix{A^\vee \otimes A^\vee \ar[r]^-{\Delta^\ast} \ar[d]_{s^\ast\otimes s^\ast} & A^\vee \ar[d]^{s^\ast}\\
              A^\vee \otimes A^\vee \ar[r]_-{\Delta^\ast} & A^\vee }
\end{displaymath}

which is equivalent to

\begin{displaymath}
    \xymatrix{A \otimes A   & A \ar[l]_-{\Delta}\\
              A \otimes A \ar[u]^{s \otimes s} & A\ar[l]^-{\Delta}\ar[u]_{s}. }
\end{displaymath}

In terms of the group law it means that we have a commutative diagram

\begin{displaymath}
    \xymatrix{G \times G \ar[r]^{\pi} \ar[d]_{(\tau,\tau)} & G \ar[d]^{\tau} \\
              G \times G \ar[r]^{\pi} & G, }
\end{displaymath}
i.e., for each $R$ and $a,b\in G(R)$, we have $(ab)^{-1} = a^{-1}b^{-1}$. This is true only when $G$ is commutative.

\end{proof}

\begin{exam}

Let $\Gamma$ is a finite abelian group. From the two examples above, we have the diagonalizable group scheme given by $k[\Gamma]$ and the constant group scheme given by $k^{\Gamma}$. These are dual to each other. By definition $k[\Gamma]^\vee$ is the set of all $k$-linear maps from $k[\Gamma]$ to $k$. Since $\Gamma$ is a basis for $k[\Gamma]$, we have a dual basis $\{\sigma^\ast | \sigma \in  \Gamma \}$ of $k[\Gamma]^\vee$.

The multiplication on $k[\Gamma]^\vee$ is given by $\Delta^\ast$. Thus
\[ \sigma^\ast \delta^\ast = \Delta^\ast(\sigma^\ast \otimes \delta^\ast) = (\sigma^\ast \otimes \delta^\ast)\Delta.\]

So for any $\rho \in \Gamma$, we have
\[ \sigma^\ast \delta^\ast (\rho) = (\sigma^\ast \otimes \delta^\ast)\Delta (\rho) = (\sigma^\ast \otimes \delta^\ast)(\rho \otimes \rho) = \sigma^\ast(\rho) \delta^\ast(\rho) \]
\[ =\left \{ \begin{array}{ll}
0 & \text{if}\ \sigma^\ast \neq \delta^\ast, \\
\sigma^\ast(\rho) & \text{if}\ \sigma^\ast = \delta^\ast. \\
\end{array}\right.
\]

So $\{\sigma^\ast | \sigma \in  \Gamma \}$ is a set of orthogonal idempotents of $k[\Gamma]^\vee$ and $k[\Gamma]^\vee$ is isomorphic to $k^{\Gamma}$ as $k$-algebras under the map $\sigma^\ast \mapsto e_\sigma$.

Similarly by looking into the comultiplication one can show that $k[\Gamma]^\vee$ is isomorphic to $k^{\Gamma}$ as Hopf $k$-algebras. This proves that the dual of a diagonalizable group scheme is a constant group scheme.

In particular, $(\underline{\mathbb{Z}/n\mathbb{Z}})^\vee=\boldsymbol{\mu_n}$.
\end{exam}

We end this section with the notion of a closed subgroup.

Let $G=\spec A$. A group scheme $H$ is called a {\it closed subgroup} of $G$ if $H=\spec A/I$ for some ideal $I$ of $A$. Intuitively, this means $H$ is defined by polynomial equations defining $G$ plus some additional ones corresponding to the ideal $I$. If one chooses additional equations randomly, the result cannot be expected to form a group. Thus the ideal $I$ has to satisfy certain conditions.

First, homomorphisms that factor through $A/I$ must be closed under multiplication. If $f, g: A\rightarrow R$ are two morphisms vanishing on $I$, their product $g\cdot h=(g,h)\Delta$ in $\spec A(R)$ must also vanish on $I$. This means that $\Delta(I)$ is mapped to $0$ under the epimorphism $A\otimes_kA\rightarrow A/I\otimes_kA/I$, hence is contained in $I\otimes_kA+A\otimes_kI$. Also, we need $s(I)\subseteq I$ since if $g$ is in $\spec A(R)$, $g^{-1}=g\circ s$ must also be in it. Finally, we need $\varepsilon (I)=0$ since the unit must be in $\spec A(R)$. Ideals $I$ of $A$ satisfying these three conditions are called {\it Hopf ideals} of $A$. If $I$ is a Hopf ideal, then $A/I$ with the induced maps of $\Delta, \varepsilon, s$ on the quotients is as well a Hopf $k$-algebra.

For example, the kernel $I_A$ of $\varepsilon: A\rightarrow k$ is a Hopf ideal, called the augmentation ideal of $A$. It corresponds to the trivial subgroup.

An important source of examples of closed subgroups is the kernels of homomorphisms. Let $G=\spec A$, $H=\spec B$, let $\Phi: G\rightarrow H$ be a homomorphism and $\phi: B\rightarrow A$ be the corresponding Hopf $k$-algebra homomorphism. For any $k$-algebra $R$, let $\textrm{Ker}(\Phi) (R)$ be the kernel of $\Phi(R): G(R)\rightarrow H(R)$. The elements of $\textrm{Ker}(\Phi) (R)$ can be described as pairs in $G(R)\times \spec k(R)$ having the same image in $H(R)$, i.e., $\textrm{Ker}(\Phi)=G\times_H\spec k$. Thus $\textrm{Ker}(\Phi)$ is represented by $A\otimes_Bk$, where $A$ and $k$ are viewed as $B$-algebra via the maps $\phi: B\rightarrow A$ and $\varepsilon_B: B\rightarrow k$. Tensoring the exact sequence $I_B\rightarrow B\xrightarrow{\varepsilon_B}k$ with $A$ over $B$, we see that $\textrm{Ker}(\Phi)$ is represented by $A/I_BA$. In particular, $\textrm{Ker}(\Phi)$ is a closed subgroup of $G$ which is a normal subgroup.

\begin{exam} Recall that $\boldsymbol{G_m}$ is represented by $k[x,\frac{1}{x}]$ and $\varepsilon(x)=1$. Thus $\varepsilon$ is essentially the evaluation map and its kernel is generated by $x-1$. Now consider the squaring homomorphism $(\ )^2: \boldsymbol{G_m}\rightarrow \boldsymbol{G_m}$. Here we have $A=k[x,\frac{1}{x}]$ and $B=k[y,\frac{1}{y}]$, and the squaring map corresponds to the Hopf $k$-algebra homomorphism $k[y,\frac{1}{y}]\rightarrow k[x,\frac{1}{x}]$ sending $y$ to $x^2$. Therefore, we have $I_BA=(x^2-1)A$ and $A/I_BA=A/(x^2-1)A$, which is the same as $k[x]/(x^2-1)$. Thus we conclude that the kernel is $\boldsymbol{\mu_2}$, which is what we naturally expected.

\end{exam}

\section{\'{E}tale and connected group schemes}

There are some important types of finite group schemes that we need to discuss before we can jump into the structure theorem of general finite group schemes over $k$. Recall that $\spec A$ is connected if $A$ has no non-trivial idempotents. A more subtle question arise when we take base extension into account. For example, take $\boldsymbol{\mu_3}$ which is represented by$A=k[x]/(x^3-1)$. If $k=\RR$, then $\spec A$ consists of two points, reflecting the decomposition $x^3-1=(x-1)(x^2+x+1)$. However, if we extend $k$ from $\RR$ to $\CC$, then $x^3-1$ splits completely; therefore $\boldsymbol{\mu_3}$ is isomorphic to the constant group scheme $\ZZ/3\ZZ$, and we get three connected components. Extensions of the base field may result in additional idempotents. Thus we need something that can detect potential idempotents. This can be handled using separable algebras, which generalize the notion of separable field extensions.

\begin{thm}\label{separablealgebra}
Let $\bar{k}$ be an algebraic closure of $k$, let $k_s$ be the separable closure of $k$ in $\bar{k}$, and let $A$ be a finite dimensional $k$-algebra. Then the following five statements are equivalent:

(1) the ring $A\otimes_k \bar{k}$ is reduced.

(2) $A\otimes_k\bar{k}\cong\bar{k}\times\bar{k}\times\cdots\times\bar{k}$.

(3) the number of $k$-algebra homomorphisms $A\rightarrow \bar{k}$ is equal to the dimension of $A$ over $k$.

(4) $A$ is a product of separable field extensions of $k$.

(5) $A\otimes_kk_s\cong k_s\times k_s\times\cdots\times k_s$.

If $k$ is perfect, these statements are also equivalent to the following one

(6) the ring $A$ is reduced.
\end{thm}

\begin{proof} \cite[p.~47]{Waterhouse}
\end{proof}

\begin{defn}
	A finite dimensional $k$-algebra $A$ satisfying the equivalent conditions (1) to (5) of Theorem \ref{separablealgebra} is called a {\it separable} $k$-algebra. 
\end{defn}

\begin{cor}
Sub-algebras, quotients, products and tensor products of separable $k$-algebras are separable.
\end{cor}
\begin{proof}
If $B$ is a sub-algebra of $A$, $B\otimes_k\bar{k}$ is a sub-algebra of $A\otimes_k\bar{k}$. If $A\otimes_k\bar{k}$ is reduced, so is $B\otimes_k\bar{k}$. Hence $B$ is separable.

Using (2) of Theorem \ref{separablealgebra}, we can prove the case of quotients, products and tensor products  similarly.
\end{proof}

\begin{defn}
A finite group scheme $G=\spec A$ is called {\it \'{e}tale} if $A$ is separable.
\end{defn}

Let $A$ be a finitely generated $k$-algebra. If $B$ is a separable sub-algebra of $A$, then $B\otimes_k\bar{k}$ is a separable sub-algebra of $A\otimes_k\bar{k}$. Since it is spanned by idempotents, its dimension is bounded by the number of connected components of $\spec A\otimes_k\bar{k}$, which is finite. Also, if $B_1$ and $B_2$ are two separable sub-algebras of $A$, then the composite $B_1B_2$ is also separable, since it is a quotient of $B_1\otimes_kB_2$. Hence there exists a largest separable sub-algebra inside $A$. We denote this by $\pi_0A$.

\begin{thm}
Let $k\subseteq K$ be fields, $A$, $B$ finitely generated $k$-algebras.

(1) $\pi_0A\otimes_kK=\pi_0(A\otimes_kK)$.

(2) $\pi_0(A\times B)=\pi_0A\times\pi_0B$.

(3) $\pi_0(A\otimes_k B)=\pi_0A\otimes_k\pi_0B$.
\end{thm}

\begin{proof}
\cite[p.~49-50]{Waterhouse}
\end{proof}

If $G=\spec A$, we write $\pi_0G=\spec \pi_0A$. Each idempotent of $A$ is in $\pi_0A$. There may also be nontrivial fields in $\pi_0A$, but since $\pi_0A\otimes_k\bar{k}\cong\bar{k}\times\bar{k}\times\cdots\times\bar{k}$, these fields reflect potential idempotents, which yields connected components of $X$ after performing base extension. The above theorem shows that $\pi_0A$ indeed captures all such potential idempotents.

\begin{thm}
Let $G=\spec A$ be an affine algebraic group scheme. Then the following four statements are equivalent:

(1) $\pi_0G$ is trivial (i.e., $\pi_0A=k$).

(2) $\spec A$ is connected.

(3) $\spec A$ is irreducible.

(4) $A/\textrm{nil}(A)$ is an integral domain.
\end{thm}

\begin{proof}
\cite[p.~51]{Waterhouse}
\end{proof}

\begin{defn}
An affine algebraic group scheme $G$ is said to be {\it connected} if it satisfies anyone of the above four equivalent conditions.
\end{defn}

Let $G=\spec A$ be an algebraic group scheme over $k$. Since the image of a separable $k$-algebra under any $k$-algebra homomorphism must again be separable, $\Delta$ must map $\pi_0A$ into $\pi_0(A\otimes_kA)=\pi_0A\otimes_k\pi_0A$. Similarly, $s$ maps $\pi_0A$ into $\pi_0A$. Thus, $\pi_0A$ with $\Delta, \varepsilon, s$ of $A$ restricted to $\pi_0A$ is a Hopf sub-algebra of $A$, or equivalently, $\pi_0G$ is an \'{e}tale group scheme. Note that the augmentation ideal $I_{\pi_0A}=I_A\cap\pi_0A$.

Let $H=\spec B$ be an \'{e}tale group scheme. Each $k$-homomorphism $B\rightarrow A$ has image in $\pi_0A$ hence can be decomposed as $B\rightarrow \pi_0A\hookrightarrow A$. Correspondingly, any homomorphism $G\rightarrow H$ factors through $\pi_0G$ as $G\rightarrow \pi_0G\rightarrow H$.

Let $G^0$ be the kernel of $G\rightarrow \pi_0G$. By the discussion at the end of Section 1.3, we know that it is a closed normal subgroup represented by $A/(I_A\cap\pi_0A)A$. Using a set of orthogonal idempotents $\{f_i\}$ corresponding to the decomposition of $\pi_0A$ into fields, we can write $A=\oplus f_iA$. The map $\varepsilon: A\rightarrow k$ must send exactly one $f_i$, say $f_0$, to 1, and all other $f_i$'s to $0$. Set $A^0=f_0A$. Then $\pi_0(A^0)=k$, $I_A\cap\pi_0A$ is generated by $1-f_0$, and the quotient representing $G^0$ is just $A^0$. To sum up, we have:

\begin{thm}
	Let $G=\spec A$ be an algebraic group scheme over $k$. Then $\pi_0A$ represents an \'{e}tale group $\pi_0G$. The kernel $G^0$ of $G\rightarrow \pi_0G$ is a connected closed normal subgroup of $G$ represented by the factor of $A$ on which $\varepsilon$ is nonzero.
	
\end{thm}

This $G^0$ is called the {\it connected component} of $G$.


\section{Finite group schemes over perfect fields}

Now assume $k$ is perfect.

\begin{lemma} Let $A$ be a finitely generated $k$-algebra. If $A/\textrm{nil}(A)$ is separable, then $\pi_0A\cong A/\textrm{nil}(A)$.
\end{lemma}

\begin{proof}
\cite[p.~52]{Waterhouse}
\end{proof}

\begin{thm}\label{decomp}
 Let $G$ be a finite group scheme over a perfect field. Then $G$ is the semi-direct product of $G^0$ and $\pi_0G$.
\end{thm}

\begin{proof}
Since $k$ is perfect, we know from Theorem ~\ref{separablealgebra} that a $k$-algebra is separable if and only if it is reduced.

Let $G=\spec A$. Then $A/\textrm{nil}(A)$ is separable. Further, $A/\textrm{nil}(A)\otimes_kA/\textrm{nil}(A)$ is separable, hence reduced. Therefore, the $k$-homomorphism
\[A\xrightarrow{\Delta} A\otimes_kA\rightarrow A/\textrm{nil}(A)\otimes_kA/\textrm{nil}(A)\] vanishs on $\textrm{nil}(A)$, hence factors through $A/\textrm{nil}(A)$. Thus $A/\textrm{nil}(A)$ defines a closed subgroup $H$ of $G$. By the above lemma this subgroup is isomorphic to $\pi_0G$ via the composite homomorphism $H\rightarrow G\rightarrow \pi_0G$.

\end{proof}

If $G$ is commutative, then the semi-product of Theorem \ref{decomp} is a direct product. If moreover $G$ is \'{e}tale, then the decomposition of  $G^\vee =G^{\vee 0}\times \pi_0G^\vee$ induces a dual decomposition $G\cong G^{\vee\vee}=(G^{\vee 0})^\vee\times (\pi_0G^\vee)^\vee$. The two factors are \'{e}tale groups (being subgroups of $G$) with a connected dual and an \'{e}tale dual respectively. Similarly, if $G$ is connected, it factors into a product of two connected groups with a connected dual and an \'{e}tale dual respectively. Thus for a general finite commutative group $G$, we have a four-fold decomposition.

\begin{cor}\label{decomppsition}
A finite group scheme over a perfect field $k$ splits canonically into four direct factors of the following disjoint type:

(1) \'{e}tale with \'{e}tale dual,

(2) \'{e}tale with connected dual,

(3) connected with \'{e}tale dual,

(4) connected with connected dual.
\end{cor}

\section{Finite group schemes over perfect fields of positive characteristic}

Let $X$ be a scheme over $\mathbb{F}_p$. The \emph{absolute Frobenius morphism} $\sigma_X: X\rightarrow X$ is the identity on points and the map $a\mapsto a^p$ on sections. Now suppose $k$ is a perfect field of positive characteristic $p$. For any scheme $X$ over $\spec{k}$ define $X^{(p)}$ as the fibre product and $F_X$ as the induced morphism in the following commutative diagram:

\[
\xymatrix{
X \ar@/_/[ddr] \ar@/^/[drr]^{\sigma_X} \ar@{.>}[dr]|-{F_X}\\
&X ^{(p)} \ar[d] \ar[r] & X\ar[d]\\
&\spec k \ar[r]^\sigma &\spec k}
\]

This $F_X$ is called the \emph{relative Frobenius morphism} of $X$ over $\spec{k}$. Note that this $F_X$ is functorial. For any group scheme $G$ over $k$, the morphism $F_G: G\rightarrow G^{(p)}$ is a homomorphism and if $G$ is commutative, $(G^{\vee})^{(p)}=(G^{(p)})^{\vee}$.

For any finite group $G$ over $k$ and $F_{G^{\vee}}:G^\vee\rightarrow (G^\vee)^{(p)}=(G^{(p)})^\vee$, define $V_G=(F_{G^\vee}))^\vee: G^{(p)}\rightarrow G$. The homomorphism $V_G$ dual to $F_{G^\vee}$ is called the \emph{Verchiebung homomorphism} of $G$. This $V_G$ is also functorial and compatible with products and base extensions. In fact, we have:

\begin{thm}
	For any finite group scheme $G$ over $k$, 
	$$V_G\circ F_G=p\cdot id_G,$$
	$$ F_G\circ V_G=p\cdot id_{G^{(p)}}.$$
\end{thm}

\begin{proof}
	\cite[p.~31]{Pink}
\end{proof}

\begin{exam}
	
	\begin{itemize}
		\item $F_G$ and $V_G$ are zero for $G=\boldsymbol{\alpha_p}$.
		\item $F_G$ is zero and $V_G$ is an isomorphism for $G=\boldsymbol{\mu_p}$.
		\item $V_G$ is zero for $G=\mathbb{Z}/n\mathbb{Z}$.
	\end{itemize}	
\end{exam}

We have seen that when $k$ is a perfect field, $G=\spec A$ is \'{e}tale if and only if $A$ is reduced. Also, it is true that when $A$ is finite dimensional over $k$, $A$ has no nontrivial idempotents if and only it $A$ is local. We say $G$ is \emph{local} if $G=G^0$ and \emph{reduced} if $G=G_{\textrm{red}}$. We say $G$ is of x-y type if $G$ is x and $G^\vee$ is y, x,y$\in$\{local, reduced\}.
 Thus Corollary \ref{decomppsition} can be phrased as the following:
	
\begin{cor} There is a unique  and functorial decomposition of $G$ as
	$$G=G_{rr}\oplus G_{rl}\oplus G_{lr}\oplus G_{ll}.$$
	where the direct summands are of reduced-reduced, reduced-local, local-reduced, and local-local type respectively.
\end{cor}
	
The functoriality implies the fact that any homomorphism between groups of different types is zero. This decomposition is also compatible with base extension.

The $n$-th iterations of Frobenius and Verschiebung are defined as the compositions of homomorphisms
$$F_G^n: G\xrightarrow{F_G}G^{(p)}\xrightarrow{F_{G^{(p)}}}G^{(p^2)}\rightarrow\cdots\rightarrow G^{(p^n)},$$
$$V_G^n: G^{(p^n)}\rightarrow\cdots\rightarrow G^{(p^2)}\xrightarrow{V_{G^{(p)}}} G^{(p)}\xrightarrow{V_G}G.$$

We say $F_G$ (resp. $V_G$) is \emph{nilpotent} if $F_G^n=0$ (resp. $V_G^n=0$) for some $n\geq 0$.

\begin{thm}
	For affine commutative group schemes $G$ we have the following equivalences:
	\begin{itemize}
		\item $G$ is reduced-reduced $\Leftrightarrow$ both $F_G$ and $V_G$ are isomorphisms.
		\item $G$ is reduced-local $\Leftrightarrow$ $F_G$ is an isomorphism and $V_G$ is nilpotent.
		\item $G$ is local-reduced $\Leftrightarrow$ $F_G$ is nilpotent and $V_G$ is an isomorphism.
		\item $G$ is local-local $\Leftrightarrow$ both $F_G$ and $V_G$ are nilpotent.
	\end{itemize}
	Also, in the decomposition $$G=G_{rr}\oplus G_{rl}\oplus G_{lr}\oplus G_{ll},$$ the orders of the four groups are: prime to $p$ for $G_{rr}$, and a power of $p$ for $G_{rl}$, $G_{lr}$ and $G_{ll}$.
\end{thm}
\begin{proof}
	\cite[~34-36]{Pink}
\end{proof}

\begin{exam}
	If $k=\bar{k}$, we have the following table for the category of finite group schemes over $k$:

\

\begin{tabular}{| c | c | c | c | c |}
  \hline                        
  Type & Order & F & V & Simple Group \\
\hline
  reduced-reduced & prime to p & isom & isom & $\ZZ/l\ZZ$, $\mu_l$, $l$ is a prime $\neq p$, order = $l$ \\
\hline
reduced-local & p power & isom & nilp & $\ZZ/p\ZZ$, order = $p$ \\
\hline
local-reduced & p power & nilp & isom & $\mu_p$, order = $p$ \\
\hline
 local-local & p power & nilp & nilp & $\alpha_p$, order = $p$ \\
\hline 
\end{tabular}
\end{exam}



\section{$p$-Divisible groups}
Let $k$ be a field of characteristic $p>0$. The concept of  $p$-divisible groups over $k$ were historically first introduced by Barsotti under the name equi-dimensional hyperdomain, cf. \cite{Barsotti}. In a 1966 paper \cite{Tate}, Tate called them $p$-divisible groups. Later in 1970, Grothendieck referred to them as Barsotti--Tate groups in \cite{Groth}. They are characteristic $p$ analogues of points of order a power of $p$ on an abelian variety. 

To give the motivation of $p$-divisible groups, consider an elliptic curve $E/k$. When we study the arithmetic of $E$, there is an object of great importance: the Tate module. Recall that if $l$ a prime number, the $l$-adic Tate module of $E$ is 
$$T_l(E):=\varprojlim_n E[l^n](\bar{k})$$
where $E[l^n]$ is the kernel (a group scheme) of the multiplication map $[l^n]: E\rightarrow E$ and the transition maps are the multiplication maps $l: E[l^{n+1}]\rightarrow E[l^n]$.

Since each $E[l^n]$ is a $\ZZ/l^n\ZZ$-module, we see that the Tate module has a natural structure as a $\ZZ_l$-module. But more importantly, it is a $\ZZ_l$ representation of the Galois group $G_{\bar{k}/k}$ since  $G_{\bar{k}/k}$ acts compatibly on $E[l^n](\bar{k})$.

The Tate module contains a great amount of information about $E$. To illustrate this, we collect some nice results below.

To begin with, we have the following

\begin{thm}
	Let $k$ be a finite field or a number field. Let $E_1, E_2/k$ be two elliptic curves. Let $l$ be a prime not equal to the characteristic of $k$. Then the natural map
	$$\Hom_k(E_1,E_2)\otimes \ZZ_l\rightarrow \Hom_{G_{\bar{k}/k}}(T_l(E_1),T_l(E_2))$$
	is an isomorphism.
\end{thm}

This was proven by Tate (cf. \cite{Tatefinite} ) for finite fields and by Faltings for number fields (cf. \cite{Faltings}).  

Also, we have the following amazing theorem:

\begin{thm}[N\'eron-Ogg-Shafarevich]
	Let $k$ be a $p$-adic local field, $l\neq p$ a prime and $E/k$ an elliptic curve. Then $E$ has good reduction if and only if the $G_{\bar{k}/k}$-representation $T_l(E)$ is unramified.
\end{thm}	

Recall that $E/k$ has good reduction if and only if there is an elliptic curve $\mathcal{E}/\mathcal{O}_k$ whose generic fiber is $E$. Thus $T_l(E)$ is able to detect the ability to lift $E$ to an elliptic curve over the ring of integers.

Lastly, we have the following standard theorem of Tate:

\begin{thm}[Tate's isogeny theorem] Let $E_1,E_2/\FF_q$ be elliptic curves over finite fields and $l$ a prime coprime to $q$. Then $E_1$ and $E_2$ are isogenous if and only if $T_l(E_1)\cong T_l(E_2)$.	
\end{thm}	

Thus the Tate module captures the isogeny class of an elliptic curve over a finite field. Moreover, one can show that $T_l(E_1)\cong T_l(E_2)$ if and only if  $E_1(\FF_{p^r})=E_2(\FF_{q^r})$ for all $r\geq 1$. So the Tate module also captures the number of points of an elliptic curve over all finite fields.

Note that in above applications we require that $l\neq p=\textrm{char}(k)$. The $p$-adic Tate module, however, is in many respects not the right object to consider. For instance, let $E/\FF_q$ be an elliptic curve, the $p$-adic Tate module is extremely simple: if $E$ is supersingular then $T_p(E)=0$; if $E$ is ordinary then $T_p(E)$ is $\ZZ_p$ with the action of some character.

We have seen that for $l\neq p$, the $l$-adic Tate-module captures the full system of group schemes $E[l^n]$. The fact that this system can be encoded into a single $\ZZ_l$-module with Galois action is due to the fact that $E[l^n]$ is \'etale for every $n$. But for $E[p^n]$ it is never \'etale,  and as such, the geometric points do not capture much of the information. So we should really consider the full system of group schemes $E[p^n]$. It turns out that it is most convenient to put these into an inductive system, and in this way we arrive at the $p$-divisible group of an elliptic curve ( or abelian variety in general).

Let us now give the definition of a $p$-divisible group in a general setting.

\begin{definition} Let $p$ be a prime number and $h$ an integer $\geq 0$. A {\it $p$-divisible group} over $k$ of {\it height} $h$ is an inductive system $H=(G_n,i_n), n\geq0$, where $G_n$ is a finite group scheme over $k$ of order $p^{nh}$, such that the sequences
	$$0 \rightarrow G_n \xrightarrow{i_n} G_{n+1} \xrightarrow{p^n} G_{n+1}$$ are exact for all $n\geq0$.
	
	A {\it homomorphism} of $p$-divisible groups from $H=(G_n, i_n)$ to $H'=(G'_n,i'_n)$ is a collection of morphisms $(f_n:G_n\rightarrow G'_n)$ compatible with the structure of $p$-divisible groups: $f_{n+1}\circ i_n=i'_n\circ f_n$ for all $n\geq0$.
	
\end{definition}

Here are some basic properties of $p$-divisible groups over $k$.

(1) The group $G_m$ can be identified with the kernel of $G_{m+n}\xrightarrow{p^n}G_{m+n}$. Let $i_{m,n}: G_n\rightarrow G_{m+n}$ be the closed immersion $i_{n+m-1}\circ\cdots\circ i_{n+1}\circ i_n$. An easy induction shows that $G_{m+n}\xrightarrow{p^m}G_{m+n}$ can be factored uniquely through $i_{n,m}$ via a morphism $j_{m,n}: G_{m+n}\rightarrow G_n$. In other words, we have a commutative diagram
\begin{displaymath}
\xymatrix
{
	G_{m+n}\ar[rr]^{p^m}\ar[rd]_{j_{m,n}}  & & G_{m+n} \\
	& G_n\ar[ru]_{i_{n,m}}&.
}
\end{displaymath}

(2) The sequence $0\rightarrow G_m\xrightarrow{i_{m,n}} G_{m+n}\xrightarrow{j_{m,n}} G_n\rightarrow 0$ is exact: It is clearly left exact. But since orders are multiplicative in exact sequences, a consideration of orders shows that it is exact.

(3) Let $m=1$ in (2). We have an exact sequence 
$$0\rightarrow G_1\xrightarrow{i_{1,n}} G_{1+n}\xrightarrow{j_{1,n}} G_n\rightarrow 0.$$
Taking Cartier duals gives the dual exact sequence
$$0\rightarrow G_n^{\vee}\xrightarrow{j_{1,n}^{\vee}} G_{1+n}^{\vee}\xrightarrow{i_{1,n}^{\vee}} G_1^{\vee}\rightarrow 0.$$
The morphisms $j_{1,n}^\vee$ are the kernels of $G_{1+n}^\vee\xrightarrow{p^n} G_{1+n}^\vee$. Also, since duality preserves the order of finite groups, we get that $H^\vee=(G_n^\vee,j_{1,n}^\vee)$ is a $p$-divisible group of the same height $h$. This $H^\vee$ is called the \emph{dual} $p$-divisible group of $H$.

\begin{exam}
	Let $\textrm{char}(k)=p>0$. Let $A$ be an abelian variety of dimension $g$ over $k$. For $n\geq 0$, let $G_n=A[p^n]$ be the kernel of $[p^n]: A\xrightarrow{p^n}A$. It is well known that $G_n$ is a finite flat group scheme over $k$ of order $p^{2gn}$. They fit in exact sequences
	$$0 \rightarrow G_n \xrightarrow{i_n} G_{n+1} \xrightarrow{p^n} G_{n+1}$$
	where $G_n\xrightarrow{i_n} G_{n+1}$ is the natural inclusion. Therefore, we get a $p$-divisible group of height $2g$, called the $p$-divisible group of $A$, and denoted by $A[p^\infty]$.	
	
	In fact, if $k$ is algebraically closed, then each $p$-divisible group $H$ over $k$ is a direct summand of $A[p^\infty]$ for some abelian variety $A$ over $k$.
\end{exam}

\chapter{The Ring of Witt Vectors} 

The theory of Witt vectors, while mostly elementary, shows up in many areas of mathematics. There are two classical references for the basic theory of Witt vectors: some exercises in Lang's book \emph{Algebra} (cf. \cite[p.~330-332]{Lang}) and a section in Serre's book \emph{Local Fields} (cf. \cite[Chapter 2, Section 6]{Serre:localfields}). Loosely speaking, a Witt vector is an infinite sequence of elements of a commutative ring. Ernst Witt put a ring structure on the set of Witt vectors, creating a ring of characteristic zero from a ring of positive characteristic. For example, the ring of Witt vectors over the finite field $\FF_p$ is $\ZZ_p$, the ring of $p$-adic integers.

In this chapter, we attempt to give a self-contained introduction to the theory of Witt vectors. We will show how Witt vectors arise naturally in studying the structure of complete discrete valuation rings. We begin by reviewing basic results about discrete valuation rings.

\section{Absolute values and valuations} 
Let $K$ be a field. An absolute value $|\cdot|$ on $K$ is a real-valued function $x\mapsto |x|$ on $K$ satisfying the following three properties:

(1) $|x|\geq 0, \forall x\in K$; $|x|=0$ if and only if $x=0$.

(2) $|xy|=|x||y|, \forall x, y\in K$.

(3) $|x+y|\leq |x|+|y|, \forall x,y\in K$.

If instead of (3) the absolute value satisfies the stronger condition

(3)$^\ast$ $|x+y|\leq \text{max}\{|x|,|y|\}, \forall x,y\in K,$

then we say it is non-archimedean.

The absolute value that assigns $1$ to every $x\neq 0$ is called the trivial absolute value.

Let $K$ be a field with a non-trivial absolute value. We can define in the usual manner the notion of a Cauchy sequence: a sequence $(x_n)_{n\geq0}$ of elements in $K$ is a \emph{Cauchy sequence} if for every $\varepsilon>0$, there exists an integer $N$ such that for all $m,n>N$, $|x_m-x_n|<\varepsilon$.

We say $K$ is {\it complete} if every Cauchy sequence in $K$ converges.

It is a standard fact that every field with an absolute value can be made complete:

\begin{prop} \label{prop:completion}
Let $K$ be a field with a non-trivial absolute value $|\cdot|_K$. There exists a pair $(\hat{K},i)$ consisting of a field $\hat{K}$, complete under an absolute value $|\cdot|_{\hat{K}}$, and an embedding $i:K\rightarrow \hat{K}$ such that $|\cdot|_K$ is induced by $|\cdot|_{\hat{K}}$ and $i(K)$ is dense in $\hat{K}$. If $(\hat{K}',i')$ is another such pair, then there exists a unique isomorphism $\varphi:\hat{K}\rightarrow \hat{K}'$ preserving the absolute values, and making the following diagram commutative:

\begin{displaymath}
\xymatrix{
\hat{K} \ar[r]^\varphi & \hat{K}'\\
K\ar[u]^i\ar[ur]_{i'}.
}
\end{displaymath}

\end{prop}

\

\begin{proof} The uniqueness is obvious. The existence is also a well-known fact which we shall briefly recall here.

The Cauchy sequences in $K$ form a ring with addition and multiplication defined component-wise. One defines a null sequence to be a sequence $(x_n)_{n\geq0}$ such that $\underset{n\rightarrow\infty}{\lim} x_n=0$. The null sequences form a maximal ideal in the ring of Cauchy sequences.

We take $\hat{K}$ to be the residue field of Cauchy sequences modulo null sequences. We embed $K$ in $\hat{K}$ diagonally by sending $x\in K$ to the constant sequence $(x)_{n\geq 0}$ modulo the null sequences.

We extend the absolute value $|\cdot|_K$ to $\hat{K}$ by continuity. Let $\xi$ be an element of $\hat{K}$ and let $(x_n)_{n\geq0}$ be a representative of $\xi$. Define $|\xi|_{\hat{K}}=\underset{n\rightarrow \infty}{\lim}|x_n|_K$. It can be verified easily that $|\cdot|_{\hat{K}}$ is an absolute value which is independent of the choice of $(x_n)_{n\geq 0}$, and which induces $|\cdot|_K$ on $K$. Also $K$ is dense in $\hat{K}$.

Finally we verify that $\hat{K}$ is complete. Let $(\xi_n)_{n\geq0}$ be a Cauchy sequence in $\hat{K}$. Since $K$ is dense in $\hat{K}$, we can find for each $n$ an $x_n\in K$ such that $|\xi_n-i(x_n)|_{\hat{K}}<1/n$. By a standard $\frac{1}{3}\varepsilon$ argument, we see that $(x_n)_{n\geq0}$ is a Cauchy sequence in $K$. More precisely, for every $\varepsilon>0$, there exists an integer $N$ such that for all $m,n>N$, we have
\begin{align*}
&|\xi_m-i(x_m)|_{\hat{K}}<\frac{1}{m}<\frac{1}{3}\varepsilon;\\
&|\xi_n-i(x_n)|_{\hat{K}}<\frac{1}{n}<\frac{1}{3}\varepsilon;\\
&|i(x_m)-i(x_m)|_{\hat{K}}=|x_m-x_n|_K<\frac{1}{3}\varepsilon.
\end{align*}
Thus $|\xi_m-\xi_)|_{\hat{K}}\leq |\xi_m-i(x_m)|_{\hat{K}}+|i(x_m)-i(x_n)|_{\hat{K}}+|i(x_n)-\xi_n|_{\hat{K}}<\varepsilon$.

Let $\xi$ be the class of $(x_n)_{n\geq0}$ in $\hat{K}$. One can verify easily that $(\xi_n)$ converges to $\xi$.
\end{proof}
A pair $(\hat{K},i)$ as in the proposition is called a {\it completion} of $K$. The standard pair obtained by the previous construction is usually called {\it the completion} of $K$.

Now we review some basic facts about discrete valuation rings.

Let $R$ be a discrete valuation ring, $\pi$ a uniformizer of $R$, $\fm=\pi R$ the maximal ideal of $R$ and $k=R/\fm$ the residue field. If $x\neq 0$ is an element of $R$, we can write $x=u\pi^n$, with $n\in \NN$ and $u$ a unit of $R$. The integer $n$ does not depend on the choice of $\pi$; it is called the {\it valuation} of $x$ and denoted by $v(x)$.

Let $K$ be the field of fractions of $A$. Let $K^\ast=K\setminus\{0\}$ be the multiplicative group of non-zero elements of  $K$. Now for $x\in K^\ast$ we can write $x=u\pi^n$, with $u$ a unit of $A$ and $n\in \ZZ$. By defining $v(x)=n$ for $x\in K^\ast$ and $v(0)=+\infty$, we get a map $v: K\rightarrow \ZZ\cup\{+\infty\}$ with the following property:

(a) $v: K^\ast\rightarrow \ZZ$ is a surjective homomorphism.

(b) $\forall x, y \in K, v(x+y) \geq \text{min}\{v(x),v(y)\}$.

Note that $R=\{x\in K| v(x)\geq 0\}$ and $\fm=\{x\in K| v(x)>0\}$. Therefore we can begin with a field $K$ and with a map $v$ as above and recover the ring $R$ from them by the above recipe. More precisely:

\begin{prop}
Let $K$ be a field and let $v:K^\ast\rightarrow\ZZ$ be a homomorphism satisfying (a) and (b) above. Then the set $R=\{x\in K| v(x)\geq 0\}$ is a discrete valuation ring having $v$ as its associated valuation.
\end{prop}

Now we fix a real number $a$ between $0$ and $1$, and put
\[\begin{array}[c]{ll}
|x|_v=a^{v(x)}, &x\in K^\ast,\\
|0|_v=0.
\end{array}\]

We can verify easily that $|\cdot|_v$ is a (non-archimedean) absolute value on $K$.

Let $\hat{K}$ be the completion of $K$ for the absolute value $|\cdot|_v$ (the topology of $K$ does not depend on the choice of the number $a$). From previous discussion we know that $\hat{K}$ is a valued field whose absolute value extends that of $K$. If we write the absolute value in the form
\[|x|_{\hat{K}}=a^{\hat{v}(x)}, x\in \hat{K},\]
then the function $\hat{v}$ on $\hat{K}$ is integer valued, and we can verify immediately that it is a discrete valuation on $\hat{K}$, whose valuation ring $\hat{R}$ is the closure of $R$ in $\hat{K}$.

The ideals $\fm^n=\pi^nA$ form a base for the neighborhoods of zero in $K$, hence also in $R$. This means that the topology of $R$ defined by $v$ coincide with the natural topology on $R$ as a local ring. Thus we have
\[\hat{R}=\ilim_{n\in\NN} R/\pi^nR.\]

The element $\pi$ is a uniformizer of $\hat{R}$, and we have $\hat{R}/\pi^n\hat{R}=R/\pi^nR$ for all $n\in\NN$.

We say $R$ is a \emph{complete} discrete valuation ring if the field $K$ is complete (i.e., if $R=\hat{R}$).

\section{Motivation for Witt vectors}

Let us keep the same notation as before: $R$ a discrete valuation ring, $K$ its field of fractions, $\pi$ a uniformizer and $k=R/\pi R$ the residue field.

Let $s: k\rightarrow R$ be a section of the canonical projection $R\twoheadrightarrow k$. We have the following:

\begin{prop} Every element $a\in R$ can be written uniquely as a convergent series
\[a=\sum_{n=0}^\infty s(x_n)\pi^n,\; x_n\in k.\]
Similarly, every element $x\in K$ can be written as
\[x=\sum_{n=n_0}^\infty s(x_n)\pi^n,\; n_0\in \ZZ, x_n\in k.\]
\end{prop}

\begin{proof}
The second assertion follows from the first one because for $x\in K$ we can write $x=\pi^{n_0}a$ with $n_0\in \ZZ$ and $a\in R$.

Now let $a\in R$. Because $s$ is a section, there is $x_0\in k$ such that $a\equiv s(x_0)$ mod $\pi$. So we can write $a=s(x_0)+a_1\pi$. Repeating the argument with $a_1$ we get $x_1\in k$ with $a_1=s(x_1)+a_2\pi$, $a=s(x_0)+s(x_1)\pi+a_2\pi^2$, and so on.

The series $\sum s_n\pi^n$ we get above is obviously unique and converges to $a$.

Conversely, every series of the form $\sum s(x_n)\pi^n$  converges in $R$ because $R$ is complete.

\end{proof}

Therefore we have a bijection of sets
\[ \prod_0^\infty k\rightarrow R, \hspace{10pt} (x_n)\mapsto \sum_{n=n_0}^\infty s(x_n)\pi^n. \]

A natural problem is then to describe the ring structure of $R$ in terms of the coefficients $(x_n)$. This of course depends on the choice of $s$, so a better question is:
Can this be done canonically? For the following we shall again assume that $k$ is a perfect field.

\begin{prop}\label{Teichmuller}
	Let $R$ be a complete noetherian local ring with maximal ideal $\mathfrak{m}$ and perfect residue field $k$ of positive characteristic $p$. Then there exists a unique section $\tau: k\rightarrow R$ which satisfies the following equivalent properties:
	
	(1) $\tau(xy)=\tau(x)\tau(y), \forall x, y\in k,$
	
	(2) $\tau(x)=\underset{n\rightarrow\infty}{\lim} s(x^{p^{-n}})^{p^n}$ for any section $s$ and any $x\in k$.
\end{prop}

The following lemma is important in proving the proposition and later on the $p$-integrality of certain polynomials.

\begin{lemma}\label{congruence}
	For all $n\geq 1$ and $a,b\in R$, if $a\equiv b \mod{\mathfrak{m}}$, then $a^{p^n}\equiv b^{p^n} \mod{\mathfrak{m^{n+1}}}$.
\end{lemma}

\begin{proof}
	First we notice that if $a\equiv b \mod{\mathfrak{m^n}}$, then $a^{p}\equiv b^{p} \mod{\mathfrak{m^{n+1}}}$. To see this, let $c:=a-b\in\mathfrak{m}^n$.  The binomial formula implies that $$a^p-b^p=(c+b)^p-b^p=c^p+\sum_{i=1}^{p-1}c^{p-i}b^i\in (c^p,pc)\subseteq \mathfrak{m}^{n+1}.$$
	The lemma follows by induction on $n$.
\end{proof}

\begin{proof}[Proof of \ref{Teichmuller}]
	The main point is to show that the limit in (2) is well-defined. Consider any section $s: k\rightarrow R$. For all $x\in k$ we have $s(x^{p^{-n}})^p\equiv s(x^{p^{1-n}}) \mod \mathfrak{m}$. Therefore by previous lemma  $s(x^{p^{-n}})^{p^n}\equiv s(x^{p^{1-n}})^{p^{n-1}} \mod \mathfrak{m^n}$. This shows that the sequence in (2) converges. 
	
	If $s': k\rightarrow R$ is another section, we have $s(x^{p^{-n}})\equiv s'(x^{p^{-n}}) \mod{\mathfrak{m}}$. Hence $s(x^{p^{-n}})^{p^n}\equiv s'(x^{p^{-n}})^{p^n} \mod{\mathfrak{m}^{n+1}}$. This shows that the limits defined in (2) coincide. 
	
	Similarly one proves that (2) is equivalent to (1). 
\end{proof}

In order to reconstruct the ring $R$ from $k$, we need to understand the addition and multiplication laws in terms of the arithmetic of $k$. If $\sum\tau(x_n)p^n+\sum\tau(y_n)p^n = \sum\tau(s_n)p^n$, we need to write $s_n$ in terms of the $x_n$ and $y_n$. Once this is done, the multiplication can be deduced from \ref{Teichmuller} (a) and the distributive law:
$$(\sum_m\tau(x_m)p^m)\cdot(\sum_n\tau(y_n)p^n)=\sum_{m,n}\tau(x_my_n)p^{m+n}.$$

To calculate the $s_n$, we proceed inductively.

For $s_0$ note that $$\tau(x_0)+\tau(y_0)\equiv \tau(s_0)\mod{p}.$$ Since $x=\tau(x) \mod{p}$, we have
$$s_0=x_0+y_0.$$

Similarly from $$\tau(x_0)+\tau(x_1)p+\tau(y_0)+\tau(y_1)p\equiv\tau(s_0)+\tau(s_1)p \mod{p^2}$$ we get that $$\tau(s_1)p\equiv \tau(x_0)+\tau(y_0)-\tau(s_0)+(\tau(x_1)+\tau(y_1))p \mod{p^2}.$$
Although we know $\tau(x_0)+\tau(y_0)-\tau(s_0)\equiv 0 \mod{p}$, we have no idea what the residue is mod ${p^2}$. The trick to calculate $s_1$ is to use the fact that $k$ is perfect. Let $x^{1/p}$ be the unique $p$-th root of $x$ in $k$. Since $x_0+y_0=s_0$, we have $x_0^{1/p}+y_0^{1/p}=s_0^{1/p}$. By Lemma \ref{congruence} and Proposition \ref{Teichmuller} (a) we have
$$\tau(s_0)=\tau(s_0^{1/p})^p=\tau(x^{1/p}+y^{1/p})^p\equiv (\tau(x_0^{1/p})+\tau(y_0^{1/p}))^p \mod{p^2}.$$ 
Therefore
$$\tau(s_1)p\equiv \tau(x_0^{1/p})^p+\tau(y_0^{1/p})^p-(\tau(x_0^{1/p})+\tau(y_0^{1/p}))^p+(\tau(x_1)+\tau(y_1))p \mod{p^2}.$$
Using binomial expansion and dividing by $p$, we obtain
$$\tau(s_1)\equiv \tau(x_1)+\tau(y_1)-\frac{1}{p}\sum_{i=1}^{p-1}\binom{p}{i}\tau(x_0^{i/p})\tau(y_0^{(p-i)/p}) \mod{p},$$ and therefore
$$s_1=x_1+y_1-\frac{1}{p}\sum_{i=1}^{p-1}\binom{p}{i}x_0^{i/p}y_0^{(p-i)/p}.$$

The above calculation does not depend on $k$. Let $X_0, X_1, Y_0, Y_1, S_0, S_1$ be variables and define $\omega_{0}(X_0) = X_{0}$, $\omega_{1}(X_0,X_1) = X_{0}^{p} + pX_{1}$. Solving the polynomial equations
\begin{align*}
\omega_{0}(S_0) &=\omega_{0}(X_0)+\omega_{0}(Y_0)\\
\omega_{1}(S_0,S_1) &=\omega_{1}(X_0,X_1)+\omega_{1}(Y_0, Y_1)
\end{align*}
yields
\begin{align*}
S_0 &=X_0+Y_0,\\
S_1 &=X_1+Y_1-\frac{1}{p}\sum_{i=1}^{p-1}\binom{p}{i}X_0^{i}Y_0^{p-i}.
\end{align*}

In particular, $S_0\in\mathbb{Z}[X_0,Y_0]$, $S_1\in\mathbb{Z}[X_0,X_1,Y_0,Y_1]$, and by substituting $\tau(x_0^{1/p})$ for $X_0$ and $\tau(y_0^{1/p})$ for $Y_0$, we have $s_0=S_0(x_0,y_0)$ and $s_1=S_1(x_0^{1/p},x_1, y_0^{1/p}, y_1)$.

By studying the patterns, Witt proposed the following:
\begin{defn}

Let $p$ be a prime number. Let $(X_{0}, X_{1}, \dots, X_{n}, \dots)$ be a sequence of variables. The \emph{Witt polynomials} are defined as follows:
\begin{align*}
\omega_{0} &= X_{0}, \\
\omega_{1} &= X_{0}^{p} + pX_{1}, \\
\omega_{2} &= X_{0}^{p^{2}} + pX_{1}^{p} + p^{2}X_{2}, \\
\vdots   & \\
\omega_{n} &= X_{0}^{p^{n}} + pX_{1}^{p^{n-1}} + \cdots + p^{n-1}X_{n-1}^{p} + p^{n}X_{n} = \sum_{i=0}^{n} p^{i}X_{i}^{p^{n-i}}.
\end{align*}
\end{defn}

Let $(Y_{0}, Y_{1}, \dots, Y_{n}, \dots)$ be another sequence of variables. Inductively find $S_n$'s that solve the polynomial equations
$$\omega_n(S_0,S_1,\dots,S_n)=\omega_n(X_0,X_1,\dots,X_n)+\omega_n(Y_0,Y_1,\dots,Y_n). $$ Clearly $S_n$ are polynomials with rational coefficients. Witt showed in \cite{Witt} that in fact $S_n\in\mathbb{Z}[X_0,X_1,\dots, X_n,Y_0, Y_1,\dots,Y_n ]$.

This result actually turns everything around and defines a natural ring structure on $\prod_{i=0}^{\infty}k$ without the prior presence of $R$. This produces a ring of characteristic zero from a field of characteristic $p$. 

This construction is related to the fact that, although the additive group of the ring of the power series $k[[t]]$ has characteristic $p$, its multiplicative group of $1$-units $1+tk[[t]]$ is torsion free.

\section{The Artin-Hasse exponential}

Recall that the M{\"o}bius function is the function $\mu(n)$ defined on positive integers as

$$\mu(n)=\left\lbrace 
\begin{array}{cc}
	(-1)^{\textrm{the number of prime factors of } n}, & n \textrm{ is square free}; \\ 
	0 & \textrm{otherwise}.
\end{array}\right.$$

\begin{lemma}
	For all $n\geq 1$,$$\underset{d|n}{\sum}\mu(d)=0.$$
\end{lemma}
\begin{proof}
	Let $n=p_1p_2\cdots p_k$ be square free. It's easy to see that $$\underset{d|n}{\sum}\mu(d)=\sum_{i=0}^{k}\binom{k}{i}(-1)^i=(1-1)^k=0.$$
	In general, if $m=p_1^{e_1}p_2^{e_2}\cdots p_k^{e_k}$, then $\underset{d|m}{\sum}\mu(d)=\underset{d|n}{\sum}\mu(d)=0$.
\end{proof}

\begin{lemma}
In $\QQ[[t]]$ we have the equality
\[
\exp(-t)=\prod_{n\geq 1}(1-t^n)^\frac{\mu(n)}{n}.
\]
\end{lemma}

\begin{proof}
	By taking the logarithm we have 
	\begin{align*}
	\log(\prod_{n\geq 1}(1-t^n)^\frac{\mu(n)}{n}) &=\sum_{n\geq1}\frac{\mu(n)}{n}\log(1-t^n) \\
	&=\sum_{n\geq1}\frac{\mu(n)}{n}\sum_{m\geq1}(-\frac{t^{nm}}{m})\\
	&=-\sum_{g\geq1}(\sum_{n|d}\mu(n))\frac{t^d}{d}\\
	&=-t.
	\end{align*}
	by using the fact that $\log(1-t)=-\underset{m\geq1}{\sum}\frac{t^{m}}{m}$ and by letting $d=nm$.
\end{proof}

Recall that for any nonzero rational number $r$, the primes in the denominator of $\binom{r}{i}$ are limited to the primes in the denominator of $r$ by $p$-adic continuity of the binomial coefficient polynomials. Therefore the denominators of the coefficients in $(1-t^n)^\frac{\mu(n)}{n}$ come from factors of $n$. The following definition separates the $p$-part of those denominators from the non-$p$-part. 

\begin{defn} Let $F(t):=\underset{p\nmid n }{\prod}(1-t^n)^\frac{\mu(n)}{n}\in \QQ[[t]]$.
\end{defn}

Clearly $F(t)\in 1+t\mathbb{Z}_{(p)}[[t]]$. The interesting fact is the following lemma, which shows that although $F(t)$ is a power series without $p$ in the denominators, its logarithm has only powers of $p$ in the denominators.

\begin{lemma}\label{AHE}
	We have $F(t)=\exp(-\underset{m\geq0}{\sum}\frac{t^{p^m}}{p^m}).$
\end{lemma}

\begin{proof}
	By taking the logarithm again we have 
		\begin{align*}
		\log F(t) &=\sum_{p\nmid n}\frac{\mu(n)}{n}\log(1-t^n) \\
		&=-t-\sum_{p|n}\frac{\mu(n)}{n}\log(1-t^n)\\
		&\overset{\mathclap{\strut n=mp}}=-t-\sum_m\frac{\mu(mp)}{mp}\log (1-t^{mp})\\
		&=-t+\frac{1}{p}\sum_{p\nmid m}\frac{\mu(m)}{m}\log (1-t^{mp})\\
		&=-t+\frac{1}{p}\log F(t^p).\\
		\end{align*}
	Here we used the fact that if $p|m$, then $\mu(mp)=0$. The lemma follows by iterating this formula.	
	
\end{proof}

\begin{lemma}\label{approximation}
	There exist unique polynomials $\psi_n\in\mathbb{Z}[x,y]$ such that
	$$F(xt)\cdot F(yt)=\prod_{n\geq 0}F(\psi_n(x,y)t^{p^n}).$$
\end{lemma}

\begin{proof}
	Since $F(t)=1-t+\cdots$ is invertible and has coefficients in $\ZZ_{(p)}$, by induction on $d$ we can find unique polynomials $\lambda_d\in\ZZ_{(p)}[x,y]$ such that 
	$$F(xt)\cdot F(yt)=\prod_{d\geq 1}F(\lambda_d(x,y)t^{d}).$$
	
	The goal is to show that $\lambda_d(x,y)$ vanishes for all $d$ which is not a power of $p$.
	Taking logarithm on both sides and apply Lemma \ref{AHE}, we get
	$$-\sum_{m\geq0}(x^{p^m}+y^{p^m})\frac{t^{p^m}}{p^m}=-\sum_{d\geq 1}\sum_{m\geq 0}\lambda_d(x,y)^{p^m}\frac{t^{dp^m}}{p^m}.$$
	
	Suppose $\lambda_d\neq 0$ for some $d$ which is not a power of $p$. Let $d_0$ be the smallest such $d$. Comparing the coefficients for $t^{d_0}$ we get that $\lambda_{d_0}= 0$, yielding a contradiction. Therefore $\lambda_d=0$ whenever $d$ is not a power of $p$. Let $\psi_n(x,y)=\lambda_{p^n}(x,y)$, we get that
	$$F(xt)\cdot F(yt)=\prod_{n\geq 0}F(\psi_n(x,y)t^{p^n}).$$
	
	Moreover, each $\psi_n(x,y)$ is given recursively as a polynomial over $\ZZ[\frac{1}{p}]$ in $x$, $y$, and $\psi_{n'}(x,y)$ for $n'<n$. By induction on $n$ we conclude that $\psi_n(x,y)\in\ZZ[\frac{1}{p}][x,y]$. Since we also know $\psi_n(x,y)\in\ZZ_{(p)}[x,y]$, we conclude that $\psi_n(x,y)\in\ZZ[x,y]$, which proves the lemma.
	
\end{proof}

Now for any ring $\Gamma$, let $\Lambda_\Gamma$ be the affine $\Gamma$-group which associates with an $\Gamma$-algebra $R$ the multiplicative group $1+tR[[t]]$ of formal power series in $R$ with constant term $1$. Let $W_\Gamma=\underset{n\geq 0}{\prod}\mathbb{A}^1_\Gamma=\spec \Gamma[X_0, X_1, \dots]$.

\begin{defn}
Let $\underline{x}=(x_0,x_1,\dots)$. The Artin-Hasse exponential is defined as the morphism
\begin{align*}
E:W_{\ZZ_{(p)}} &\rightarrow \Lambda_{\ZZ_{(p)}}, \\
\underline{x} &\mapsto E(\underline{x},t):=\prod_{n\geq 0}F(x_nt^{p^n}).
\end{align*}
\end{defn}

\begin{prop}
	There exist unique polynomials $s_n\in\ZZ[x_0,\dots,x_n,y_0,\dots,y_n]$ such that $E(\underline{x},t)E(\underline{y},t)=E(\underline{s}(\underline{x},\underline{y}),t)$ with $\underline{s}(\underline{x},\underline{y})=(s_0(x_0,y_0), s_1(x_0,x_1,y_0,y_1),\dots)$. Moreover, the morphism $\underline{s}:W_\ZZ \times W_\ZZ \rightarrow W_\ZZ$ makes $W_\ZZ$ a commutative group scheme over $\ZZ$. 
\end{prop}

\begin{proof}
	The first part is proved by successive approximation using Lemma \ref{approximation}. Thus we just need to find the identity and the inverse morphism. The identity is given by $\ul{0}=(0,0,\dots)$. To find the inverse, first it is not hard to check that 
	$$F(t)^{-1}=\left\{ \begin{array}{ll}
	F(-t) & \textrm{if } p\neq 2, \\ 
	\prod_{n\geq 0} F(-t^{p^n}) & \textrm{if } p=2.
	\end{array} \right.
	$$
	Again by successive approximation we can find unique morphism $\ul{i}:W_\ZZ\rightarrow W_\ZZ$. The group axioms follow from the facts that over $\ZZ_{(p)}$, the morphism $E: W_{\ZZ_{(p)}}\rightarrow \Lambda_{\ZZ_{(p)}}$ is a closed embedding and that $E(\ul{0},t)=1$ and that $E(\ul{x},t)^{-1}=E(\ul{i}(\ul{x}),t)$.
\end{proof}

\section{The ring of Witt vectors over $\ZZ$}

Recall that $E(\underline{x},t)=\prod_{n\geq 0}F(x_nt^{p^n})$. Using Lemma \ref{AHE}, we can write
\begin{align*}
E(\underline{x},t)&=\prod_{n\geq 0}F(x_nt^{p^n})\\
&=\prod_{n\geq 0}\exp\left( -\sum_{m\geq0}\frac{(x_nt^{p^n})^{p^m}}{p^m}\right) \\
&=\exp\left(-\sum_{n,m\geq 0}p^nx_n^{p^m}\frac{t^{p^{n+m}}}{p^{n+m}} \right) \\
&=\exp\left(-\sum_{l\geq n \geq 0}p^nx_n^{p^{l-n}}\frac{t^{p^{l}}}{p^{l}} \right).
\end{align*}

Let $\Phi_l(\underline{x})=x_0^{p^l}+px_1^{p^{l-1}}+\cdots +p^lx_l=\sum_{n=0}^l p^nx_n^{p^{l-n}}$, we have
$$E(\underline{x},t)=\exp(-\sum_{l\geq 0}\Phi_l(\underline{x})\frac{t^{p^l}}{p^l}).$$
Apply logarithm to $E(\underline{x},t)E(\underline{y},t)=E(\underline{s}(\underline{x},\underline{y}),t)$ we get $$\log E(\underline{x},t)+\log E(\underline{y},t)= \log E(\underline{s}(\underline{x},\underline{y},t),$$
which is equivalent to 
$$-\sum_{l\geq 0}\Phi_l(\underline{x})\frac{t^{p^l}}{p^l}-\sum_{l\geq 0}\Phi_l(\underline{y})\frac{t^{p^l}}{p^l}=-\sum_{l\geq 0}\Phi_l(\underline{s}(\underline{x},\underline{y})))\frac{t^{p^l}}{p^l}.$$
Therefore we have $$\Phi_l(\underline{x})+\Phi_l(\underline{y})=\Phi_l(\underline{s}(\underline{x},\underline{y})).$$
This proves the following

\begin{prop}
The above group law on $W_\ZZ$ is the unique one such that for each $l$, $\Phi_l:W_\ZZ\rightarrow (\mathbb{A}_\ZZ^1,+)=\boldsymbol{G_a}_\ZZ$ is a homomorphism.
\end{prop}

\begin{rmk}\label{isom}
	Let $R$ be a ring. An element $\underline{x}=(x_0,x_1,\dots)\in W(R)$ is called a \emph{Witt vector}. The $x_0, x_1, \dots$ are called \emph{components} of $\underline{x}$, and the $\Phi_0(\underline{x}), \Phi_1(\underline{x}),\dots$ are called \emph{phantom components} of $\underline{x}$. Note that $(\Phi_l\otimes_\ZZ\ZZ[\frac{1}{p}])_{l\geq 0}$ defines an isomorphism
	$$W_{\ZZ[\frac{1}{p}]}\rightarrow \prod_{l\geq 0}\mathbb{A}^1_{\ZZ[\frac{1}{p}]},$$
	$$\underline{x}\mapsto (\Phi_l(\underline{x}))_{l\geq 0}.$$
	But this isomorphism reduces to $\Phi_l(\underline{x})\equiv x_0^{p^l} \mod p$.
\end{rmk}

\begin{defn}
	The Teichm\"{u}ller lift is the morphism $$\tau: \mathbb{A}^1_\ZZ \to W_\ZZ, x\mapsto \tau(x) = (x, 0, \dots).$$
\end{defn}

\begin{prop}
	The Teichm\"{u}ller lift is multiplicative.
\end{prop}	

\begin{proof}
	Over $\ZZ[\frac{1}{p}]$ we have $\Phi_l(\tau(x))=x^{p^l}$ and $E(\tau(x),t)=F(xt)$, which is clearly multiplicative.
\end{proof}

\begin{thm}
	There exists a unique ring-structure on the $\ZZ$-group $W_\ZZ$ such that each $\Phi_l:  W_\ZZ\rightarrow \mathbb{A}^1_\ZZ$ is a ring homomorphism.
\end{thm}	

\begin{proof}
	The isomorphism in Remark \ref{isom} shows that the theorem is true over $\ZZ[\frac{1}{p}]$. Therefore we need to show that the addition and multiplication, as well as the identities and the additive inverse, are morphisms defined over $\ZZ$. It's easy to check that $\tau(0)=(0,0,\dots)$ and $\tau(1)=(1,0,\dots)$ are the additive and multiplicative identities. Others follow from the lemma below.
\end{proof}	

\begin{lemma}
	For every morphism $u: \mathbb{A}^1_\ZZ \times \mathbb{A}^1_\ZZ \rightarrow \mathbb{A}^1_\ZZ$ there exists a unique morphism $\underline{v}: W_\ZZ \times W_\ZZ \rightarrow W_\ZZ$ such that for all $l\geq 0$, the following diagram commutes:
	\begin{displaymath}
	\xymatrix
	{
		W_\ZZ\times W_\ZZ \ar[r]^{\Phi_l\times \Phi_l}\ar[d]_{\underline{v}}  & \mathbb{A}^1_\ZZ\times \mathbb{A}^1_\ZZ \ar[d]^{u} \\
		W_\ZZ \ar[r]_{\Phi_l} & \mathbb{A}^1_\ZZ.
	}
	\end{displaymath}
	 
\end{lemma}	

\begin{proof}
	By the isomorphism in Remark \ref{isom} there exist unique $\underline{v}=(v_0,v_1,\dots)$ with $v_n\in \ZZ[\frac{1}{p}][x_0,\dots,x_n,y_0,\dots,y_n]$ satisfying the desired conditions. We show by induction on $n$ that $v_n$ have integer coefficients.
	
	Let $A=\ZZ[x_0,x_1,\dots,y_0,y_1,\dots]$. Since $\Phi_0(\underline{x})=x_0$, we have $v_0=u(x_0,y_0)\in A$. Now fix $n\geq 0$ and assume that for all $i\leq n$, $v_i\in A$. For any $\underline{x}=(x_0,x_1,\dots)$, let $\underline{x}^{(p)}=(x_0^p,x_1^p,\dots)$. One can check easily from the definition that $$\Phi_{n+1}(\underline{x})=\Phi(\underline{x}^{(p)})+p^{n+1}x_{n+1}.$$ Using this and the property of $\underline{v}$ we have
	
	\begin{align*}
		\Phi_n(\underline{v}(\underline{x},\underline{y})^{(p)})+p^{n+1}v_{n+1}(\underline{x},\underline{y}) &=\Phi_{n+1}(\underline{v}(\underline{x},\underline{y}))\\
		&= u(\Phi_{n+1}(\underline{x}), \Phi_{n+1}(\underline{y}))\\
		&= u(\Phi_{n}(\underline{x}^{(p)})+p^{n+1}x_{n+1}, \Phi_{n}(\underline{y}^{(p)})+p^{n+1}y_{n+1}).
	\end{align*}	
	
	Since $\Phi_n(\underline{v}(\underline{x},\underline{y})^{(p)})$ and $u(\Phi_{n}(\underline{x}^{(p)})+p^{n+1}x_{n+1}, \Phi_{n}(\underline{y}^{(p)})+p^{n+1}y_{n+1})$ are in $A$ by assumption, we have $p^{n+1}v_{n+1}(\underline{x},\underline{y})\in A$. Moreover, we have 	
		\begin{align*}
			p^{n+1}v_{n+1} &\equiv u(\Phi_{n}(\underline{x}^{(p)}),\Phi_{n}(\underline{y}^{(p)}))-\Phi_n(\underline{v}^{(p)}) \mod p^{n+1}A\\
			&=\Phi_n(\underline{v}(\underline{x}^{(p)},\underline{y}^{(p)}))-\Phi_n(\underline{v}^{(p)}).
		\end{align*}
	Also, $v_i\in A$ implies that $v_i(\underline{x}^{(p)},\underline{y}^{(p)})\equiv v_i(\underline{x},\underline{y})^{(p)} \mod  pA$.
	Therefore we have
	$$v_i(\underline{x}^{(p)},\underline{y}^{(p)})^{p^{n-i}}\equiv (v_i(\underline{x},\underline{y})^{(p)})^{p^{n-i}} \mod  p^{n-i+1}A,$$
	hence
	$$p^iv_i(\underline{x}^{(p)},\underline{y}^{(p)})^{p^{n-i}}\equiv p^i(v_i(\underline{x},\underline{y})^{(p)})^{p^{n-i}} \mod  p^{n+1}A,$$	
	and consequently
	$$\Phi_n(\underline{v}(\underline{x}^{(p)},\underline{y}^{(p)}))\equiv \Phi_n(\underline{v}^{(p)})\mod p^{n+1}A.$$
	Therefore we have $p^{n+1}v_{n+1}\in p^{n+1}A$ and hence $v_{n+1}\in A$.
	
\end{proof}

\begin{exam}
	Let $\underline{s}=(s_0,s_1,\dots)$ and $\underline{p}=(p_0,p_1,\dots)$ be the morphisms that correspond to the addition and multiplication respectively. We can compute the following:
	
	\begin{align*}
		s_0(\underline{x}, \underline{y}) &= x_0 + y_0, \\
		s_1(\underline{x}, \underline{y}) &= x_1 + y_1 + \frac{x_0^p+y_0^p-(x_0+y_0)^p}{p}, \\
		s_2(\underline{x}, \underline{y}) &= x_2 + y_2 + \frac{x_1^p+y_1^p-s_1^p}{p} + \frac{x_0^{p^2}+y_0^{p^2}-(x_0+y_0)^{p^2}}{p^2},\\
		&\vdots  \\
		p_{0}(\underline{x}, \underline{y}) &= x_{0}y_{0},\\
		p_{1}(\underline{x}, \underline{y}) &= y_{0}^{p}x_{1} + y_{1}x_{0}^{p} + px_{1}y_{1},\\
		p_{2}(\underline{x}, \underline{y}) &= \frac{1}{p}((x_{0}^{p^{2}}y_{1}^{p} + x_{1}^{p}y_{0}^{p^{2}}) - (y_{0}^{p}x_{1}+y_{1}x_{0}^{p}+ p x_{1}y_{1})^p)\\
		&+(x_{0}^{p^{2}}y_{2} + x_{1}^{p}y_{1}^{p} + x_2y_{0}^{p^{2}})+p(x_{1}^{p}y_{2}+x_{2}y_{1}^{p}) + p^{2}x_{2}y_{2},\\
		&\vdots
	\end{align*}

As one can see, the formulas are becoming complicated very quickly but all have integer coefficients.	
	
\end{exam}	

Finally, let $n\geq 1$, we introduce \emph{Witt vectors of finite length $n$}. For this recall that the $i$-th components of $\underline{x}+\underline{y}$, $\underline{x}\cdot\underline{y}$ and $-\underline{x}$ depend only on the first $i$ components of $\underline{x}$ and $\underline{y}$. Thus the same formulas define a ring structure on $W_{n,\ZZ}(A):=\prod_{i=0}^{n-1}\mathbb{A}^1_A$ for any ring $A$, such that the truncation map
$$W(A)\rightarrow W_n(A),\quad \underline{x}\mapsto (x_0,x_1,\dots,x_{n-1})$$
is a ring homomorphism.

\section{The ring of Witt vectors over $k$} 

Let $k$ be a field of characteristic $p>0$. Let $W_k=W_{\FF_p}\times_{\spec \FF_p} \spec k$. Let $F$ and $V$ be the Frobenius and Verschibung for the additive group $W_k$. They are group endomorphisms satisfying $F\circ V=V\circ F=p\cdot \id$. The following proposition collects some of their properties.

\begin{prop}
	
	\begin{enumerate}[(i).]
		\item $F((x_0,x_1,\dots))=(x_0^p,x_1^p,\dots)$.
		\item $V((x_0,x_1,\dots))=(0, x_0,x_1,\dots)$.
		\item $p((x_0,x_1,\dots))=(0,x_0^p,x_1^p,\dots)$.
		\item $F(\ul{x}+\ul{y})=F(\ul{x})+F(\ul{y})$.
		\item $F(\ul{x}\cdot\ul{y})=F(\ul{x})\cdot F(\ul{y})$.
		\item $\ul{x}\cdot V(\ul{y})=V(F(\ul{x})\cdot \ul{y})$.
		\item $E(\ul{x}\cdot V(\ul{y}),t)=E(F(\ul{x})\cdot \ul{y}, t^P)$.
	\end{enumerate}
	
\end{prop}	

\begin{proof}
\cite[p.~45]{Pink}
\end{proof}	
\begin{thm}
	The ring of Witt vectors $W(k)$ is a complete discrete valuation ring with uniformizer $p$ and residue field $k$.
\end{thm}	

\begin{proof}
	Since $k$ is perfect, we have $p^nW(k)=V^n(W(k))$ for all $n\geq 1$. Thus $W(k)/p^nW(k)\cong W_n(k)$ and $W(k)/pW(k)\cong W_1(k)=k$. Using this, by induction on $n$ one shows that $W_n(k)$ is a $W(k)$-module of length $n$. Since $W(k)\cong \plim_nW_n(k)$, the theorem follows.
\end{proof}

\begin{thm}[Witt] Let $R$ be a complete noetherian local ring with residue field $k$.
	
	(i) There is a unique ring homomorphism $u: W(k)\rightarrow R$ such that the following diagram commutes:
	\begin{displaymath}
	\xymatrix
	{
		W(k) \ar[rr]^{u}\ar[rd]_{} & & R \ar[ld]^{} \\
		& k  &.
	}
	\end{displaymath}
	
	(ii) If $R$ is a complete discrete valuation ring with uniformizer $p$, then $u$ is an isomorphism.
\end{thm}

\begin{proof}
\cite[p.~47]{Pink}
\end{proof}

\begin{exam}
$W(\FF_p)=\ZZ_p$, $W_n(\FF_p)=\ZZ/p^n\ZZ$.
\end{exam}

We conclude this section by stating without proof the classification theorem of complete discrete valuation rings with perfect residue field $k$ in \cite{Serre:localfields}. Let $R$ be a complete discrete valuation ring with residue field $k$. Suppose $k$ is perfect. In the case that $R$ and $k$ have the same characteristic, we have:

\begin{thm}[Equal Characteristic Case] Let $R$ be a complete discrete valuation ring with residue field $k$. Suppose that $R$ and $k$ have the same characteristic and that $k$ is perfect. Then $R$ is isomorphic to $k[[t]]$, the ring of formal power series with coefficients in $k$.
\end{thm}
\begin{proof} \cite[p.~33]{Serre:localfields} \end{proof}

Now suppose the characteristic of $R$ is not equal to that of $k$, i.e. $R$ has characteristic zero  and $k$ has characteristic $p>0$. Let $v$ be the discrete valuation associated to $R$. We can identify $\ZZ$ as a sub-ring of $R$. Since $p$ goes to zero in $k$, $p$ is in the maximal ideal $\fm$ of $R$ and $v(p)>0$. The integer $e=v(p)$ is called the {\it absolute ramification index} of $R$. We say $R$ is absolutely unramified if $e=1$, i.e., if $p$ is a uniformizer of $R$. For such rings we have the following structure theorem:

\begin{thm}[Unequal Characteristic Case] Let $k$ be a perfect field of characteristic $p>0$. There exists a complete discrete valuation ring which is absolutely unramified and has $k$ as its residue field. Such a ring is unique up to unique isomorphism.
	
\end{thm}

\begin{proof}\cite[p.~36]{Serre:localfields} 
\end{proof}	
This ring is none other than $W(k)$.

Finally in the ramified case, we have:
\begin{thm}
	Let $R$ be a discrete valuation ring of characteristic unequal to that of its residue field $k$. Let $e$ be its absolute ramification index. Then there exists a unique homomorphism from $W(k)$ to $R$ which makes the following diagram commutative:
	\begin{displaymath}
	\xymatrix
	{
		W(k) \ar[rr]^{}\ar[rd]_{} & & R \ar[ld]^{} \\
		& k  &.
	}
	\end{displaymath}
	This homomorphism is injective, and $R$ is a free $W(k)$-module of rank $e$.
\end{thm}
\begin{proof}\cite[p.~37]{Serre:localfields}
\end{proof}

\chapter{Dieudonn\'{e} Theory}

In this chapter we give a (mostly) self-contained account of the fundamental results of Dieudonn\'e Theory. We explicitly construct the (contravariant) Dieudonn\'e functor, starting with finite group
schemes of local-local type first, then expanding to include all finite group
schemes of $p$-power order. Many of the proofs use ideas found in \cite{Pink}.



\section{Finite Witt group schemes}

Let $k$ be a field. From now on we abbreviate $W=W_k$, restoring $k$ only when the dependence on the field $k$ is discussed.

For any integer $n\geq 1$, let $W_n=W/V^nW$ be the additive group scheme of Witt vectors of length $n$ over $k$. Truncation induces natural epimorphisms $r: W_{n+1}\twoheadrightarrow W_n$, and Verschiebung induces natural monomorphisms $v: W_n\hookrightarrow W_{n+1}$, such that $r\circ v=v\circ r=V$. For any $n,n'>1$ they induce a short exact sequence
$$0\rightarrow W_{n'}\xrightarrow{v^n}W_{n+n'}\xrightarrow{r^{n'}}W_n\rightarrow 0.$$
Together with the natural isomorphism $W_1\cong \boldsymbol{G_a}$, these exact sequences describes $W_n$ as a successive extension of $n$ copies of $\boldsymbol{G_a}$.

For any integers $n,m\geq 1$, let $W_n^m$ be the kernel of $F^m$ on $W_n$. As above, truncation induces natural epimorphisms $r: W_{n+1}^m\twoheadrightarrow W_n^m$, and Verschiebung induces natural monomorphisms $v: W_n^m\hookrightarrow W_{n+1}^m$, such that $r\circ v=v\circ r=V$. Similarly, the inclusion induces natural monomorphisms $i: W_n^m\hookrightarrow W_n^{m+1}$, and Frobenius induces natural epimorphisms $f: W_n^{m+1}\twoheadrightarrow W_n^m$, such that $i\circ f=f\circ i=F$. For any $n,n',m,m'\geq 1$ they induce short exact sequences 
$$0\rightarrow W_{n'}^m\xrightarrow{v^n}W_{n+n'}^m\xrightarrow{r^{n'}}W_n^m\rightarrow 0.$$
$$0\rightarrow W_n^m\xrightarrow{i^{m'}}W_n^{m+m'}\xrightarrow{f^m}W_n^{m'}\rightarrow 0.$$
Together with the natural isomorphism $W_1^1\cong \boldsymbol{\alpha}_p$, these exact sequences describe $W_n^m$ as a successive extension of $nm$ copies of $\boldsymbol{\alpha}_p$. 

For later use we note the following fact: 

\begin{lemma}\label{factor}
	Let $G$ be a finite commutative group scheme with $F^m_G=0$ and $V^n_G=0$. Then any homomorphism $\phi: G\rightarrow W_{n'}^{m'}$ with $m'\geq m$ and $n'\geq n$ factors uniquely through the embedding $i^{m'-m}\circ v^{n'-n}: W_n^m\rightarrow W_{n'}^{m'}$.
\end{lemma} 
\begin{proof}
	By the functoriality of Frobenius we have the following commutative diagram
	
	$$
	\xymatrix{
		G\ar[r]^{F_G^m} \ar[d]^\varphi & G^{(p^m)} \ar[d]^{\varphi^{(p^m)}} \\
		W_{n'}^{m'}\ar[r]^{F^m_{W_{n'}^{m'}}} & W_{n'}^{m'}
	}.
	$$
	Therefore the assumption implies that $F_{W_{n'}^{m'}}^m\circ \phi=\phi^{(p^m)}\circ F_g^m=0$. Thus $\phi$ factors through the kernel of $F^m$ on $W_{n'}^{m'}$, which is the image of $i^{m'-m}: W_{n'}^m\hookrightarrow W_{n'}^{m'}$. Similar argument with $V_G^n$ in place of $F_g^m$ shows the rest.
	
\end{proof}

\section{The Dieudonn\'e functor in the local-local case}

We will show that all commutative finite group schemes of local-local type can be constructed from the Witt group schemes $W_n^m$. The main step towards this is the following results on extensions:

\begin{prop}\label{ext}
	For any short exact sequence $0\rightarrow W_n^m\rightarrow G\rightarrow \boldsymbol{\alpha}_p\rightarrow 0$ there exists a homomorphism $\varphi$ making the following diagram commutative:
	$$
	\xymatrix{
		0\ar[r] &W_n^m\ar[r]\ar@{_{(}->}[d] &G\ar[r]\ar[dl]^\varphi &\boldsymbol{\alpha}_p\ar[r] &0 \\
		&W_{n+1}^{m+1}
	}.
	$$
\end{prop}

\begin{proof}
	\cite[~49-53]{Pink}
\end{proof}

\begin{prop}\label{embed}
	Every commutative finite group scheme $G$ of local-local type can be embedded into $(W_n^m)^{\oplus r}$ for some $n,m$ and $r$.
\end{prop}

\begin{proof}
	To prove this by induction on $|G|$, we consider a short exact sequence $0\rightarrow G'\rightarrow G\rightarrow \boldsymbol{\alpha}_p\rightarrow 0$ and assume that there exists an embedding $\varphi=(\varphi_1,\dots,\varphi_r): G'\hookrightarrow (W_n^m)^{\oplus r}$. For $1\leq i\leq r$ define $G_i$ such that the upper left square in the following commutative diagram with exact rows is a pushout:

	$$\xymatrix{
		0\ar[r] &G' \ar[r]\ar[d]^{\varphi_i} &G \ar[r]\ar[d] &\boldsymbol{\alpha}_p \ar[r]\ar@{=}[d] &0\\
		0\ar[r] &W_n^m \ar[r]\ar[d]^{i\circ v} &G_i \ar[r]\ar@{.>}[dl] &\boldsymbol{\alpha}_p \ar[r] &0\\
		&W_{n+1}^{m+1}
	}.$$
	The dashed arrow, which exists by Proposition \ref{ext}, determines an extension of the composite embedding $i\circ v\circ\phi: G'\hookrightarrow (W_{n+1}^{m+1})^{\oplus r}$ to a homomorphism $G\rightarrow (W_{n+1}^{m+1})^{\oplus r}$. The direct sum of this with the composite homomorphism $G\twoheadrightarrow \boldsymbol{\alpha}_p=W_1^1\hookrightarrow W_{n+1}^{m+1}$ is an embedding $G\hookrightarrow (W_{n+1}^{m+1})^{\oplus r+1}$.
	
\end{proof}

\begin{prop}\label{copres}
	Every commutative finite group scheme $G$ with $F_G^m=0$ and $V_G^n=0$ fits in an exact sequence 
	$$0\rightarrow G\rightarrow (W_n^m)^{\oplus r}\rightarrow (W_n^m)^{\oplus s}$$
	for some $r,s\geq 1$.
\end{prop}

\begin{proof}
	By Proposition \ref{embed} there exists an embedding $G\hookrightarrow (W_{n'}^{m'})^{\oplus r}$ for some $n', m'$, and $r$. After composing it in each factor with the embedding $i\circ v: W_{n'}^{m'}\hookrightarrow W_{n'+1}^{m'+1}$, if necessary, we may assume that $n'\geq n$ and $m'\geq m$. Then Lemma \ref{factor} implies that the embedding factors through a homomorphism $G\rightarrow (W_n^m)^{\oplus r}$, which is again an embedding. Let $H$ denote its cokernel. Since $F^m=0$ and $V^n=0$ on $W_n^m$, the same is true on $(W_n^m)^{\oplus r}$ and hence on $H$. Repeating the first part of the proof with $H$ in the place of $G$, we therefore find an embedding $H\hookrightarrow (W_n^m)^{\oplus s}$ for some $s$. The proposition follows.
\end{proof}

Now we describe the Dieudonn{\'e} functor in the local-local case.

Note that $W_n$ form a directed system under the morphisms $v:
W_n\rightarrow W_{n+1}$, which takes $(a_0, \dots, a_{n-1})$ to $(0,a_0, \dots, a_{n-1})$. Furthermore, the collection of all $W_n^m$ form a directed system via the homomorphisms $v$ and $i$:
$$
\xymatrix{
W_n^m \ar@{^{(}->}[d]^v \ar@{^{(}->}[r]^i &W_n^{m+1}\ar@{^{(}->}[d]^v\\
W_{n+1}^m \ar@{^{(}->}[r]^i &W_{n+1}^{m+1}}
.
$$

\begin{exam}
	Over $\mathbb{F}_p$ we have $$W(\mathbb{F}_p)=\mathbb{Z}_p,\quad W_n(\mathbb{F}_p)=\mathbb{Z}/p^n \mathbb{Z},\quad \varinjlim W_n(\mathbb{F}_p)\cong \mathbb{Q}_p/\mathbb{Z}_p,$$ and in general, $$W_n(k)=W(k)/p^n W(k),\quad \varinjlim W_n(k)\cong \mathrm{Frac}(W(k))/W(k)=W(k)[\frac{1}{p}]/W(k).$$
\end{exam}

Let $\sigma: W(k)\rightarrow W(k)$ be the ring endomorphism induced by $F$.

\begin{defn}
Denote by $D=D_k$ the ring of ``non-commutative polynomials'' over $W(k)$ in two variables $F$ and $V$, subject to the following relations:

(1) $F\xi=\sigma(\xi)F, \forall \xi\in W(k)$,

(2) $V\sigma(\xi)=\xi V, \forall \xi\in W(k)$,

(3) $FV=VF=p$.
\end{defn}

Note that $D$ as a $W(k)$ module is free with a basis $\{\dots,V^2,V,1,F,F^2,\dots\}$.

\begin{defn}
Let $G$ be a finite commutative group scheme of local-local type. We define the (contravariant) Dieudonn\'e module of $G$ to be

 $$M(G)=\varinjlim _{m,n} \Hom(G,W_n^m),$$
with its induced left $D_k$-module structure via the actions of $D_k$ on $W_n^m$.
\end{defn}

\begin{rmk} The action of $W(k)$ on $W_n^m$ needs to be modified so that it is compatible with the transition maps between the $W_n^m$'s. Namely, the action of $x\in W(k)$ on $W_n^m$ is modified to be multiplication by $\sigma^{-n}(x)$.
\end{rmk}

\begin{thm}\label{locloc}The functor $G\mapsto M(G)$ gives an exact anti-equivalence of categories between {finite $k$-group schemes of local-local type} and {left $D_k$-modules of finite $W(k)$-length with $F$, $V$ nilpotent}.

\end{thm}

This ``main theorem of contravariant Dieudonn{\'e} theory in the local-local case'' is essentially a formal consequence of the results obtained thus far. Let $D_n^m=D/(DF^m+DV^n)$. By definition we have homomorphisms $D_n^m\rightarrow \End(W_n^m)$. Also recall that as $M(W_n^m)=\underset{u,v}{\varinjlim} \Hom(W_n^m,W_v^u)$, we have homomorphisms $\End(W_n^m)\rightarrow M(W_n^m)$.

\begin{prop}\label{iso}
(1) $D_n^m\xrightarrow{\cong} \End(W_n^m)\xrightarrow{\cong} M(W_n^m)$.

(2) $\textrm{length}_{W(k)}(M(G))=\log_p|G|$.

\end{prop}

\begin{proof}
We prove the second isomorphism in (1), then (2), then finally the first isomorphism in (1).

For (1) since $W_n^m\hookrightarrow W_{n'}^{m'}$ is a monomorphism for all $n\leq n'$ and $m\leq m'$, the map $D_n^m\rightarrow \End(W_n^m)$ is injective. By lemma it is also surjective. This proves the second isomorphism of (1).

Now for (2) we prove this using induction on $|G|$. The assertion is trivial for $|G|=1$. For $G=\boldsymbol{\alpha}_p=W_1^1$, we have $M(\boldsymbol{\alpha}_p)=M(W_1^1)=\End(\boldsymbol{\alpha}_p)\cong k$ has $W(k)$-length 1.

If $|G|>1$, there exists a short exact sequence
$$0\rightarrow G'\rightarrow G\rightarrow\boldsymbol{\alpha}_p\rightarrow 0.$$

Assume the assertion holds for $G'$. The induced sequence
$$0\leftarrow M(G')\leftarrow M(G)\leftarrow M(\boldsymbol{\alpha}_p)\leftarrow 0$$
is exact except possibly at $M(G')$. To prove exactness, consider $[\varphi]\in M(G')$ represented by a homomorphism $\varphi: G'\rightarrow W_n^m$ for some $m,n$. Consider the morphism of short exact sequences
$$
\xymatrix{
0\ar[r] &G' \ar[r]\ar[d]^\varphi &G \ar[r]\ar[d] &\boldsymbol{\alpha}_p \ar[r]\ar@{=}[d] &0\\
0\ar[r] &W_n^m \ar[r] &H \ar[r] &\boldsymbol{\alpha}_p \ar[r] &0
}
$$
where $H$ is the pushout of the left hand square.

Apply Proposition \ref{ext} to the lower exact sequence yields a homomorphism $H\rightarrow W_n^m$ extanding the homomorphism $i\circ v:W_n^m\rightarrow W_{n+1}^{m+1}$:
$$
\xymatrix{
0\ar[r] &G' \ar[r]\ar[d]^\varphi &G \ar[r]\ar[d] &\boldsymbol{\alpha}_p \ar[r]\ar@{=}[d] &0\\
0\ar[r] &W_n^m \ar[r]\ar[d]^{iv} &H \ar[r]\ar@{.>}[dl] &\boldsymbol{\alpha}_p \ar[r] &0\\
&W_{n+1}^{m+1}
}.
$$
The composite homomorphism $\psi: G\rightarrow H\rightarrow W_{n+1}^{m+1}$ defines an element $[\psi]\in M(G)$ which maps to $[\varphi]$. This proves the exactness of the sequence, therefore
\begin{align*}
\textrm{length}_{W(k)}M(G) &=\textrm{length}_{W(k)}M(G')+\textrm{length}_{W(k)}M(\boldsymbol{\alpha}_p)\\
&=\log_p|G'|+1\\
&=\log_p|G|.
\end{align*}

Returning to (1)  one easily checks that every non-trivial $D$-submodule of $D_n^m$ contains the residue class of $F^{m-1}V^{n-1}$. Since the image of $F^{m-1}V^{n-1}$ in $\End(W_n^m)$ is non-zero, we deduce that the map $D_n^m\rightarrow \End(W_n^m) $ is injective.

One directly calculates that $\textrm{length}_{W(k)}D_n^m=nm$. By (2) we also have $\textrm{length}_{W(k)}M(W_n^m)=nm$. Thus $D_n^m\rightarrow \End(W_n^m) $ is an injective homomorphism of $D$-modules of equal finite length, hense it is an isomorphism.

\end{proof}

\begin{lemma}
The functor $M$ is exact.
\end{lemma}

\begin{proof}
	The functor $M$ is left exact by construction. For any exact sequence $0\rightarrow G'\rightarrow G\rightarrow G''\rightarrow 0$, Proposition \ref{iso} (2) and the multiplicativity of group orders imply that the image of the induced map $M(G)\rightarrow M(G'')$ has the same finite length over $W(k)$ as $M(G')$ itself. Thus the map is surjective, and $M$ is exact.
\end{proof}

\begin{lemma}
If $F_G^m=0$ and $V_G^n=0$, then $F^m$ and $V^n$ annihilates $M(G)$. In particular, the functor $M$ lands in the indicated subcategory.
\end{lemma}

\begin{proof}
	This is clear from the definition of $M(G)$, the functoriality of $F$ and $V$, and Proposition \ref{iso} (2).
\end{proof}	

\begin{lemma}
The functor $M$ is fully faithful.
\end{lemma}

\begin{proof}
	Let $G$ and $H$ be finite commutative group schemes over $k$ of local-local type. Choose $m, n$ such that $F^m, V^n$ annihilate $G$, $H$. For convenience, we abbreviate $U:=W_n^m$ in this proof. By Proposition \ref{copres}, we may choose a copresentation
	$$0\rightarrow H\rightarrow U^r\rightarrow U^s$$
	for some $r$ and $s$. By exactness of $M$, we obtain a presentation of $D$ modules
	$$0\leftarrow M(H) \leftarrow M(U)^r\leftarrow M(U)^s.$$
	Applying the left exact functor $\textrm{Hom}(G,-)$ and $\textrm{Hom}_D(-,M(G))$, we obtain a commutative diagram with exact rows
	$$
	\xymatrix{
		0\ar[r] &\textrm{Hom}(G,H) \ar[r]\ar[d]^M &\textrm{Hom}(G,U^r) \ar[r]\ar[d]^M &\textrm{Hom}(G,U^s) \ar[d]^M \\
		0\ar[r] &\textrm{Hom}_D(M(H),M(G)) \ar[r] &\textrm{Hom}_D(M(U^r),M(G)) \ar[r] &\textrm{Hom}_D(M(U^s),M(G))\\
	}
	$$
	where the vertical arrows are induced by the functor $M$. We must prove that the left vertical arrow is bijective. By the $5$-Lemma it suffices to show that the other two vertical arrows are bijective. Since $M$ is an additive functor, this in turn reduces to direct summands for $U^r$ and $U^s$. All in all, it suffices to prove the bijectivity in the case $H=U=W_n^m$, which is clear from Proposition \ref{iso} (1).
	
\end{proof}

\begin{lemma}
The functor $M$ is essentially surjective.
\end{lemma}

\begin{proof}
	Let $N$ be a left $D$-module of finite length with $F$ and $V$ nilpotent. Suppose that $F^m$ and $V^n$ annihilate $N$. Then there exists an epimorphism of $D$ modules $(D_n^m)^{\oplus r}\twoheadrightarrow N$ for some $r$. Its kernel is again annihilated by $F^m$ and $V^n$; hence there exists a presentation
	$$(D_n^m)^{\oplus s}\xrightarrow{\varphi}(D_n^m)^{\oplus r}\rightarrow N\rightarrow 0.$$
	Since $D_n^m=M(W_n^m)$ and $M$ is fully faithful, we have $\varphi=M(\psi)$ for a unique homomorphism $(W_n^m)^{\oplus r}\xrightarrow{\psi}(W_n^m)^{\oplus s}$. Setting $G(N):=\ker (\psi)$, the $5$-Lemma shows that $N\cong M(G(N))$.
\end{proof}	

Putting the above results together, we see that Theorem \ref{locloc} is proven.

\subsection{A duality theorem of Dieudonn\'e modules}

Let $k$ be a perfect field of characteristic $p>0$, and consider the torsion $W(k)$-module $W(k)[\frac{1}{p}]/W(k)$.
Let $M$ be a finite length $W(k)$-module, define
$$M^\vee:=\Hom_{W(k)}(M,W(k)[\frac{1}{p}]/W(k)).$$

\begin{prop}\label{dual}
The functor $M\mapsto M^\vee$ defines an equivalence from the category of finite length $W(k)$-modules to itself, and there is a functorial isomorphism $M\cong (M^\vee)^\vee$.
\end{prop}

\begin{proof}
	Since the functor is additive, and every $M$ is a direct sum of cyclic modules, it suffices to prove the isomorphism in the case $M=W(k)/p^nW(k)$, which is clear.
\end{proof}	

Now suppose $M$ is a $D$-module, define maps $F^\ast, V^\ast: M^\vee\rightarrow M^\vee$ by

$$(F^\ast m^\ast)(m):=\sigma(m^\ast(Vm))$$
$$(V^\ast m^\ast)(m):=\sigma^{-1}(m^\ast(Fm))$$
for $m\in M$ and $m^\ast\in M^\vee$. Together with the action of $W(k)$ on $M^\vee$ through its action on $\varinjlim W_n(k)$, this turns $M^\vee$ into a $D$-module.

\begin{cor}
The functor $M\mapsto M^\vee$ defines an equivalence from the category of finite length $D$-modules to itself, and there is a functorial isomorphism $M\cong (M^\vee)^\vee$.
\end{cor}

\begin{proof}
	This is a direct consequence of Proposition \ref{dual}.
\end{proof}	

\begin{prop}\label{duality}
For any finite commutative group scheme $G$ over $k$ of local-local type, there is a functorial isomorphism of $D$-modules
$$M(G^\vee)\cong M(G)^\vee.$$
\end{prop}

\begin{proof}
Recall that 
$D_n^m:=D/(DF^m+DV^n)\cong \End(W_n^m)\cong M(W_n^m).$
This essentially reduces to the duality between $W_n^m$ and $W_m^n$.
\end{proof}

\section{The Dieudonn\'e functor in the \'etale case}

Let $D$ act on $W_n$ on the left, where $F$ and $V$ act as such and $\xi\in W(k)$ through multiplication by $\sigma^{-n}(\xi)$. Then the monomorphisms $v: W_n\hookrightarrow W_{n+1}$ are $D$-equivariant. Also, the $W_n^m$ form a fundamental system of infinitesimal neighborhoods of zero in all $W_n$. Thus for any finite commutative group scheme $G$ of local-local type, we have
$$M(G)=\varinjlim _{m,n} \Hom(G,W_n^m)=\varinjlim _{n} \Hom(G,W_n).$$
Using the latter description we now give a similar results for finite commutative group schemes of reduced-local type:

\begin{thm}
	The functor $G\mapsto M(G)=\varinjlim _{n} \Hom(G,W_n)$ induces an anti-equivalence from the category of finite commutative \'etale group schemes over $k$ of $p$-power order to the category of left $D$-modules of finite length with $F$ an isomorphism.
	Moreover, $\textrm{length}_{W{k}}M(G)=\log_p|G|$.
	
\end{thm}	
\begin{proof}
	The idea is to prove this over algebraically closed field $\bar{k}$ first, then use Galois decent and functoriality of $M$ to prove the general case. See \cite[~p.71]{Pink}.
\end{proof}

\section{The Dieudonn\'e functor in the general case}

Recall from Chapter 2 that any finite commutative group scheme $G$ of $p$-power order admits a unique decomposition
$$G=G_{rl}\oplus G_{lr}\oplus G_{ll}.$$
Previously we have already defined $M(G_{ll})$ and $M(G_{rl})$. Since $G_{lr}^\vee$ is of reduced-local type, we can define 
$$M(G_{lr}):=M(G_{lr}^\vee)^\vee$$ and thus
$$M(G):=M(G_{rl})\oplus M(G_{lr}^\vee)^\vee\oplus M(G_{ll}).$$
By construction this is a finite length left $D$-module, and we have
$$\textrm{length}_{W(k)}M(G)=\log_p|G|.$$

\begin{lemma}
	Every finite length left $D$-module has a unique and functorial decomposition 
	$$M=M_{rl}\oplus M_{lr}\oplus M_{ll}$$
	where $F$ is isomorphic (resp. nilpotent, nilpotent) and $V$ is nilpotent (resp. isomorphic, nilpotent) on $M_{rl}$ (resp. $M_{lr}$, $M_{ll}$).
\end{lemma}

\begin{proof}
	The images of $F^n: M\rightarrow M$ form a descending sequence of $D$-submodules of $M$. Since $M$ has finite length, this sequence stabilizes, say with $F^nM=M'$ for all $n\gg 0$. Then $F: M'\rightarrow M'$ is an isomorphism, and by looking at the length we have $M=M'\oplus \ker(F^n)$. Repeating the argument with $V$ on $\ker(F^n)$ we obtain the desired decomposition. Uniqueness and functoriality are clear.
\end{proof}	

Recall from Proposition \ref{duality} that there is a functorial isomorphism $M(G_{ll}^\vee)\cong M(G_{ll})^\vee$. This isomorphism extends to $G$. We have now proven:

\begin{thm}
The above functor $M$ induces an anti-equivalence of categories between finite commutative groups schemes over $k$ of $p$-power order and left $D$-modules of finite length. Moreover, $\textrm{length}_{W{k}}M(G)=\log_p|G|$, and there is a functorial isomorphism $M(G^\vee)\cong M(G)^\vee$.
\end{thm}

\section{The Dieudonn\'e functor for $p$-divisible groups}

First let us prove a lemma.

\begin{lemma}
	Let $\cdots\rightarrow M_{n+1}\xrightarrow{\pi_n}M_n\rightarrow \cdots \rightarrow M_1$ be a projective system of $W(k)-modules$ with the following properties.
	
	(1) The sequence $M_{n+1}\xrightarrow{p^n}M_{n+1}\xrightarrow{\pi_n}M_n\rightarrow 0$ is exact for all $n$.  
	
	(2) $M_n$ is of finite length for all $n$.
	
	Let $M=\varprojlim_nM_n$. Then $M$ is a finitely generated $W(k)$-module and the canonical map $M\rightarrow M_n$ identifies $M_n$ with $M/p^nM$ for all $n$.
\end{lemma}	

\begin{proof}
	From (1) it follows that 
	$$M_{n+m}\xrightarrow{p^n}M_{n+m}\xrightarrow{\pi_n\circ \cdots \circ\pi_{n+m-1}}M_n\rightarrow 0$$
	is exact for all $n$ and $m$. Taking the inverse limit over $m$, we obtain an exact sequence 
	$$M\xrightarrow{p^n}M\rightarrow M_n\rightarrow 0$$
	where the second map is the canonical projection, hence the last assertion. 
	
	Now let $m_1, \dots, m_r$ be elements in $M$ generating $M/pM=M_1$; consider the $W(k)$-module homomorphism $\varphi: W(k)^r\rightarrow M$ such that $\varphi(a_1,\dots,a_r)=a_1m_1+\dots+a_rm_r$. It induces surjective maps $W(k)^r/p^nW(k)^r\rightarrow M/p^nM$ for all $n$, hence is surjective as an inverse limit of surjective maps of finite length modules.
\end{proof}	

\begin{defn}
	Let $H=(G_n,i_n)$ be a $p$-divisible group over $k$. Define $$M(H):=\varprojlim_nM(G_n).$$
\end{defn}

Combining previous results, we have the following:

\begin{thm}
	The functor $H\mapsto M(H)$ induces an anti-equivalence between the category of $p$-divisible groups over $k$ and the category of $D$-modules which are free of finite rank over $W(k)$.
	
	Furthermore, we have:
	
	(1) For any perfect extension $K/k$, there is a functorial isomorphism of $D$-modules $$M(H\otimes_kK)\cong W(K)\otimes_{W(k)}M(H).$$
	
	(2) If $H^\vee$ is the dual of $H$, then
	$$M(H^\vee)=\textrm{Hom}_{W(k)}(M(H),W(k)),$$
	with $(F_{M(H^\vee)}(f))(m)=f(V_M(m))^{(p)}$, $(V_{M(H^\vee)}(f))(m)=f(F_M(m))^{p^{-1}}$.
\end{thm}	

We call $M(H)$ the Dieudonn\'e module of $H$. Recall that $D$ is the ring of ``non-commutative polynomials'' over $W(k)$ in two variables $F$ and $V$, subject to the following relations:

(1) $F\xi=\sigma(\xi)F$, $\forall \xi\in W(k)$,

(2) $V\sigma(\xi)=\xi V$, $\forall \xi\in W(k)$,

(3) $FV=VF=p$.

To give a free $D$-module $M$ of finite rank over $W(k)$ is equivalent to give a free $W(k)$-module $M$ of finite rank together with two maps $F: M\rightarrow M$ and $V: M\rightarrow M$, such that

(1) $F$ is $\sigma$-linear: $F(\xi m)=\sigma(\xi)F(m)$, $\forall \xi\in W(k), m\in M$,

(2) $V$ is $\sigma^{-1}$-linear: $V(\xi m)=\sigma^{-1}(\xi)V(m)$, $\forall \xi\in W(k), m\in M$,

(3) $F\circ V=V\circ F=p$.

Furthermore, write $F=\varphi$ and $V=p\circ\varphi^{-1}$, the triple $(M,F,V)$ is equivalent to the pair $(M, \varphi)$ where $M$ is as before and $\varphi: M\rightarrow M$ is $\sigma$-linear such that $pM\subseteq\varphi M\subseteq M$. Thus $(M,\varphi)$ is an $F$-crystal with all Hodge slopes equal to $0$ or $1$.

\begin{defn}
	Let $H$ be a $p$-divisible group over $k$. Let $(M,F,V)$ be its Dieudonn\'e module. The \emph{dimension} of $H$ is $\dim_k(M/FM)$, the \emph{codimension} of $H$ is $\dim_k(FM/pM)$. 
\end{defn}

Let $d$ be the dimension of $H$ and $c$ the codimension of $H$. It is not hard to check that $c+d$ is the rank of $M$, which is also the hight of $H$. Furthermore, the number $c$ is the dimension of $H^\vee$, the dual $p$-divisible group of $H$.
\chapter{Truncated Barsotti--Tate Groups}

Through out this chapter, let $k$ be an algebraically closed field of characteristic $p>0$. Let $c, d\geq 0$ be such that $h:=c+d>0$.

\section{Barsotti--Tate groups of level 1}
Let $H$ be a $p$-divisible group of height $h$ over $k$. Let $H[p]$ be its $p$-kernel, which is a finite group scheme over $k$ of order $p^h$. Kraft showed in the unpublished manuscript \cite{Kraft} that, fixing $h$, there are only finitely many such group schemes up to isomorphism. This result was re-obtained, independently, by Oort. Together with Ekedahl he used it to define and study a stratification of the moduli space of principally polarized abelian varieties over $k$. Their results can be found in \cite{Oort:stramoduli}. We will review Kraft's result in this section.

Let $\mathscr{C}(1)_k$ be the category of finite group schemes over $k$ which are annihilated by $p$. Note that $W(k)/pW(k)=k$, Dieudonn\'e theory tells us that this category is equivalent to the category of triples $(M,F,V)$ where
\begin{itemize}
\item $M$ is a finite dimensional $k$-vecter space,
\item $F:M\rightarrow M$ is a $\sigma$-linear endomorphim,
\item $V:M\rightarrow M$ is a $\sigma^{-1}$-linear endomorphim,
\end{itemize}
such that $F\circ V=V\circ F=0$. 

A necessary and sufficient condition for $G\in \mathscr{C}(1)_k$ to be a Barsotti--Tate group of level 1 (i.e. $p$-kernel of a $p$-divisible group) is that the sequence
$$G\xrightarrow{F_G} G^{(p)}\xrightarrow{V_G} G$$
is exact. On the Dieudonn\'e module this means that we have
$$\ker(F)=\im(V) \textrm{ and } \ker(V)=\im(F).$$

In the unpublished manuscript \cite{Kraft}, Kraft showed that the objects of $\mathscr{C}(1)_k$ admit a normal form. To describe this, one distinguished two types of group schemes.

(1) \emph{Linear type}. Consider a linear graph $\Gamma$:

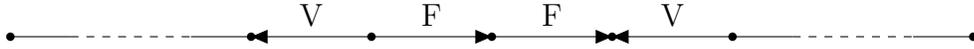
\begin{center}
	\begin{tikzpicture}[scale=.8]
	
	\draw (0,0)--(1,0);
	\draw [dashed] (1,0) -- (3,0);
	\draw (3,0)--(4,0);
	\draw [arrows={-triangle 45}] (6,0)--(4,0);
	\draw [arrows={-triangle 45}] (6,0)--(8,0);
	\draw [arrows={-triangle 45}] (8,0)--(10,0);
	\draw [arrows={-triangle 45}] (12,0)--(10,0);
	\draw (12,0)--(13,0);
	\draw [dashed](13,0)--(15,0);
	\draw (15,0)--(16,0);
	
	\fill (0,0) circle (2pt);
	\fill (4,0) circle (2pt);
	\fill (6,0) circle (2pt);
	\fill (8,0) circle (2pt);
	\fill (10,0) circle (2pt);
	\fill (12,0) circle (2pt);
	\fill (16,0) circle (2pt);
	
	\node [above] at (5,0) {V};
	\node [above] at (7,0) {F};
	\node [above] at (9,0) {F};
	\node [above] at (11,0) {V};

	\end{tikzpicture}
\end{center}

with all edges labeled either by $F$ (drawn as $\xrightarrow{F}$) or by $V$ (drawn as $\xleftarrow{V}$). We associate to $\Gamma$ a Dieudonn\'e module $(M,F,V)$ in the following way. For each vertex $v$ we have a base vector $e_v$. Define $F$ and $V$ according to the arrows in $\Gamma$. So $F(e_v)=e_w$ if there is an $F$-arrow from $v$ to $w$, and $F(e_v)=0$ is there is not an $F$-arrow starting at $v$. Similarly, $V(e_v)=e_w$ if there is a $V$-arrow from $v$ to $w$, and $V(e_v)=0$ is there is not a $V$-arrow starting at $v$.

One readily checks that this defines a Dieudonn\'e module $M_\Gamma=(M_\Gamma,F,V)$. Write $G_\Gamma$ for the corresponding group scheme. 

(2) \emph{Circular type}. Consider a circular graph $\Gamma$:

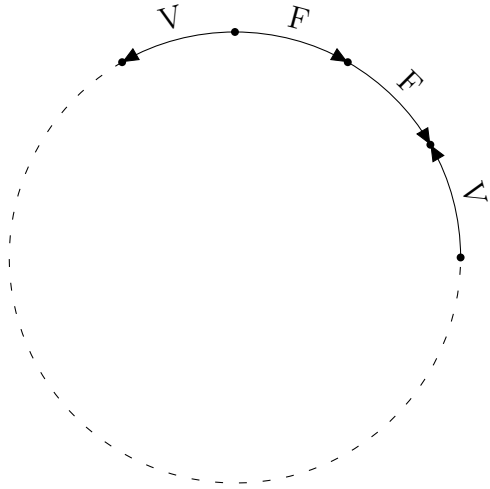
\begin{center}
\begin{tikzpicture}


\node(1) at ($(3,3) +(0:3)$) {};
\node(2) at ($(3,3) +(30:3)$) {};
\node(3) at ($(3,3) +(60:3)$) {};
\node(4) at ($(3,3) +(90:3)$) {};
\node(5) at ($(3,3) +(120:3)$) {};


\draw[loosely dashed] (5) arc (120:360:3);
\draw [arrows={-triangle 45}] (1) arc (0:30:3);
\draw [arrows={-triangle 45}] (3) arc (60:30:3) ;
\draw [arrows={-triangle 45}] (4) arc (90:60:3) ;
\draw [arrows={-triangle 45}] (4) arc (90:120:3) ;

\node[rotate=-75] at ($(3,3) +(15:3.3)$) {V};
\node[rotate=-45] at ($(3,3) +(45:3.3)$) {F};
\node[rotate=-15] at ($(3,3) +(75:3.3)$) {F};
\node[rotate=15] at ($(3,3) +(105:3.3)$) {V};

\fill (1) circle (1.5pt);
\fill (2) circle (1.5pt);
\fill (3) circle (1.5pt);
\fill (4) circle (1.5pt);
\fill (5) circle (1.5pt);

\end{tikzpicture}
\end{center}
where again all edges are labeled by $F$ or $V$. We require that the $F-V$ pattern is not periodic, i.e., $\Gamma$ is not invariant under a non-trivial rotation. Two circular diagrams which differ by a rotation will be considered equivalent. By the same rules as in the linear case, we obtain a Dieudonn\'e module $M_\Gamma=(M_\Gamma,F,V)$ and write $G_\Gamma$ for the corresponding group scheme. 

\begin{exam}
	The category $\mathscr{C}(1)_k$ has three simple objects: $\underline{\ZZ/p\ZZ}$, $\boldsymbol{\mu}_p$ and $\boldsymbol{\alpha}_p$. The corresponding graphs are :
	\medskip
	
	
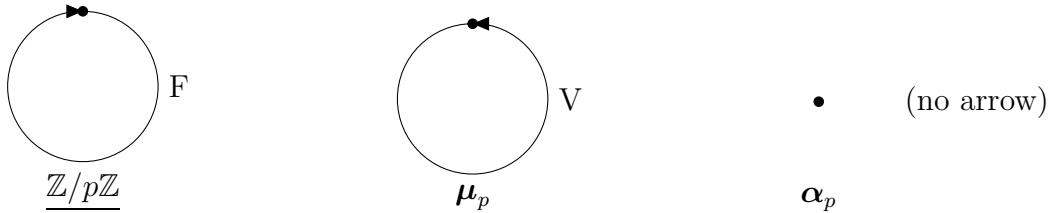
\begin{tikzpicture}
	\draw [arrows={-triangle 45}] (0,0) arc (90:-270:1);
	\fill (0,0) circle (2pt);
	\node [right] at (1,-1){F};
	\node [below] at (0,-2) {$\underline{\ZZ/p\ZZ}$};
	\end{tikzpicture}\hfill
	\begin{tikzpicture}
	\draw [arrows={-triangle 45}] (0,0) arc (-270:90:1);
	\fill (0,0) circle (2pt);
	\node [right] at (1,-1){V};
	\node [below] at (0,-2) {$\boldsymbol{\mu}_p$};
	\end{tikzpicture}\hfill
	\begin{tikzpicture}
	\fill (0,0) circle (2pt);
	\node [right] at (1,0){(no arrow)};
	\node [below] at (0,-1) {$\boldsymbol{\alpha}_p$};
	\end{tikzpicture}

\end{exam}

\begin{thm}[Kraft]\label{kraft}

(i) If $\Gamma$ is a diagram of the type described as above then $G_\Gamma$ is indecomposable. If $G_\Gamma \cong G_{\Gamma'}$, then $\Gamma$ and $\Gamma'$ are equivalent, i.e., either $\Gamma=\Gamma'$ of linear type or $\Gamma$ and $\Gamma'$ are of circular type and differ by a rotation.

(ii) Every indecomposable object $G$ in $\mathscr{C}(1)_k$ is isomorphic to some $G_\Gamma$.

(iii) The group scheme $G_\Gamma$ is a $BT_1$ if and only if $\Gamma$ is a circular diagram.

(iv) The Cartier dual of $G_\Gamma$ is isomorphic to $G_{\Gamma^\vee}$ where $\Gamma^\vee$ is the diagram obtained by changing all $F$-arrows into $V$-arrows and vice versa.

\end{thm}

Let $H$ be a $p$-divisible group over $k$ with codimension $c$ and dimension $d$. Apply Theorem \ref{kraft}, we can see that $H[p]$ can be represented by a (circular) word with $F$ appearing $c$ times and $V$ appearing $d$ times. Let $h=c+d$ be the height of $H$. Then $H[p]$ can be parameterized by the cosets of $S_h/S_c\times S_d$. In particular, the number of isomorphism classes of such $H[p]$ is $\frac{h!}{c!\cdot d!}$.

This observation also motivates the following definition. Let $c, d\geq 0$ be two integers such that $h=c+d>0$. Let $\pi\in S_h$ be a permutation on $J=\{1,2,\dots,h\}$. Define a Dieudonn\'e module $(M,\varphi_\pi)$ as follows.
Let $M=W(k)^h$. Let $\{e_1, e_2,\dots,e_h\}$ be a $W(k)$-basis of $M$. Define $\varphi_\pi: M\rightarrow M$ on basis elements as
$$\varphi_\pi(e_i)=\left\{\begin{array}{lr}
pe_{\pi(i)},\quad& 1\leq i\leq d; \\
e_{\pi(i)},\quad& d<i\leq h.
\end{array}
\right.$$
and extend $\varphi_\pi$ to $M$ $\sigma$-linearly. Write $H_\pi$ for the corresponding $p$-divisible group. Then $H_\pi$ is a $p$-divisible group of codimension $c$ and dimension $d$. 

\begin{defn}
The $p$-divisible groups $H_\pi$ defined as above are called canonical lifts of Barsotti--Tate groups of level 1.
\end{defn}

\begin{thm}[Vasiu]
	Let $c,d, h$ be as before. For every $p$-divisible group $(M,\varphi)$ with dimension $d$ and height $h$, there is $\pi\in S_h$ and $g\in GL_M(W(k))$ with $g\equiv 1_m \mod p$, such that $(M,\varphi)$ is isomorphic to $(M,g\varphi_\pi)$.	
\end{thm}	

In particular, we have that every $BT_1$ is isomorphic to some $H_\pi[p]$, which justifies the name.
\section{Automorphism group schemes}

Let $U$ be a finite dimensional vector space over $k$. Then the functor
\[\boldsymbol{GL}_U(R)=\Aut_R(U\otimes_kR)\]
that associates to every $k$-algebra $R$ the group of $R$-linear automorphisms of $U\otimes_kR$ is representable by a scheme, and so is a group scheme. In fact suppose that $e_1, 
\dots, e_n$ is a basis of $U$. Then every $f\in \Aut_R(U\otimes_kR)$ is determined by
\[f(e_i)=\sum_{i=1}^{n}\lambda_{ij}e_j.\]

So giving a $R$-linear automorphism of $U\otimes_kR$ is the same as to give a matrix $[\lambda_{ij}]_{i,j}$ with determinant in $S^\times$. So $\boldsymbol{GL}_U$ is represented by the localization of $k[x_{ij}]$ at $\det(x_{ij})$. If $U=k^n$, then $\boldsymbol{GL}_U$ is just $\boldsymbol{GL}_n$ in example \ref{aaa}. 

Now assume $U$ has some sort of algebraic structure, for instance a bilinear multiplication, or even a whole Hopf $k$-algebra structure. Let $F(R)$ be those $R$-linear endomorphisms $U\otimes_kR\rightarrow U\otimes_kR$ that preserve the given structure. The condition that an $R$-linear endomorphism preserve the structure is given by polynomial equations in the matrix entries: for multiplication, e.g., we only need the equations saying that the product is preserved for basis elements. Thus $F$ is representable and we call it $\Ends(U)$, the \emph{endomorphism group scheme} of $U$. Similarly, we can define $\Auts(U)$, the \emph{automorphism group scheme} of $U$, by taking the invertible  $R$-linear automorphisms preserving the structure.

It is clear that $\Auts(U)$ is an open subscheme of $\Ends(U)$. Also $\Auts(U)$ is a closed subgroup scheme of $\boldsymbol{GL}_U$.

If $G=\spec A$ is a finite group scheme over $k$, we write $\Auts(G):=\Auts(A)$ and $\Ends(G):=\Ends(A)$.

\begin{defn}
	Let $H$ be a $p$-divisible group over $k$ of codimension $c$ and dimension $d$, and $m\in \NN^\ast$. We write
	\[\gamma_H(m):=\dim(\Auts(H[p^m])).\]
	
	We call $(\gamma_H(m))_{m\ge 1}$ the \emph{centralizing sequence} of $H$ and we call $s_H:=\gamma_H(n_H)$ the \emph{specializing height} of $H$, where $n_H$ is the isomorphism number of $H$.
\end{defn}

The importance of the number $\gamma_H(m)$ stems out from the following three main facts (cf. \cite{GV1}):

(i) They are codimension of the versal level $m$ strata.

(ii) They can compute the isomorphism number $n_H$.

(iii) They are a main source of invariants that go up under specializations.

\begin{thm}[Gabber--Vasiu]\label{gamma}
	The centralizing sequence of $H$ has the following three basic properties:
	(a) The sequence $(\gamma_H(m))_{m\ge 1}$ is an increasing sequence in $\NN$.
	
	(b) For each $l\in \NN^\ast$, the sequence $(\gamma_H(m+l)-\gamma_H(m))_{m\ge1}$ is a decreasing sequence in $\NN$.
	
	(c) If $cd>0$, then we have $\gamma_H(1)<\gamma_H(2)<\cdots<\gamma_H(n_H)$.
	
	(d) For $m\ge n_H$ we have $\gamma_H(m)=\gamma_H(n_H)\le cd$.
\end{thm}

\begin{proof}
	See \cite[Theorem 1]{GV1}.
\end{proof}

We can associate another automorphism group scheme to $H[p^m]$ by using the Dieudonn\'e module of $H$, cf. \cite{Vasiu:levelm}. Let $(M, \varphi,\vartheta)$ be the Dieudonn\'e module of $H$. Let $\Auts(H[p^m])_{\textrm{crys}}$ be the group scheme over $k$ of automorphisms of $(M/p^mM, \varphi_m, \vartheta_m)$. Thus, if $R$ is a $k$-algebra and if $\sigma_R$ is the Frobenius endomorphism of $W_m(R)$, then $\Auts(H[p^m])_{\textrm{crys}}$ is the subgroup of $\mathcal{D}_m(R)=\boldsymbol{GL}_M(W_m(R))$ formed by those $W_m(R)$-linear automorphisms $f$ of $W_m(R)\otimes_{W_m(k)}M/p^mM$ that satisfy the identities $(1_{W_m(R)}\otimes \varphi_m) \circ f^{(\sigma)}=f\circ (1_{W_m(R)}\otimes \varphi_m)$ and $f^{(\sigma)}\circ (1_{W_m(R)}\otimes \vartheta_m) = (1_{W_m(R)}\otimes \vartheta_m)\circ f)$; here $\varphi_m$ and $\vartheta_m$ are viewed as $W_m(k)$-linear maps $(M/p^mM)^{(\sigma)}\rightarrow M/p^mM$ and $M/p^mM \rightarrow (M/p^mM)^{(\sigma)}$ respectively. We similarly define the group scheme $\Ends(H[p^m])_{\textrm{crys}}$ of endomorphisms of $(M/p^mM, \varphi_m, \vartheta_m)$. We will see in the next section that there is a close relation between $\Auts(H[p^m])$ and $\Auts(H[p^m])_{\textrm{crys}}$.

\section{Orbit spaces of truncated Barsotti--Tate groups}

In this section we include the group action $\mathbb{T}_m$ introduced in \cite[Section 5]{Vasiu:reconstr}. Its set of orbits parametrizes the isomorphism classes of truncated Barsotti--Tate groups of level $m$ over $k$ that have codimension $c$ and dimension $d$.

Let $H$ be a $p$-divisible group of codimension $c$ and dimension $d$. Let $(M,\varphi)$ be the Dieudonn\'e module of $H$. Let $\vartheta=p\circ\varphi: M\rightarrow M$ be the Verschiebung map of $(M,\varphi)$. 

\subsection{The scheme $\mathcal{H}$}
The classification of $F$-crystals over $k$ by Dieudonn\'e (see \cite{Demazure1}, Ch. IV) implies that we have a direct sum decomposition $(M[\frac{1}{p}],\varphi) = \oplus(W_s,\varphi)$ into simple $F$-isocrystals over $k$. Since the Newton slopes of a Dieudonn\'e module is $0$ or $1$, we have a direct sum decomposition $M=F^1\oplus F^0$ such that $F^1/pF^1$ is the kernel of the reduction modulo $p$ of $\varphi$. This naturally gives a direct sum decomposition of $W(k)$-modules
$$\End(M) = \Hom(F^0,F^1)\oplus\End(F^1)\oplus\End(F^0)\oplus\Hom(F^1,F^0).$$

Let $\mathcal{W}_+$ be the maximal subgroup scheme of $\boldsymbol{GL}_M$ that fixes both $F^1$ and $M/F^1$; it is a closed subgroup scheme of $\boldsymbol{GL}_M$ whose Lie algebra is the direct summand $\Hom(F^0,F^1)$ of $\End(M)$ and whose relative dimension is $cd$. Let $\mathcal{W}_0:=\boldsymbol{GL}_{F^1}\times_{W(k)}\boldsymbol{GL}_{F^0}$; it is a closed subgroup scheme of $\boldsymbol{GL}_M$ whose Lie algebra is the direct summand $\End(F^1)\oplus\End(F^0)$ of $\End(M)$ and whose relative dimension is $d^2+c^2$. Let $\mathcal{W}_-$ be the maximal subgroup scheme of $\boldsymbol{GL}_M$ that fixes both $F^0$ and $M/F^0$; it is a closed subgroup scheme of $\boldsymbol{GL}_M$ whose Lie algebra is the direct summand $\Hom(F^1,F^0)$ of $\End(M)$ and whose relative dimension is $cd$. Let
$$\mathcal{H}:=\mathcal{W}_+\times_{W(k)}\mathcal{W}_0\times_{W(k)}\mathcal{W}_-;$$
it is a smooth, affine scheme over $\spec(W(k))$ of relative dimension $cd+d^2+c^2+cd=h^2$.

\subsection{The functor $\mathbb{W}_m$}

Let $m\geq 1$. Let $G$ be a smooth affine group scheme over $\spec(W(k))$. Let $\mathbb{W}_m(G)$ be the contravariant functor from the category of affine schemes over $k$ to the category of sets groups such that
$$\mathbb{W}_m(G)(\spec R): = G(W_m(R)).$$
It is well known that this functor is representable by an affine group scheme over $k$ of finite type to be denoted also by $\mathbb{W}_m(G)$, cf \cite[Sect. 4, Cor. 4, p.~641]{Greenberg}. We have $\mathbb{W}_m(G)(k)=G(W_m(k))$ and a natural identification $\mathbb{W}_1(G)=G_k$.

If $I$ is an ideal of $R$ of square $0$, then the ideal $\ker(W_m(R)\twoheadrightarrow W_m(R/I))$ is nilpotent and thus the reduction map $G(W_m(R))\rightarrow G(W_m(R/I))$ is surjective (cf. \cite[Chapter 2, Prop. 6]{bosch1990neron}). From this and loc. cit. we get that the scheme $\mathbb{W}_m(G)$ is smooth.

Suppose now that $G$ is a smooth, affine group scheme over $\spec W(k)$. Then $\mathbb{W}_m(G)$ is a smooth affine group over $k$. The length reduction $W(k)$-epimorphisms $W_{m+1}\twoheadrightarrow W_m(R)$ define naturally a smooth epimorphism
\[\textrm{Red}_{m+1,G}: \mathbb{W}_{m+1}(G)\twoheadrightarrow \mathbb{W}_m(G)\]
of affine group schemes over $k$. The kernel of $\textrm{Red}_{m+1,G}$ is the vector group over $k$ defined by the Lie algebra $\lie(G_k)^{\sigma^m}$, and thus it is a unipotent commutative group isomorphic to $\mathbb{G}_a^{\dim(G_k)}$. Using this and the identification $\mathbb{W}_1(G)=G_k$, by induction on $m\in \NN^\ast$ we get that:

(1) we have $\dim(\mathbb{W}_m(G))=m\dim(G_k)$;

(2) the group $\mathbb{W}_m(G)$ is connected if and only if $G_k$ is connected.

\subsection{The group action $\mathbb{T}_m$}

Let $\sigma_\varphi: M\xrightarrow{~} M$ be the $\sigma$-linear automorphism such that
$$\sigma_\varphi(m)=\left\{
\begin{array}{lr}
\frac{1}{p}\varphi(m), & m\in F^1;\\
\varphi(m), & m\in F^0.
\end{array}
\right.$$
Let $\sigma_\varphi$ act on the sets underlying the groups $\boldsymbol{GL}_M(W(k))$ and $\boldsymbol{GL}_M(W_m(k))$ in the natural way: if $g\in \boldsymbol{GL}_M(W(k))$, then
$$\sigma_\varphi(g)=\sigma_\varphi\circ g\circ \sigma_\varphi^{-1}, \quad \sigma_\varphi(g[m])=(\sigma_\varphi\circ g\circ \sigma_\varphi^{-1})[m].$$
For $g\in \mathcal{W}_+(W(k))$ (resp. $g\in \mathcal{W}_0(W(k))$ or $g\in \mathcal{W}_-(W(k))$ ) we have $\varphi(g) = \sigma_\varphi(g^p)$ (resp. $\varphi(g) = \sigma_\varphi(g)$ or $\varphi(g^p) = \sigma_\varphi(g)$).

Let $\mathcal{H}_m:=\mathbb{W}_m(\mathcal{H})$ and $\mathcal{D}_m:=\mathbb{W}_m(\boldsymbol{GL}_M)$. We have a natural action
$$\mathbb{T}_m: \mathcal{H}_m\times_k\mathcal{D}_m \rightarrow \mathcal{D}_m$$
defined on $k$-valued points as follows. If $h=(h_1, h_2, h_3) \in \mathcal{H}(W(k))$ and $g\in\boldsymbol{GL}_M(W(k))$, then the product of $h[m]=(h_1[m],h_2[m],h_3[m])\in \mathcal{H}_m(k)=\mathcal{H}(W_m(k))$ and $g[m]\in \mathcal{D}_m(k)=\boldsymbol{GL}_M(W_m(k))$ is the element
\begin{align*}
\mathbb{T}_m(h_m,g_m) :&= (h_1h_2h_3^pg\varphi(h_1h_2h_3^p)^{-1})[m]\\
&= (h_1h_2h_3^pg\varphi(h_3^p)^{-1}\varphi(h_2)^{-1}\varphi(h_1)^{-1})[m]\\
&= (h_1h_2h_3^pg\sigma_\varphi(h_3)^{-1}\sigma_\varphi(h_2)^{-1}\sigma_\varphi(h_1)^{-1})[m]\\
&= h_1[m]h_2[m]h_3[m]^pg[m]\sigma_\varphi(h_3[m])^{-1}\sigma_\varphi(h_2[m])^{-1}\sigma_\varphi(h_1[m])^{-1}\in\mathcal{D}_m(k).
\end{align*}

The formula $\mathbb{T}_m(h_m,g_m)=(h_1h_2h_3^pg\varphi(h_1h_2h_3^p)^{-1})[m]$ shows that the action $\mathbb{T}_m$ is intrinsically associated to $D$, i.e., it does not depend on the choice of the direct sum decomposition $M=F^1\oplus F^0$. For later use we mention that
$$\mathbb{T}_1(h_1,g_1) = h_1[1]h_2[1]g[1]\sigma_\varphi(h_3[1])^{-1}\sigma_\varphi(h_2[1])^{-1}\in\mathcal{D}_1(k)=\boldsymbol{GL}_M(k).$$

We have the following

\begin{lemma}
	Let $g_1, g_2\in\boldsymbol{GL}_M(W(k))$. Then the points $g_1[m], g_2[m]\in\mathcal{D}_m(k)$ belong to the same orbit of the action $\mathbb{T}_m$ if and only if the following two Dieudonn\'e modules $(M/p^mM, g_1[m]\varphi_m,\vartheta_mg_1[m]^{-1})$ and $(M/p^mM, g_2[m]\varphi_m,\vartheta_mg_2[m]^{-1})$ are isomorphic.
\end{lemma}
\begin{proof}
	\cite[Ch 2, Sect. 2.2]{Vasiu:levelm}
\end{proof}

\begin{cor}
	The set of orbits of the action $\mathbb{T}_m$ are in natural bijection to the set of isomorphism classes of truncated Barsotti--Tate groups of level $m$ over $k$ which have codimension $c$ and dimension $d$.
\end{cor}

Let $\mathcal{O}_m$ be the orbit of $1_M[m]\in\mathcal{D}_m(k)$ under the action of $\TT_m$. Let $\mathcal{S}_m$ be the subgroup scheme of $\mathcal{H}_m$ which is the stabilizer of $1_M[m]$ under the action $\TT_m$. Let $\mathcal{C}_m$ be the reduced group of $\mathcal{S}_m$. Let $\mathcal{C}^0_m$ be the identity component of $\mathcal{C}_m$. Then we have

\begin{thm}[Vasiu]
	Let $m\in \NN^\ast$. With the above notations, the following three properties hold:
	
	(a) the connected smooth affine group $\mathcal{C}_m^0$ is unipotent;
	
	(b) there exist two finite homomorphisms
	\[\iota_m: \mathcal{S}_m\rightarrow \Auts(H[p^m])_{\textrm{crys}}\]and\[ \zeta_m: \Auts(H[p^m])\rightarrow\Auts(H[p^m])_{\textrm{crys}}\]
	which at the level of $k$-valued points induce isomorphisms
	\[\iota_m(k): \mathcal{S}_m(k)\rightarrow \Auts(H[p^m])_{\textrm{crys}}(k)\] and \[\zeta_m(k): \Auts(H[p^m])(k)\rightarrow\Auts(H[p^m])_{\textrm{crys}}(k);\]
	
	(c) we have $\dim(\mathcal{S}_m)=\dim(\mathcal{C}_m)=\dim(\mathcal{C}_m^0)=\gamma_H(m)$.
		
\end{thm}

\begin{proof}
	\cite[Theorem 5]{Vasiu:levelm}.
\end{proof}

\begin{cor}
	The following properties hold:
	
	(a) The open embeddings $\Auts(H[p^m])\hookrightarrow \Ends(H[p^m])$ and $\Auts(H[p^m])_{\textrm{crys}}\hookrightarrow \Ends(H[p^m])_{\textrm{crys}}$ are also closed.
	
	(b) The identity component $\Auts(H[p^m])^0$ (resp. $\Auts(H[p^m])_{\textrm{crys}}^0$) is the translation of the identity component $\Ends(H[p^m])^0$ (resp. $\Ends(H[p^m])_{\textrm{crys}}^0$) via the $k$-valued point $1_{H[p^m]}$ (resp. $1_{M/p^mM}$).
	
	(c) Each $k$-valued point of $\Ends(H[p^m])_{\textrm{crys}}^0$ is a nilpotent endomorphism of $M/p^mM$.
\end{cor}

\begin{proof}
	\cite[Corollary 6]{Vasiu:levelm}.
\end{proof}
\section{Endomorphisms of truncated Barsotti--Tate groups}

Let $H_1$ and $H_2$ be two $p$-divisible groups over $k$. Let $(M_1,F,V)$ and $(M_2,F,V)$ be the Dieudonn\'e modules of $H_1$ and $H_2$ respectively. Let $H_{H_1,H_2}=\Hom_{W(k)}(M_1,M_2)$. Let $(\Hom_{W(k)}(M_1,M_2)[\frac{1}{p}],F)$ be the $F$-isocrystal defined by the rule
$$F(g)=F\circ g\circ F^{-1} = V^{-1}\circ g\circ V.$$ 
The classification of $F$-isocrystals over $k$ (see \cite[Ch.2, Sect. 4]{Manin:formalgroups}, \cite{Demazure1}, etc.) implies that we have a direct sum decomposition $\Hom_{W(k)}(M_1,M_2)[\frac{1}{p}]=\oplus_{\alpha\in\QQ}W(\alpha)$ that is left invariant by $F$ and that has the property that all Newton polygon slopes of $(W(\alpha), F)$ are $\alpha$. Write 
$$\Hom_{W(k)}(M_1,M_2)[\frac{1}{p}]=N_+\oplus N_0\oplus N_-$$
where all slopes of $(N_+,F)$ are positive, all slopes of $(N_0,F)$ are $0$ and all slopes of $(N_-,F)$ are negative.

\begin{lemma}[Vasiu]\label{keylemma}
	Let $m$ be a positive integer. Each homomorphism of truncated Dieudonn\'e modules
	$$\xi_m: (M_1/p^mM_1,F,V)\rightarrow(M_2/p^mM_2,F,V)$$
	can be lifted to a $W(k)$-linear map $\xi: M_1\rightarrow M_2$ such that 
	$F(\xi)-\xi\in p^m\Hom_{W(k)}(M_1,M_2)$.
\end{lemma}
\begin{proof}
	\cite[p.~822]{traversosolved}

\end{proof}

Let $h=c+d>0$. Let $\pi$ be a permutation on $J=\{1,2,\dots,h\}$. Recall that $H_\pi$ is the $p$-divisible group whose Dieudonn\'e module $(M,\varphi_\pi)$  is defined as follows.  Let $\{e_1, e_2,\dots,e_h\}$ be a $W(k)$-basis of $M$. We have
$$\varphi_\pi(e_i)=\left\{\begin{array}{lr}
pe_{\pi(i)},\quad& 1\leq i\leq d; \\
e_{\pi(i)},\quad& d<i\leq h.
\end{array}
\right.$$
Therefore we have $M=F^1\oplus F^0$ where $F^1=\oplus_{i=1}^{d} W(k)e_i$ and  $F^0=\oplus_{i=d+1}^{h} W(k)e_i$.

Let $e_{i,j}: M\rightarrow M$ be the $W(k)$-linear map such that for each $l\in J$ we have $e_{i,j}(e_l)=\delta_{j,l}e_i$. Then $\{e_{i,j}| i,j\in J\}$ form a $W(k)$-basis of $\End(M)$.

Recall that $\varphi_\pi$ acts on $\End_{W(k)}(M)[\frac{1}{p}]$ $\sigma$-linearly by the rule $\varphi_\pi(g)=\varphi_\pi\circ g\circ\varphi_\pi^{-1}$. In particular, we have

\[
\varphi_\pi(e_{i,j})=\left\{
\begin{array}{ll}
pe_{\pi(i),\pi(j)} &\textrm{if } (i,j)\in J_+,\\
e_{\pi(i),\pi(j)} &\textrm{if } (i,j)\in J_0,\\
\frac{1}{p}e_{\pi(i),\pi(j)} &\textrm{if } (i,j)\in J_-.
\end{array}
\right.
\]
where $J_+=\{(i,j)\in J^2| i\leq d < j\}$, $J_0=\{(i,j)\in J^2| i,j\leq d \textrm{ or } i,j>d\}$ and $J_-=\{(i,j)\in J^2| j\leq d < i\}$, as illustrated by the following diagram 

\[ \left( \begin{array}{c:c} J_0 &  J_+ \\ \hdashline
J_- & J_0 \end{array} \right). \] 

Let
\[
\varepsilon_{i,j}=\left\{
\begin{array}{ll}
1 &\textrm{if } (i,j)\in J_+,\\
0 &\textrm{if } (i,j)\in J_0,\\
-1 &\textrm{if } (i,j)\in J_-.
\end{array}
\right.
\]
Then we can write
\[\varphi_\pi(e_{i,j})=p^{\varepsilon_{i,j}}e_{\pi(i),\pi(j)}.\]

Now let $g\in \End_{W(k)}(M)$, write $g=\sum_{(i,j)\in J^2} \underline{x}_{i,j}e_{i,j}$ with $\underline{x}_{i,j}\in W(k)$. We have 
$$\varphi_\pi(g)=\sum_{(i,j)\in J^2} \sigma(\underline{x}_{i,j})p^{\varepsilon_{i,j}}e_{\pi(i),\pi(j)}.$$

In order for $\varphi_\pi(g)\in \End_{W(k)}(M)$, we need $\sigma(\underline{x}_{i,j})\in pW(k)$, or equivalently $\underline{x}_{i,j}\in pW(k)$, for all $(i,j)\in J_-$. Write

\[
\mu_{i,j}=\left\{
\begin{array}{ll}
1& \text{if } (i,j)\in J_-,\\
0& \text{if } (i,j)\in J_+\cup J_0. 
\end{array}
\right.
\]

Then we have $\varphi_\pi(g)\in \End_{W(k)}(M)$ if and only if $g=\sum_{(i,j)\in J^2} p^{\mu_{i,j}}\underline{x}_{i,j}e_{i,j}$, where $\underline{x}_{i,j} \in W(k)$ is arbitrary.

Now let $g\in \End_{W(k)}(M)$ be such that $\varphi_\pi(g)-g\in p^m\End_{W(k)}(M)$, then we have 
\[\sum_{(i,j)\in J^2} p^{\mu_{i,j}+\varepsilon_{i,j}}\sigma(\underline{x}_{i,j})e_{\pi(i),\pi(j)}\equiv \sum_{(i,j)\in J^2} p^{\mu_{i,j}}\underline{x}_{i,j}e_{i,j}\ \mod p^m.\]
Thus by comparing coefficients, we get
\begin{equation}\label{eq1}
p^{\mu_{i,j}+\varepsilon_{i,j}}\sigma(\underline{x}_{i,j})\equiv  p^{\mu_{\pi(i),\pi(j)}}\underline{x}_{\pi(i),\pi(j)} \mod p^m,\quad \forall (i,j)\in J^2.
\end{equation}

\subsection{Endomorphisms of $H_\pi[p]$}

When $m=1$, the equation \ref{eq1} becomes

\begin{equation}\label{eq2}
p^{\mu_{i,j}+\varepsilon_{i,j}} x_{i,j}^p=  p^{\mu_{\pi(i),\pi(j)}}x_{\pi(i),\pi(j)}, \quad \forall (i,j)\in J^2
\end{equation}

where $x_{i,j} \in W_1(k)=k$. 

We can break this system of equations into smaller systems, one for each orbit of $(\pi,\pi)$ on $J^2$. Let $\mathcal{O}=((i_1,j_1), (i_2,j_2), \dots, (i_l,j_l))$ be such an orbit, i.e., $ (\pi(i_1),\pi(j_1))=(i_{2},j_{2}), (\pi(i_2),\pi(j_2))=(i_{3},j_{3}), \dots, (\pi(i_l),\pi(j_l))=(i_1,j_1)$. For simplicity we write $x_s=x_{i_s,j_s}$, $ \varepsilon_s=\varepsilon_{i_s,j_s}$ and $\mu_s=\mu_{i_s,j_s}$. Let $\varepsilon_{\mathcal{O}}=(\varepsilon_1,\varepsilon_2,\dots, \varepsilon_l)$
and $\mu_{\mathcal{O}}=(\mu_1,\mu_2,\dots, \mu_l)$. Here and in what follows all indices are taken modulo the size of the orbit $|\mathcal{O}|=l$.

If $|\mathcal{O}|=1$, then we have $x_1^p=x_1$ when $\varepsilon_1=0$, or $x_1=0$ when $\varepsilon_1=\pm 1$. The former contributes a factor of $p$ to the number of connected components of $\Ends(H_\pi[p])$.

Assume $|\mathcal{O}|\geq 2$. Let us first examine the case where $\mathcal{O}\subseteq J_0$. In this case, we have $\varepsilon_\mathcal{O}=(0,0,\dots,0)$ and $\mu_\mathcal{O}=(0,0,\dots,0)$. The equation \ref{eq2} becomes
\[x_1^p=x_2,\]
\[x_2^p=x_3,\]
\[\cdots\]
\[x_l^p=x_1.\]
Thus we have $x_1^{p^l}=x_1$. This equation has $p^l$ distinct solutions in $k$. The polynomial $x_1^{p^l}-x_1$ give rise to $p^l$ connected components of $\Ends(H_\pi[p])$.

Now assume $\mathcal{O}\nsubseteq J_0$. Then $\varepsilon_\mathcal{O}$ contains non-zero numbers, say $\varepsilon_s$. If $\varepsilon_s=1$, then we have $0=p^{\mu_{s+1}}x_{s+1}$. Unless $\mu_{s+1}=1$ (i.e., $\varepsilon_{s+1}=-1$), we get $x_{s+1}=0$. Let $u$ be the smallest positive integer such that $\varepsilon_{s+u}=-1$ or $\infty$ if no such integer exists, then by repeating the previous argument, we get that $x_{s+i}=0$ for all $0<i<u$.

Similarly, if $\varepsilon_s=-1$, then we have $p^{\varepsilon_{s-1}}x_{s-1}^p=0$. Unless $\varepsilon_{s-1}=1$, we get $x_{s-1}=0$. Let $v$ be the smallest positive integer such that $\varepsilon_{s-v}=1$ or $\infty$ if no such integer exists, then by repeating the previous argument, we get that $x_{s-i}=0$ for all $0<i<v$.

Therefore, only segments of the form $<-1, 0, \dots, 0, 1>$ in $\varepsilon_\mathcal{O}$ produce non-trivial equations. Let $<\varepsilon_s, \dots, \varepsilon_t>$ be such a segment, then we have 
\begin{align*}
	x_s^p &=x_{s+1},\\
	x_{s+1}^p &=x_{s+2},\\
	&\cdots ,\\
	x_{t-1}^p &=x_{t},\\
	0 &= p^{\mu_{t+1}}x_{t+1},\\
	&\cdots	.
\end{align*}

Therefore we have $x_{s+1} = x_s^p, \dots, x_{t}=x_s^{p^{t-s}}$ and a free variable $x_s$ which adds $1$ to $\gamma_{H_\pi}(1)$.

To sum up, we have the following 

\begin{thm}
	Let the notations be as before. Then
	
	(1) $\gamma_{H_\pi}(1)$ is the number of segments of the type $<-1,0,\dots,0,1>$ that appear in $\varepsilon_{\mathcal{O}}$ as $\mathcal{O}$ ranges through all the orbits.
	
	(2)
	The number of connected components of $\Ends(H_\pi[p])$ is $p^{c_1}$ where $c_1=\sum_{\varepsilon_{\mathcal{O}}=(0,0,\dots,0)}|\mathcal{O}| $.
\end{thm}

Part (1) of this theorem is a special case of \cite[1.2. Basic Theorem A]{Vasiu:modpshimura}.

\begin{exam}
	In this example all induces $i,j\in J=\{1,2,\dots,h\}$ are taken modulo $h$. Suppose that $g.c.d(c,d)=1$ and that $H$ is minimal in the sense of \cite[Sect. 1.1]{Oort:minimal}. We can assume that we have $\pi(i)=i+d$ for all $i\in J$, cf. \cite[Sect 1.4]{Oort:minimal}. The $p$-divisible group $H$ is uniquely determined up to isomorphism by $H[p]$, cf. \cite[Example 3.3.6]{CBP}. Therefore we can assume $(M,\varphi)=(M,\varphi_\pi)$.
	
	In particular, let $c=2$, $d=3$. Then $\pi$ has five orbits on $J^2$. They are
	\begin{align*}
	\mathcal{O}_1 &= [(1,1)]=((1,1),(4,4),(2,2),(5,5),(3,3)), \varepsilon_{\mathcal{O}_1}=(0,0,0,0,0);\\
	\mathcal{O}_2 &= [(1,2)]=((1,2),(4,5),(2,3),(5,1),(3,4)), \varepsilon_{\mathcal{O}_2}=(0,0,0,-1,1);\\
	\mathcal{O}_3 &= [(1,3)]=((1,3),(4,1),(2,4),(5,2),(3,5)), \varepsilon_{\mathcal{O}_3}=(0,-1,1,-1,1);\\
	\mathcal{O}_4 &= [(1,4)]=((1,4),(4,2),(2,5),(5,3),(3,1)), \varepsilon_{\mathcal{O}_4}=(1,-1,1,-1,0);\\
	\mathcal{O}_5 &= [(1,5)]=((1,5),(4,3),(2,1),(5,4),(3,2)), \varepsilon_{\mathcal{O}_5}=(1,-1,0,0,0).
	\end{align*}
	Therefore $\gamma_H(1)=0+1+2+2+1=6$. 
	
	In general, for minimal $p$-divisible groups $H$ we have $\gamma_H(1)=cd.$
\end{exam}

\subsection{Endomorphisms of $H_\pi[p^2]$}

Now let $m=2$. Recall that for $\underline{x}=(a,b)\in W_2(k)$, we have $\sigma(\underline{x})=(a^p,b^p)$ and $p\underline{x}=(0,a^p)$. Write $\underline{x}_{i,j}=(a_{i,j}, b_{i,j})$, the equation \ref{eq2} becomes

\begin{equation}\label{eq3}
p^{\mu_{i,j}+\varepsilon_{i,j}} (a_{i,j}^p,b_{i,j}^p)=  p^{\mu_{\pi(i),\pi(j)}}(a_{\pi(i),\pi(j)}, b_{\pi(i),\pi(j)}),\quad \forall (i,j)\in J^2
\end{equation}
where $a_{i,j},b_{i,j}\in k.$

As before, let $\mathcal{O}=((i_1,j_1), (i_2,j_2), \dots, (i_l,j_l))$ be an orbit. For simplicity we write $\varepsilon_s=\varepsilon_{i_s,j_s}$, $\mu_s=\mu_{i_s,j_s}$, $a_s=a_{i_s,j_s}$ and $b_s=b_{i_s,j_s}$, with the understanding that all indices are taken modulo $|\mathcal{O}|$. Let $\varepsilon_{\mathcal{O}}=(\varepsilon_1,\varepsilon_2,\dots, \varepsilon_l)$ and $\mu_{\mathcal{O}}=(\mu_1,\mu_2,\dots, \mu_l)$. 

If $|\mathcal{O}|=1$, then we have
\begin{align*}
	(a_1^p,b_1^p) &= (0, a_1^p),\hspace{1cm} \textrm{ if } \varepsilon_1=-1,\\
	(a_1^p,b_1^p) &= (a_1, b_1),\hspace{1cm} \textrm{ if } \varepsilon_1=0,\\
	(0,a_1^{p^2}) &= (a_1, b_1),\hspace{1cm} \textrm{ if } \varepsilon_1=1.
\end{align*}	

Therefore, we get $a_1=b_1=0$ when $\varepsilon_1=\pm 1$, and $a_1^p=a_1, b_1^p=b_1$ when $\varepsilon_1=0$. The latter contributes a factor of $p^2$ to the number of connected components of $\Ends(H_\pi[p^2])$.

Now assume $|\mathcal{O}|\geq 2$. The relation between $(a_s,b_s)$ and $(a_{s+1},b_{s+1})$ is given by $\varepsilon_s$ and $\mu_{s+1}$. By checking all six possible combinations of $\varepsilon_s$ and $\mu_{s+1}$, we get the following four disjoint cases. 

\begin{itemize}
	
	\item If $\varepsilon_s \in\{-1, 0\}$ and $\mu_{s+1}=0$, then $$(a_s^p,b_s^p)=(a_{s+1},b_{s+1}).$$	
	
	\item If $\varepsilon_s=1$ and $\mu_{s+1}=1$, then $$(0,a_s^{p^2})=(0,a_{s+1}^p).$$ 
	
	\item If $\varepsilon_s=1$ and $\mu_{s+1}=0$, then $$(0,a_s^{p^2})=(a_{s+1},b_{s+1}).$$
	
	\item If $\varepsilon_s \in \{-1, 0\}$ and $\mu_{s+1}=1$, then $$(a_s^p,b_s^p)=(0,a_{s+1}^p).$$	
\end{itemize}

Since $\mu_{s+1}$ depends on $\varepsilon_{s+1}$, these relations ultimately come from $\varepsilon_\mathcal{O}$. If we draw a line $x-y$ when $y=x$, a single headed arrow $x\rightarrow y$ when $y=x^p$ and a double headed arrow $x\twoheadrightarrow y$ when $y=x^{p^2}$, then we can use the following four diagrams to represent the four types of relations between $(a_s,b_s)$ and $(a_{s+1},b_{s+1})$ according to $<\varepsilon_s, \varepsilon_{s+1}>$:
\begin{itemize}
	
	\item $<-1,0>$ or $<-1,1>$ or $<0,0>$ or $<0,1>$:
	$$
	\xymatrix{
		\ar@{--}[r] &a_s \ar[r] &a_{s+1} \ar@{--}[r] & \mathbin{\phantom{0}} \\
		\ar@{--}[r] &b_s \ar[r] &b_{s+1} \ar@{--}[r] & \mathbin{\phantom{0}}
	}
	$$
	
	\item $<1,-1>$:
	$$
	\xymatrix{
		\ar@{--}[r] &a_s \ar[r] &a_{s+1} \ar@{--}[r] & \mathbin{\phantom{0}} \\
		\ar@{--}[r] &b_s &b_{s+1} \ar@{--}[r] & \mathbin{\phantom{0}}
	}
	$$
	
	\item $<1,0>$ or $<1,1>$:
	$$
	\xymatrix{
		\ar@{--}[r] &a_s \ar@{->>}[dr] & 0 \ar@{--}[r] & \mathbin{\phantom{0}} \\
		\ar@{--}[r] &b_s  &b_{s+1} \ar@{--}[r] & \mathbin{\phantom{0}}
	}
	$$
	\item $<-1,-1>$ or $<0,-1>$:
	$$
	\xymatrix{
		\ar@{--}[r] & 0  & a_{s+1} \ar@{--}[r] & \mathbin{\phantom{0}} \\
		\ar@{--}[r] &b_s\ar@{-}[ur]  &b_{s+1} \ar@{--}[r] & \mathbin{\phantom{0}}
	}
	$$
\end{itemize}

By putting these diagrams together, we can represent the equations \ref{eq3} on $\mathcal{O}$ by a graph $\Gamma_\mathcal{O}$. Here and in what follows by graph we mean an oriented graph which has some vertices marked as $0$ and which has multiple types of edges (like $-$, $\rightarrow$, $\twoheadrightarrow$, etc.).

\begin{exam}\label{diagrams}
	(a) Let $\varepsilon_\mathcal{O}=(-1,-1,1,1)$. Then $\Gamma_\mathcal{O}$ is:
	
		$$
		
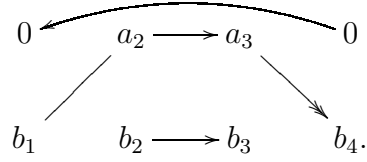
\xymatrix{
			0 	 &a_2 \ar[r] & a_3 \ar@{->>}[rd] & 0 \ar@/_1pc/[lll] \\
			b_1	 \ar@{-}[ur] &b_2 \ar[r]  &b_3 & b_4.
		}
		$$\\
		
		Therefore we get $a_1=a_4=0$ and two free variables $b_1$ and $b_2$.
		
	(b)	Let $\varepsilon_\mathcal{O}=(-1,1,-1,1)$. Then $\Gamma_\mathcal{O}$ is:
		$$
		
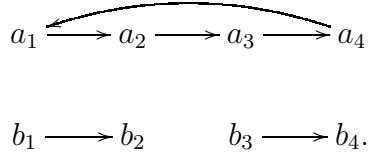
\xymatrix{
			a_1\ar[r] 	 &a_2 \ar[r] & a_3 \ar[r] & a_4 \ar@/_1pc/[lll] \\
			b_1	 \ar[r] &b_2   &b_3\ar[r] & b_4.
		}
		$$
			
		There are two free variables $b_1$ and $b_3$. Combining the four equations in the first row, we get $ a_1=a_1^{p^4}$.
	
	(c)	Let $\varepsilon_\mathcal{O}=(0,0,0,0)$. Then $\Gamma_\mathcal{O}$ is:
	
		$$
		
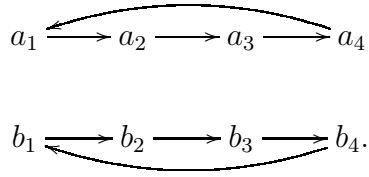
\xymatrix{
			a_1\ar[r] 	 &a_2 \ar[r] & a_3 \ar[r] & a_4 \ar@/_1pc/[lll] \\
			b_1	 \ar[r] &b_2 \ar[r]   &b_3\ar[r] & b_4 . \ar@/^1pc/[lll]
		}
		$$\\
		
	This gives $a_1=a_1^{p^4}$ and $b_1=b_1^{p^4}$.	
	
	(d)	Let $\varepsilon_\mathcal{O}=(-1,0,-1,-1,1,1,0,1)$. We use $\ast$, $\circ$ and $\bullet$ to represent variables. We use $\circ$ for a variable $y_s$ that doesn't depend on $a_{s-1}$ or $b_{s-1}$ (an ``open start''), use $\bullet$ for a variable $y_s$ that has no relation with $a_{s+1}$ or $b_{s+1}$ (an ``open end''), and use $\ast$ for other variables. Then we can draw $\Gamma_\mathcal{O}$ as follows:
	
	$$
	
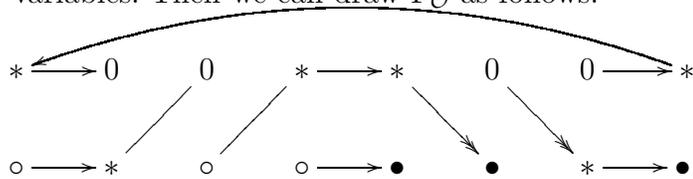
\xymatrix{
		\ast \ar[r] &0 &0	 &\ast \ar[r] & \ast \ar@{->>}[rd] & 0\ar@{->>}[rd] & 0 \ar[r] &\ast \ar@/_2pc/[lllllll] \\
		\circ \ar[r] & \ast\ar@{-}[ur] &\circ \ar@{-}[ur] &\circ \ar[r] &\bullet &\bullet  &\ast \ar[r] & \bullet.
	}
	$$\\
	
\end{exam}

From above examples, we can see that for an orbit $\mathcal{O}$ with $|\mathcal{O}|\geq 2$, the graph $\Gamma_\mathcal{O}$ further decomposes into linear graphs and circular graphs which are connected components. It is clear that a zero vertex in a linear graph will make all variables zero, and a zero vertex cannot belong to a circular graph. We call linear graphs that have no zero vertices \emph{free} linear graphs.

\begin{lemma}\label{dc} Let $\mathcal{O}$ and $\Gamma_\mathcal{O}$ be as before. Then we have the following.
	
	(a) Each circular graph in $\Gamma_\mathcal{O}$ contributes a factor of $p^{|\mathcal{O}|}$ to the number of connected components of $\Ends(H_\pi[p^2])$.
	
	(b) Each free linear graph in $\Gamma_\mathcal{O}$ adds $1$ to the dimension of $\Auts(H_\pi[p^2])$. 
\end{lemma}	

\begin{proof}
	(a) Take a circular graph in $\Gamma_\mathcal{O}$. If we assign weight $0$ to arrows with no heads, $1$ to single-headed arrows and $2$ to double-headed arrows, then we get an equation $y_s^{p^w}=y_s$, where $w$ is the sum of the weights of all arrows in the circular graph. Since the number of weight $0$ arrows is equal to the number of weight $2$ arrows in a circular graph, we have that $w=|\mathcal{O}|$.  
	
	(b) It is obvious that a free linear graph produces $1$ free variable, namely, the one that corresponds to the open start of the free linear graph.
\end{proof}	

To track the free linear graphs in $\Gamma_\mathcal{O}$, we make the following observation.

1. All zero vertices are in the first row. All free linear graphs start and end in the second row. Moreover, let $b_s$ (resp. $b_t$) be the variable corresponding to the start (resp. the end) of the
free linear graph, then we have $\varepsilon_s=-1$ and $\varepsilon_t=1$.

 Indeed, the only way to get an open start ($\circ_{b_s}$) is to use the following two diagrams. They correspond to $\varepsilon_s=-1$

	$$
	
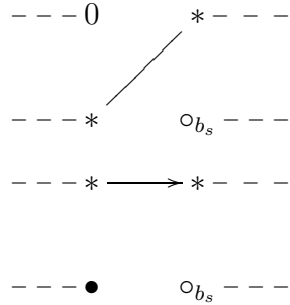
\xymatrix{
		\ar@{--}[r] & 0  & \ast \ar@{--}[r] & \mathbin{\phantom{0}} \\
		\ar@{--}[r] &\ast\ar@{-}[ur]  &\circ_{b_s} \ar@{--}[r] & \mathbin{\phantom{0}}
	}
	$$
	$$
	\xymatrix{
		\ar@{--}[r] &\ast \ar[r] &\ast \ar@{--}[r] & \mathbin{\phantom{0}} \\
		\ar@{--}[r] &\bullet &\circ_{b_s} \ar@{--}[r] & \mathbin{\phantom{0}}
	}
	$$
	
Similarly, the only way to get an open end ($\bullet_{b_t}$) is to use the following two diagrams. They correspond to $\varepsilon_t=1$.

$$
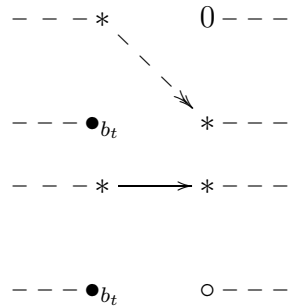
\xymatrix{
	\ar@{--}[r] & \ast \ar@{-->>}[dr] & 0 \ar@{--}[r] & \mathbin{\phantom{0}} \\
	\ar@{--}[r] &\bullet_{b_t} &\ast \ar@{--}[r] & \mathbin{\phantom{0}}
}
$$
$$
\xymatrix{
	\ar@{--}[r] &\ast \ar[r] &\ast \ar@{--}[r] & \mathbin{\phantom{0}} \\
	\ar@{--}[r] &\bullet_{b_t} &\circ \ar@{--}[r] & \mathbin{\phantom{0}}
}
$$

2. Let us study the effect of $0$'s in $\varepsilon_\mathcal{O}$. Take a segment of the form $<\varepsilon_{i-1}, 0, \varepsilon_{i+1}>$. By combining the diagrams earlier, we get four cases depending on $<\varepsilon_{i-1},0,\varepsilon_{i+1}>$:

\begin{itemize}
	\item $<-1,0,0>$ or $<-1,0,1>$ or $<0,0,0>$ or $<0,0,1>$:
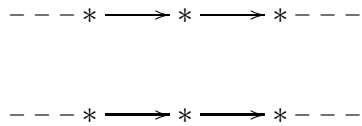
	$$
	\xymatrix{
		\ar@{--}[r] &\ast \ar[r] &\ast \ar[r] &\ast \ar@{--}[r] & \mathbin{\phantom{0}} \\
		\ar@{--}[r] &\ast \ar[r] &\ast \ar[r] &\ast \ar@{--}[r] & \mathbin{\phantom{0}}
	}
	$$
	
	\item $<1,0,-1>$:
	
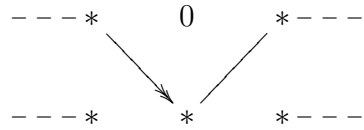
$$
	\xymatrix{
		\ar@{--}[r] &\ast \ar@{->>}[rd] &0 &\ast \ar@{--}[r] & \mathbin{\phantom{0}} \\
		\ar@{--}[r] &\ast &\ast\ar@{-}[ur] &\ast \ar@{--}[r] & \mathbin{\phantom{0}}
	}
	$$
	\item $<1,0,0>$ or $<1,0,1>$:
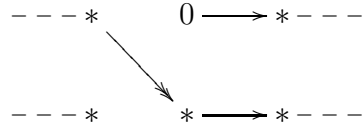
	$$
	\xymatrix{
		\ar@{--}[r] &\ast \ar@{->>}[dr] & 0 \ar[r] &\ast \ar@{--}[r] & \mathbin{\phantom{0}} \\
		\ar@{--}[r] &\ast  &\ast \ar[r] &\ast \ar@{--}[r] & \mathbin{\phantom{0}}
	}
	$$
	\item $<-1,0,-1>$ or $<0,0,-1>$:
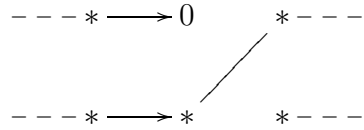
	$$
	\xymatrix{
		\ar@{--}[r] &\ast\ar[r] & 0  & \ast \ar@{--}[r] & \mathbin{\phantom{0}} \\
		\ar@{--}[r] &\ast\ar[r] &\ast \ar@{-}[ur]  &\ast \ar@{--}[r] & \mathbin{\phantom{0}}
	}
	$$
	
\end{itemize}	
	
By Comparing these diagrams with the diagrams of $<\varepsilon_{i-1},\varepsilon_{i+1}>$:
\begin{itemize}
	
	\item $<-1,0>$ or $<-1,1>$ or $<0,0>$ or $<0,1>$:
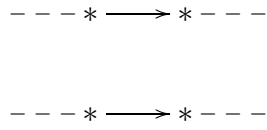
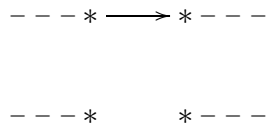
	$$
	\xymatrix{
		\ar@{--}[r] &\ast \ar[r] &\ast \ar@{--}[r] & \mathbin{\phantom{0}} \\
		\ar@{--}[r] &\ast \ar[r] &\ast \ar@{--}[r] & \mathbin{\phantom{0}}
	}
	$$
	
	\item $<1,-1>$:
	$$
	\xymatrix{
		\ar@{--}[r] &\ast \ar[r] &\ast \ar@{--}[r] & \mathbin{\phantom{0}} \\
		\ar@{--}[r] &\ast &\ast \ar@{--}[r] & \mathbin{\phantom{0}}
	}
	$$
	
	\item $<1,0>$ or $<1,1>$:
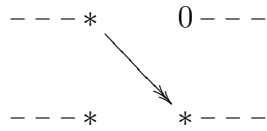
	$$
	\xymatrix{
		\ar@{--}[r] &\ast \ar@{->>}[dr] & 0 \ar@{--}[r] & \mathbin{\phantom{0}} \\
		\ar@{--}[r] &\ast  &\ast \ar@{--}[r] & \mathbin{\phantom{0}}
	}
	$$
	\item $<-1,-1>$ or $<0,-1>$:
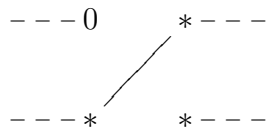
	$$
	\xymatrix{
		\ar@{--}[r] & 0  &\ast \ar@{--}[r] & \mathbin{\phantom{0}} \\
		\ar@{--}[r] &\ast\ar@{-}[ur]  &\ast \ar@{--}[r] & \mathbin{\phantom{0}}
	}
	$$
\end{itemize}
	we see that deleting the zeros from $\varepsilon_\mathcal{O}$ does not affect the type of graphs we get from $\Gamma_\mathcal{O}$ (it will make some graphs shorter of course). To keep track of the number of free linear graphs in $\Gamma_\mathcal{O}$, we can therefore ignore the zeros in $\varepsilon_\mathcal{O}$.

Assume that $\varepsilon_\mathcal{O}$ has no zero.  Let $b_s$ be a free variable. Let us examine what the segment $<\varepsilon_s,\varepsilon_{s+1},\dots,\varepsilon_t>$ that corresponds to the linear graph staring with $b_s$ is like.

We know that $\varepsilon_s=-1$. If $\varepsilon_{s+1}=1$, then the segment ends and we get a free linear graph from $<\varepsilon_s,\varepsilon_{s+1}>=<-1,1>$:
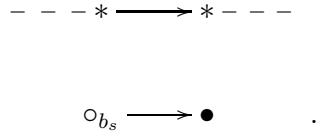
$$
\xymatrix{
	\ar@{--}[r] &\ast \ar[r] &\ast \ar@{--}[r] & \mathbin{\phantom{0}} \\
	 &\circ_{b_s} \ar[r] &\bullet  & \mathbin{\phantom{0}}.
}
$$

If $\varepsilon_{s+1} = -1$, then we have 

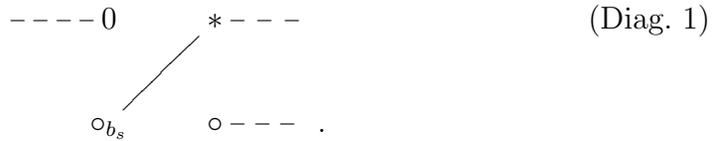
\begin{equation*}\tag{Diag. 1}
\xymatrix{
\ar@{--}[r]&	0 	 &\ast \ar@{--}[r] & \mathbin{\phantom{0}}   \\
 \mathbin{\phantom{0}}	&\circ_{b_s}	 \ar@{-}[ur] & \circ \ar@{--}[r] & \mathbin{\phantom{0}} .
}
\end{equation*}

In order for $b_s$ to be free we must have $\varepsilon_{s+2}=1$, for three consecutive $-1$ produces a diagram 
	$$
	
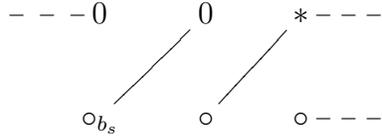
\xymatrix{
		\ar@{--}[r] &0 & 0  & \ast \ar@{--}[r] & \mathbin{\phantom{0}} \\
		 &\circ_{b_s} \ar@{-}[ur] &\circ \ar@{-}[ur]  &\circ \ar@{--}[r] & \mathbin{\phantom{0}}
	}
	$$
	forcing $b_s=0$.
	
	For $\varepsilon_{s+3}$, we can have two choices. If $\varepsilon_{s+3}=1$, then the segment ends and we have a free linear graph from $<\varepsilon_s,\varepsilon_{s+1},\varepsilon_{s+2},\varepsilon_{s+3}>=<-1,-1,1,1>$:
	
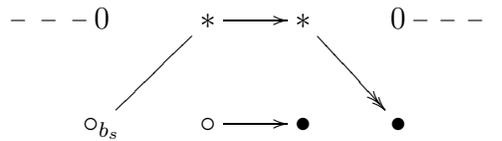
	$$
	\xymatrix{
	\mathbin{\phantom{0}} \ar@{--}[r]	&0 	 &\ast \ar[r] & \ast \ar@{->>}[rd] & 0 \ar@{--}[r] &\mathbin{\phantom{0}} \\
	\mathbin{\phantom{0}}	&\circ_{b_s}	 \ar@{-}[ur] &\circ \ar[r]  &\bullet & \bullet &\mathbin{\phantom{0}}.
	}
	$$
	(We have another one but that is counted by the segment $<\varepsilon_{s+1},\varepsilon_{s+2}>=<-1,1>$.)
	
	If $\varepsilon_{s+3}=-1$. Then we have a segment $<-1,-1,1,-1,\varepsilon_{s+4},\dots>$, which produce the diagram
	\begin{equation*}\tag{Diag. 2}
		
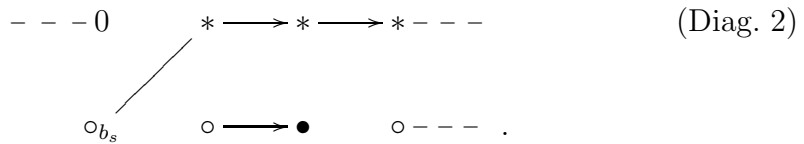
\xymatrix{
		\mathbin{\phantom{0}}\ar@{--}[r]	&0 	 &\ast \ar[r] & \ast \ar[r] & \ast \ar@{--}[r] &\mathbin{\phantom{0}}\\
		\mathbin{\phantom{0}}	&\circ_{b_s}	 \ar@{-}[ur] &\circ \ar[r]  &\bullet & \circ \ar@{--}[r] &\mathbin{\phantom{0}}.
		}
	\end{equation*}
	
	Comparing Diag. 1 with Diag. 2, we see that deleting $\varepsilon_{s+2}=1$ and $\varepsilon_{s+3}=-1$ does not affect the type of the graph starting with $b_s$. Therefore we can remove $\varepsilon_{s+2}$ and $\varepsilon_{s+3}$ and examine $<-1,-1,\varepsilon_{s+4},\dots>$.
	
	Repeating the arguments, we see that we can have some pairs of 1 and -1, then eventually two consecutive $1$'s to terminate the segment.
	
Therefore, if $b_s$ be a free variable, then the segment $<\varepsilon_s,\varepsilon_{s+1},\dots,\varepsilon_t>$ that corresponds to the free linear graph staring with $b_s$ is either $<-1,1>$ or $<-1,-1,\dots,1,1>$ where $\dots$ can be any number of pairs of 1 and -1.	
	
Factoring in the zeros we could have in $\varepsilon_\mathcal{O}$, we can see that such a segment  $<\varepsilon_s,\varepsilon_{s+1},\dots,\varepsilon_t>$ can be described by the following three properties: 

\begin{itemize}
	\item We have $\varepsilon_s=-1, \varepsilon_{t}=1$.
	\item We have $\sum_{i=s}^t \varepsilon_{i} = 0.$
	\item For all $s\leq j<t$, we have $-2\leq \sum_{i=s}^j \varepsilon_{i} < 0$.
\end{itemize}

This motivates the following definition.

\begin{defn}\label{fls}
Let $\mathcal{O}=((i_1,j_1), (i_2,j_2), \dots, (i_l,j_l))$ be an orbit and $\varepsilon_{\mathcal{O}}=(\varepsilon_1,\varepsilon_2,\dots,\varepsilon_l)$ as before. Let $n\in \NN^\ast$. A segment $<\varepsilon_s,\varepsilon_{s+1},\dots,\varepsilon_{t}>$ of $\varepsilon_{\mathcal{O}}$ is called a \emph{free linear segment of level $n$} if it satisfies the following conditions.
\begin{itemize}
	\item We have $\varepsilon_s=-1, \varepsilon_{t}=1$.
	\item We have $\sum_{i=s}^t \varepsilon_{i} = 0$.
	\item For all $s\leq j<t$ , we have $-n\leq \sum_{i=s}^j \varepsilon_{i} < 0$.
	\item There exists a $j_0$, $s\leq j_0<t$, such that $\sum_{i=s}^{j_0} \varepsilon_{i}=-n$.
\end{itemize}
We write $a_n(\mathcal{O})$ for the number of free linear segments of level $n$ in $\varepsilon_{\mathcal{O}}$.
\end{defn}

Thus we have 

\begin{thm}\label{nfls}
	Let $\mathcal{O}$ and $\varepsilon_\mathcal{O}$ be as before. Then the number of free linear graphs in $\Gamma_\mathcal{O}$ is $a_1(\mathcal{O})+a_2(\mathcal{O})$.
\end{thm}

\begin{proof}
	From definition \ref{fls}, it is clear that no two free linear segments start with the same $\varepsilon_s$. The theorem follows from the above consideration.
\end{proof}

\begin{cor}
	Let $\pi$, $H_\pi$ be as before, then
	\[\gamma_{H_\pi}(2)=\sum_{\mathcal{O}}a_1(\mathcal{O})+a_2(\mathcal{O}).\]
\end{cor}	
\begin{proof}
	This is a direct consequence of Theorem \ref{nfls} and Lemma \ref{dc}.
\end{proof}		

Next we examine the circular graphs in $\Gamma_\mathcal{O}$. 

First note that an orbit $\mathcal{O}=(0,0,\dots,0)\subseteq J_0$ will produce two circular graphs, as shown below.

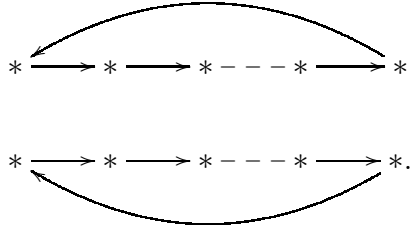
\[
\xymatrix{
	\ast \ar[r] 	 &\ast \ar[r] & \ast \ar@{--}[r] & \ast \ar[r] &\ast \ar@/_2pc/[llll] \\
	\ast	 \ar[r] &\ast \ar[r]   &\ast \ar@{--}[r] & \ast \ar[r] &\ast . \ar@/^2pc/[llll]
}
\]\\

Now assume $\mathcal{O}\nsubseteq J_0$. Note that in this case $\Gamma_\mathcal{O}$ can have at most $1$ circular graph, because each circular graph uses $|\mathcal{O}|$ distinct vertices and having $\pm 1$ in $\varepsilon_\mathcal{O}$ makes some vertices zero.

Ignore the zeros first in $\varepsilon_\mathcal{O}$ as before, we see that the orbits that produce circular graphs are such that $\varepsilon_\mathcal{O}=(-1,1,-1,1,\dots,-1,1)$, as shown below. 

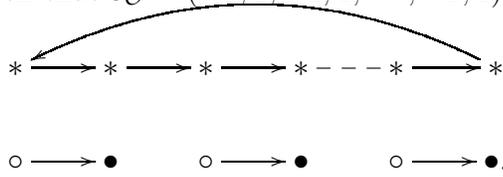
\[
\xymatrix{
	\ast \ar[r] 	 &\ast \ar[r] & \ast \ar[r] &\ast  \ar@{--}[r] & \ast \ar[r] &\ast \ar@/_2pc/[lllll] \\
	\circ	 \ar[r] &\bullet   &\circ\ar[r] &\bullet  & \circ \ar[r] &\bullet .
}
\]\\

Factoring in the zeros we could have in $\varepsilon_\mathcal{O}$ and taking into account the number of circular graphs, we make the following definition.

\begin{defn}\label{co}
	Let $\mathcal{O}=((i_1,j_1), (i_2,j_2), \dots, (i_l,j_l))$ be an orbit and $\varepsilon_{\mathcal{O}}=(\varepsilon_1,\varepsilon_2,\dots,\varepsilon_l)$ as before. Let $n\geq 0$. Then $\mathcal{O}$ is called a \emph{circular orbit of level $n$} if $\varepsilon_\mathcal{O}$ satisfies the following conditions.
	\begin{itemize}
		\item $\sum_{i=1}^l \varepsilon_i = 0.$
		
		\item For all $u,v\in \{1,2,\dots,l\}$, $ |\sum_{i=u}^v \varepsilon_i| \leq n$.
		
		\item There exist $u,v \in \{1,2,\dots,l\}$ such that $ |\sum_{i=u}^v \varepsilon_i| = n$.
	\end{itemize}
	We write $\mathcal{C}_\pi(n)$ for the set of circular orbits of level $n$.
\end{defn}

Therefore, circular orbits of level $0$ are those $\mathcal{O}$ with $\varepsilon_\mathcal{O}=(0,0,
\dots, 0)$. They each produce two circular graphs. Circular orbits of level $1$ are those $\mathcal{O}$ whose $\varepsilon_\mathcal{O}$ contains the same number of $-1$'s and $1$'s, alternating with $0$'s in between. They each produce one circular graph. Also, note that circular orbits of level $2$ or above have $2$ or more consecutive $-1$'s and  don't produce circular graphs.

\begin{thm}
	Let the notations be as before. Then the number of connected components of $\Ends(H_\pi[p^2])$ is $p^{c_2}$ where $$c_2=\sum_{\mathcal{O}\in\mathcal{C}_\pi(0)}2|\mathcal{O}|+\sum_{\mathcal{O}\in\mathcal{C}_\pi(1)}|\mathcal{O}| .$$
\end{thm}
\begin{proof}
	This is a direct consequence of Lemma \ref{dc}.
\end{proof}	

\subsection{Endomorphisms of $D_\pi[p^m]$}

We can add more rows to the diagrams to represent equations \ref{eq1} when $m\geq 3$. For example, when $m=3$, write $\underline{x}=(a,b,c)\in W_3(k)$. The effects of $-1$, $0$ and $1$ in $\varepsilon_\mathcal{O}$ can be represented by the following diagrams.

\begin{itemize}

	\item $<0,0>$ or $<-1,0>$ or $<0,1>$ or $<-1,1>$:
	$$
	\xymatrix{
		\ar@{--}[r] &\ast \ar[r] &\ast \ar@{--}[r] & \mathbin{\phantom{0}} \\
		\ar@{--}[r] &\ast \ar[r] &\ast \ar@{--}[r] & \mathbin{\phantom{0}}\\
		\ar@{--}[r] &\ast \ar[r] &\ast \ar@{--}[r] & \mathbin{\phantom{0}}
	}
	$$
	
	\item $<1,-1>$ :
	$$
	\xymatrix{
		\ar@{--}[r] &\ast \ar[r] &\ast \ar@{--}[r] & \mathbin{\phantom{0}} \\
		\ar@{--}[r] &\ast \ar[r] &\ast \ar@{--}[r] & \mathbin{\phantom{0}}\\
		\ar@{--}[r] &\bullet  &\circ \ar@{--}[r] & \mathbin{\phantom{0}}
	}
	$$
	
	\item $<-1,-1>$ or $<0,-1>$:
	$$
	\xymatrix{
		\ar@{--}[r] & 0  & \ast \ar@{--}[r] & \mathbin{\phantom{0}} \\
		\ar@{--}[r] &\ast \ar@{-}[ur]  &\ast \ar@{--}[r] & \mathbin{\phantom{0}}\\
		\ar@{--}[r] &\ast \ar@{-}[ur]  &\circ \ar@{--}[r] & \mathbin{\phantom{0}}
	}
	$$
	\item $<1,1>$ or $<1,0>$:
	$$
	\xymatrix{
		\ar@{--}[r] &\ast \ar@{->>}[dr] & 0 \ar@{--}[r] & \mathbin{\phantom{0}} \\
		\ar@{--}[r] &\ast \ar@{->>}[dr]  &\ast \ar@{--}[r] & \mathbin{\phantom{0}}\\
		\ar@{--}[r] &\bullet  &\ast \ar@{--}[r] & \mathbin{\phantom{0}}
	}
	$$
	
\end{itemize}

By combining the diagrams above and using a similar argument, we get that for $m=3$ each free linear segment of level $n\le 3$ produces a free linear diagram. In general, we get the following:

\begin{thm}\label{main1}
	Let $\pi$, $H_\pi$ be as before, then
	\[\gamma_{H_\pi}(m)=\sum_{\mathcal{O}}\sum_{n=1}^{m}a_n(\mathcal{O}).\]
	\qed
\end{thm}

As for the number of connected components, we note that each circular orbit of level $0$ produces $m$ circular graphs in $\Gamma_\mathcal{O}$, each circular orbits of level $1$ produces $m-1$ circular graphs, and each circular orbits of level $m-1$ produces $1$ circular graph. Circular orbits of level $m$ or above don't produce any circular graphs. Therefore, we have the following

\begin{thm}\label{main2}
	Let $\pi$, $H_\pi$ be as before, then the number of connected components of $\Ends(H_\pi[p^m])$ is $p^{c_m}$ where $$c_m=\sum_{n=0}^{m-1}\sum_{\mathcal{O}\in\mathcal{C}_\pi(n)}(m-n)|\mathcal{O}|.$$
	\qed
\end{thm}

\begin{cor}\label{cor}
	Let $\gamma_{H_\pi}(m)=\dim\Auts(H_\pi[p^m])$ be as before. Then
	
	(1) The sequence $(\gamma_{H_\pi}(m))_{m\geq 1}$ is increasing.
	
	(2) The sequence $(\gamma_{H_\pi}(m+1)-\gamma_{H_\pi}(m))_{m\geq 1}$ is decreasing.
\end{cor}	
\begin{proof}
	By Theorem \ref{main1} we have \[\gamma_{H_\pi}(m+1)-\gamma_{H_\pi}(m)=\sum_{\mathcal{O}}a_{m+1}(\mathcal{O}).\]
	Therefore, $\gamma_{H_\pi}(m+1)-\gamma_{H_\pi}(m)\geq 0$, which proves (1).

	For each orbit $\mathcal{O}$, let $\mathcal{L}_\mathcal{O}(m)$ be the set of free linear segments of level $m$ in $\varepsilon_\mathcal{O}$. Then by definition  $a_m(\mathcal{O})=|\mathcal{L}_\mathcal{O}(m)|$. We show that there is an injective map $\mathcal{L}_\mathcal{O}(m+1)\hookrightarrow \mathcal{L}_\mathcal{O}(m)$. 
	
	First we claim that the (indices of) segments in $\mathcal{L}_\mathcal{O}(m+1)$ are disjoint. If this is not the case, then since all indices are taken modulo $|\mathcal{O}|$ and no two free linear segments start or end at the same index, we can assume that there are two distinct segments $<\varepsilon_{s_1},\dots,\varepsilon_{t_1}>$ and $<\varepsilon_{s_2},\dots,\varepsilon_{t_2}>$ in $\mathcal{L}_\mathcal{O}(m+1)$ such that $s_1<s_2<t_2< t_1$ or $s_1< s_2<t_1<t_2$.
	
	Assume $s_1<s_2<t_2< t_1$. Since $<\varepsilon_{s_2},\dots,\varepsilon_{t_2}>$ is a segment in $\mathcal{L}_\mathcal{O}(m+1)$, there is $s_2<j_0<t_2$ such that $\sum_{i=s_2}^{j_0}\varepsilon_i=-(m+1)$. Then $\sum_{i=s_1}^{j_0}\varepsilon_i=\sum_{i=s_1}^{s_2-1}\varepsilon_i+\sum_{i=s_2}^{j_0}\varepsilon_i<0-(m+1)\le -(m+2)$, a contradiction to the fact that $<\varepsilon_{s_1},\dots,\varepsilon_{t_1}>$ is in $\mathcal{L}_\mathcal{O}(m+1)$.
	
	 Assume $s_1< s_2<t_1<t_2$. Since $\sum_{i=s_1}^{t_1}\varepsilon_i=0$ and  $\sum_{i=s_1}^{s_2-1}\varepsilon_i<0$, we must have $\sum_{i=s_2}^{t_1}\varepsilon_i>0$, which is not allowed. This proves the claim. 
	 
	 Now let $m\ge 2$ and let $<\varepsilon_{s},\varepsilon_{s+1},\dots,\varepsilon_{t-1},\varepsilon_{t}>$ be a segment in $\mathcal{L}_\mathcal{O}(m+1)$. We may assume $<\varepsilon_{s},\varepsilon_{s+1},\dots,\varepsilon_{t-1},\varepsilon_{t}> = <-1,-1,\dots,1,1>$. Remove $\varepsilon_s=-1$, $\varepsilon_t=1$ and consider the segment $<\varepsilon_{s+1},\dots,\varepsilon_{t-1}>$. This segment may decompose into several free linear segments. Let $s<j_0<t$ be such that $\sum_{i=s}^{j_0}\varepsilon_i=-(m+1)$. Since $\varepsilon_s=-1$, we have $\sum_{i=s+1}^{j_0}\varepsilon_i=-m$. Let $u$ be the maximum of the integers $s<j<j_0$ such that $\sum_{i=s+1}^{j}\varepsilon_i=0$ or $s$ if such $j$ does not exist.  Let $v$ be the minimum of the integers $j_0<j<t$ such that $\sum_{i=j}^{t-1}\varepsilon_i=0$ or $t$ if such $j$ does not exist. We claim that the segment $<\varepsilon_{u+1},\dots,\varepsilon_{v-1}>$ is a segment in $\mathcal{L}_\mathcal{O}(m)$.
	 
	 1. We check that $\sum_{i=u+1}^{v-1}\varepsilon_i=0$. Indeed, this follows from the facts $\sum_{i=s+1}^{t-1}\varepsilon_i=0$, $\sum_{i=s+1}^{u}\varepsilon_i=0$ and $\sum_{i=v}^{t-1}\varepsilon_i=0$.
	 
	 2. We check that for $u+1\le j<v-1$, we have $\sum_{i=u+1}^{j}\varepsilon_i<0$. Assume by contradiction that there is a $j$ such that $\sum_{i=u+1}^{j}\varepsilon_i\ge 0$. If $\sum_{i=u+1}^{j}\varepsilon_i> 0$, then $\sum_{i=s}^{j}\varepsilon_i =\varepsilon_s+ \sum_{i=s+1}^{u}\varepsilon_i+\sum_{i=u+1}^{j}\varepsilon_i=-1+0+\sum_{i=u+1}^{j}\varepsilon_i\ge 0$, which is not allowed. Now assume $\sum_{i=u+1}^{j}\varepsilon_i=0$. By the maximality of $u$ we must have $j\ge j_0$. Similarly, by the minimality of $v$ we must have $j+1\le j_0$. Such $j$ does not exists.
	 
	 3. Clearly, for $u+1\le j<v-1$, we have $\sum_{i=u+1}^{j}\varepsilon_i\ge -m$ as $\sum_{i=s}^{j}\varepsilon_i=\varepsilon_s+\sum_{i=s+1}^{u}\varepsilon_i+\sum_{i=u+1}^{j}\varepsilon_i=-1+0+\sum_{i=u+1}^{j}\varepsilon_i\ge -(m+1)$. Moreover, we have $\sum_{i=u+1}^{j_0}\varepsilon_i=-m$.
	 
	 Therefore, we conclude that $<\varepsilon_{u+1},\dots,\varepsilon_{v-1}>$ is a segment in $\mathcal{L}_\mathcal{O}(m)$. Hence we get a map $\mathcal{L}_\mathcal{O}(m+1)\rightarrow \mathcal{L}_\mathcal{O}(m)$. This map is injective because the segments in $\mathcal{L}_\mathcal{O}(m+1)$ are disjoint. Thus we have $a_{m+1}(\mathcal{O})=|\mathcal{L}_\mathcal{O}(m+1)|\le |\mathcal{L}_\mathcal{O}(m)|=a_m(\mathcal{O}) $. Therefore, \[\gamma_{H_\pi}(m+1)-\gamma_{H_\pi}(m)=\sum_{\mathcal{O}}a_{m+1}(\mathcal{O})\le \sum_{\mathcal{O}}a_{m}(\mathcal{O}) = \gamma_{H_\pi}(m)-\gamma_{H_\pi}(m-1).\]
	
\end{proof}	

\begin{rmk}
	Corollary \ref{cor} is a particular case of Theorem \ref{gamma} (a) and (b). But for the particular case of $H_\pi$'s, Corollary \ref{cor} is a refinement of Theorem \ref{gamma} (a) and (b) and its proof.
\end{rmk}

	Since $\End(H)\cong\End(H^\vee)$, we have $\gamma_H(m)=\gamma_{H^\vee}(m)$. Therefore in the following examples we can assume $d\leq c$. For simplicity we write $\gamma(m)=\gamma_{H_\pi}(m)$ and $\beta(m)=\beta_{H_\pi}(m)$ for the number of connected components of $\Ends(H_\pi[p^m])$.

\begin{exam}
	Let $c=d=2$ and $\pi=(1234)$. There are $4$ orbits:
	\begin{itemize}
	\item $\mathcal{O}_1 = [(1,1)]=((1,1),(2,2),(3,3),(4,4)),\quad \varepsilon_{\mathcal{O}_1}=(0,0,0,0)$.
	
	This is a circular orbit of level $0$. It has no free linear segments.
	
	\item $\mathcal{O}_2 = [(1,2)]=((1,2),(2,3),(3,4),(4,1)),\quad \varepsilon_{\mathcal{O}_2}=(0,1,0,-1)$.
	
	This is a circular orbit of level $1$. It has one free linear segment $<-1,0,1>$ of level $1$.

	\item $\mathcal{O}_3 = [(1,3)]=((1,3),(2,4),(3,1),(4,2)),\quad \varepsilon_{\mathcal{O}_3}=(1,1,-1,-1,)$;
	
	This is a circular orbit of level $2$. It has one free linear segment $<-1,1>$ of level $1$ and one free linear segment $<-1,-1,1,1>$ of level $2$.
	
	\item $\mathcal{O}_4 = [(1,4)]=((1,4),(2,1),(3,2),(4,3)),\quad \varepsilon_{\mathcal{O}_4}=(1,0,-1,0)$.
	
	This is a circular orbit of level $1$. It has one free linear segment $<-1,0,1>$ of level $1$.
	\end{itemize}
	
	Therefore we have	
	\begin{align*}
		\gamma(1) &=\sum_{\mathcal{O}}a_1(\mathcal{O})=0+1+1+1=3,\\
		\gamma(m) &=\sum_{\mathcal{O}}(a_1(\mathcal{O})+a_2(\mathcal{O}))=0+1+2+1=4,\quad \forall m\ge 2.\\
		\beta(1) &=p^4,\\
		\beta(m) &=p^{(m+2(m-1)+ (m-2))4}=p^{16(m-1)}, \quad \forall m\ge 2.
		\end{align*}
	Since $\gamma(m)$ stops at $m=2$, we get that $n_{H_\pi}=2$ by Theorem \ref{gamma}.
		
\end{exam}

\begin{exam}
	Let $d=1$ and $c\geq 1$. Let $h=c+d$ and $\pi=(12\cdots h)$. There are $h$ orbits of size $h$:
	
		\begin{align*}
			\mathcal{O}_1 &= [(1,1)], \quad \varepsilon_{\mathcal{O}_1}=(0,0,\dots,0,0);\\
			\mathcal{O}_2 &= [(1,2)], \quad \varepsilon_{\mathcal{O}_2}=(1,0,\dots,0,-1);\\
			\mathcal{O}_3 &= [(1,3)],\quad \varepsilon_{\mathcal{O}_3}=(1,0,
			\dots,-1,0);\\
			&\vdots\\
			\mathcal{O}_h &= [(1,h)], \quad \varepsilon_{\mathcal{O}_h}=(1,-1,\dots,0,0).
		\end{align*} 

	All orbits except $\mathcal{O}_1$ are circular of level $1$ with one free linear segment. Simple calculation shows that we have
	\[\gamma(m) =h-1=c,\quad \forall m\ge 1\]
	and
	\[\beta(m) =p^{mh^2-h^2+h}=p^{m(c+1)^2-c(c+1)}, \quad \forall m\ge 1.\]
	We also get that $n_{H_\pi}=1$.
\end{exam}

\begin{exam}\label{h}
	Let $d\leq c$ be arbitrary positive integers. Let $h=c+d$ and $\pi=(12\cdots h)$. As before there are $h$ orbits of size $h$. Choose the order of the orbits and the starting element in each orbit as follows:
	
	
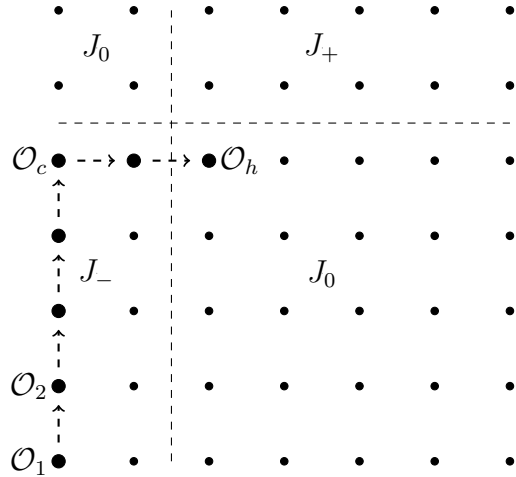
\begin{center}
		\begin{tikzpicture}[scale=0.5]
		\begin{scope}
		
		\foreach \x in {-3,-2,...,3}{                           
			\foreach \y in {-3,-2,...,3}{                       
				\node[draw,circle,inner sep=1pt,fill] at (2*\x,2*\y) {}; 
			}
		}
		
		\draw[dashed] (-6,3) -- (6,3);
		\draw[dashed] (-3,-6) -- (-3,6);
		
		\filldraw[black] (-6,-6) circle (5pt) node[anchor=east] {$\mathcal{O}_1$};
		\filldraw[black] (-6,-4) circle (5pt) node[anchor=east] {$\mathcal{O}_2$};
		\filldraw[black] (-6,-2) circle (5pt);
		\filldraw[black] (-6,0) circle (5pt);
		\filldraw[black] (-6,2) circle (5pt) node[anchor=east] {$\mathcal{O}_c$};
		\filldraw[black] (-4,2) circle (5pt);
		\filldraw[black] (-2,2) circle (5pt) node[anchor=west] {$\mathcal{O}_h$};
		
		\filldraw[black] (1,-1) circle (0pt) node {$J_0$};
		\filldraw[black] (-5,5) circle (0pt) node {$J_0$};
		\filldraw[black] (1,5) circle (0pt) node {$J_+$};
		\filldraw[black] (-5,-1) circle (0pt) node {$J_-$};
		
		\draw[dashed, thick,->] (-6,-5.5) -- (-6,-4.5);
		\draw[dashed, thick,->] (-6,-3.5) -- (-6,-2.5);
		\draw[dashed, thick,->] (-6,-1.5) -- (-6,-0.5);
		\draw[dashed, thick,->] (-6,0.5) -- (-6,1.5);
		\draw[dashed, thick,->] (-5.5,2) -- (-4.5,2);
		\draw[dashed, thick,->] (-3.5,2) -- (-2.5,2);
		
		\end{scope}
		\end{tikzpicture}
	\end{center}
	
	Then we have the following:

\[\mathcal{O}_1 = [(h,1)], \quad \varepsilon_{\mathcal{O}_1}=(-1,\underbrace{0,\dots,0}_{d-1},1,\underbrace{0,\dots,0}_{c-1});\]	

\[\mathcal{O}_2 = [(h-1,1)], \quad \varepsilon_{\mathcal{O}_2}=(-1,-1,\underbrace{0,\dots,0}_{d-2},1,1,\underbrace{0,\dots,0}_{c-2});\]	
\[\vdots\]
\[\mathcal{O}_{d-1} = [(c+2,1)], \quad \varepsilon_{\mathcal{O}_{d-1}}=(\underbrace{-1,\dots,-1}_{d-1},0,\underbrace{1,\dots,1}_{d-1},\underbrace{0,\dots,0}_{c-d+1});\]

\[\mathcal{O}_{d} = [(c+1,1)], \quad \varepsilon_{\mathcal{O}_d}=(\underbrace{-1,\dots,-1}_{d},\underbrace{1,\dots,1}_{d},\underbrace{0,\dots,0}_{c-d});\]

\[\mathcal{O}_{d+1} = [(c,1)], \quad \varepsilon_{\mathcal{O}_{d+1}}=(\underbrace{-1,\dots,-1}_{d},0,\underbrace{1,\dots,1}_{d},\underbrace{0,\dots,0}_{c-d-1});\]
\[\vdots\]
\[\mathcal{O}_{c} = [(d+1,1)], \quad \varepsilon_{\mathcal{O}_c}=(\underbrace{-1,\dots,-1}_{d},\underbrace{0,\dots,0}_{c-d},\underbrace{1,\dots,1}_{d});\]

\[\mathcal{O}_{c+1} = [(d+1,2)], \quad \varepsilon_{\mathcal{O}_{c+1}}=(\underbrace{-1,\dots,-1}_{d-1},\underbrace{0,\dots,0}_{c-d+1},\underbrace{1,\dots,1}_{d-1},0);\]

\[\mathcal{O}_{c+2} = [(d+1,3)], \quad \varepsilon_{\mathcal{O}_{c+2}}=(\underbrace{-1,\dots,-1}_{d-2},\underbrace{0,\dots,0}_{c-d+2},\underbrace{1,\dots,1}_{d-2},0,0);\]
\[\vdots\]
\[\mathcal{O}_{h-1} = [(d+1,d)], \quad \varepsilon_{\mathcal{O}_{h-1}}=(-1,\underbrace{0,\dots,0}_{c-1},1,\underbrace{0,\dots,0}_{d-1});\]
\[\mathcal{O}_{h} = [(d+1,d+1)], \quad \varepsilon_{\mathcal{O}_h}=(0,\dots,0).\]

Ignoring the zeros, we have 

\[\varepsilon_{\mathcal{O}_i}=\varepsilon_{\mathcal{O}_{h-i}}=(\underbrace{-1,\dots,-1}_{i},\underbrace{1,\dots,1}_{i}),\quad 1\le i < d,\]
and 
\[\varepsilon_{\mathcal{O}_i}=(\underbrace{-1,\dots,-1}_{d},\underbrace{1,\dots,1}_{d}),\quad d\le i \le c.\]
				 
From this we have $a_n(\mathcal{O}_h)=0$, $\forall n >0$;
\[a_n(\mathcal{O}_i)=\left\{\begin{array}{ll}
1,\quad & n\le i,\\
0 & n > i,
\end{array}\right. \quad \textrm{when } 1\le i<d \textrm{ or } c<i\le h-1;\]
and 
\[a_n(\mathcal{O}_i)=\left\{\begin{array}{ll}
1,\quad & n\le d,\\
0 & n > d,
\end{array}\right. \quad \textrm{when } d\le i \le c.\]
Therefore we have
	\begin{align*}
		\gamma(1) &=\sum_{i=1}^{h}a_1(\mathcal{O}_i)=(h-1),\\
		\gamma(2) &=\gamma(1)+\sum_{i=1}^{h}a_2(\mathcal{O}_i)= (h-1)+(h-3),\\
		\gamma(3) &=\gamma(2)+\sum_{i=1}^{h}a_3(\mathcal{O}_i)= (h-1)+(h-3)+(h-5),\\
		&\vdots\\
		\gamma(d) &=\gamma(d-1)+\sum_{i=1}^{h}a_d(\mathcal{O}_i)= (h-1)+(h-3)+\cdots+(h-2d+1)\\
		\gamma(m) &=\gamma(d), \quad \forall m > d.
	\end{align*}
	
	Equivalently,
		 
	\[\gamma(m)=\left\{\begin{array}{ll}
		\overset{m}{\underset{i=1}{\sum}}(h-2i+1)=m(h-m),\quad & 1\le m \le d;\\
		cd, & m > d.
	\end{array}\right.\]

	For the number of connected components, note that there is $1$ circular orbit of level $0$, namely $\mathcal{O}_h$. There are $2$ circular orbits of level $i$ for each $1\le i\le d-1$, namely, $\mathcal{O}_i$ and $\mathcal{O}_{h-i}$. There are $c-d+1$ circular orbits of level $d$, namely, $\mathcal{O}_d,\mathcal{O}_{d+1},\dots,\mathcal{O}_c$. Some calculation shows that
	
		\[\beta(m)=\left\{\begin{array}{ll}
		p^{mh+2(m-1)h+2(m-2)h+\dots+2h},\quad & 1\le m \le d;\\
		p^{mh+2(m-1)h+2(m-2)h+\dots+2(m-d+1)h+(c-d+1)(m-d)h}, & m > d.
		\end{array}\right.\]
Equivalently,
\[\beta(m)=\left\{\begin{array}{ll}
p^{m^2h},\quad & 1\le m \le d;\\
p^{mh^2-cdh}, & m > d.
\end{array}\right.\]
From the formula for $\gamma(m)$, we get that $n_{H_\pi}=d=\min\{c,d\}$.
\end{exam}

\begin{rmk}
	By Theorem \ref{gamma} or Corollary \ref{cor} we have $\gamma(2)-\gamma(1)\le \gamma(1)$ and thus $\gamma(2)\le 2\gamma(1)$. Also, we have $\gamma(3)-\gamma(2)\le \gamma(2)-\gamma(1)$ and therefore
	\[\gamma(3)\le 2\gamma(2)-\gamma(1)\le2\gamma(2)-\frac{1}{2}\gamma(2)=\frac{3}{2}\gamma(2).\]
	By induction we have 
	\[\gamma(m)\le \frac{m}{m-1}\gamma(m-1), \quad \forall m >1,\]
	and more generally
	\[\gamma(m)\le \frac{m}{n}\gamma(n), \quad \forall m > n \ge 1.\]
	
	The formula
	\[\gamma(m)=\left\{\begin{array}{ll}m(h-m),\quad & 1\le m \le d,\\
	cd, & m > d,
	\end{array}\right.\]
	in Example \ref{h} show that we have 
	\[\lim_{h\rightarrow \infty}\frac{\gamma(m)}{\gamma(n)}=\frac{m}{n}, \quad \forall 1\le m,n \le d.\]
	It is not clear if this equality can be obtained at finite heights.
	
\end{rmk}	

\fontsize{11}{11pt} \selectfont 
\cleardoublepage
\addcontentsline{toc}{chapter}{Bibliography}

\bibliographystyle{amsalpha}
\bibliography{references}

\end{document}